\documentclass[fleqn]{article}

\usepackage[english]{babel}
\usepackage[utf8]{inputenc}
\usepackage{amsmath}
\usepackage{graphicx}

\usepackage{url}

\usepackage{amsthm}

\usepackage{thmtools, thm-restate}

\usepackage{hyperref}
\usepackage{amssymb}
\usepackage{amsthm}

\usepackage{tikz}
\usetikzlibrary{arrows}

\usepackage{array}
\usepackage{epigraph}

\newtheorem{theorem}{Theorem}
\newtheorem{proposition}{Proposition}
\newtheorem{definition}{Definition}

\newtheorem{lemma}{Lemma}

\theoremstyle{remark}
\newtheorem*{remark}{Remark}

\theoremstyle{remark}
\newtheorem*{note}{Note}

\title{Higher Equipments, Double Colimits and Homotopy Colimits}
\date{2019\\ August}
\author{Redi Haderi\\ Mathematics Department, Bilkent University
	\\  redi.haderi@bilkent.edu.tr}

\begin{document}
	
	\maketitle

	\begin{abstract}
		This document is centered around a main idea: simplicial categories, by which we mean simplicial objects in the category of categories, can be treated as a two-fold categorical structure and their double category theory is homotopically meaningful. The most well-known two-fold structures are double categories, typically used to organize bimodules in various contexts. However there is no double category of spaces even though notions of bimodule are conceivable. We first remedy this defect of double category theory by constructing a meaningful simplicial category of spaces. Then we develop the analogy with double categories by defining double colimits and by postulating an equipment property, which is promptly satisfied in the examples. As an application we prove that certain double colimits are naturally interpreted as homotopy colimits. Quite surprisingly this analogy unveils a principle: simplicial categories are to simplicially enriched categories what double categories are to 2-categories! 
	\end{abstract}
\pagebreak

\tableofcontents

\section*{Acknowledgements}

I would like to thank all the members of the Mathematics Department at Bilkent University. In paricular Laurence Barker, Ergün Yalçın  and Matthew Gelvin for their precious advice, my friend Melih Üçer for our good conversations, and my adviser, Özgün Ünlü , for encouraging me to develop these ideas and complete this work.

Immense gratitude goes to my family for always believing in me and my wife for being a wonderful companion and office mate. This work is dedicated to them. Especially my mother, my first mathematics teacher.

\pagebreak

\section{Introduction and summary}

\epigraph{Atiyah described mathematics as the “science of analogy.” In this vein, the purview
	of category theory is \textit{mathematical analogy}.}{Emily Riehl \\ "Category Theory in Context" \\ \cite{riehl2017category}}

\subsection*{Double categories meet sSet-categories}

There are many angles from which to introduce and motivate the content of this paper. We begin by describing an observation which triggered us to develop the formalism that follows in the next sections. 

We have the following two realms of mathematical structures:
\begin{itemize}
	\renewcommand\labelitemi{--}
	\item That of double categories and their category theory, usually associated with bimodule-like structures. In this world we find terms such as proarrow, equipment, double limits and colimits etc.
	\item That of categories enriched over simplicial sets, \textbf{sSet}-categories for short, usually associated with homotopical settings and higher category theory. In this world we find terms such as homotopy limit and colimit.
\end{itemize} 

Given a diagram of categories 
$$F : \mathcal{J} \rightarrow \textbf{Cat}$$
we may construct a category $\textbf{Gro}(F)$, called the \textit{Grothendieck construction} of $F$. If we denote $\mathcal{C}_i = F(i)$ for $i \in \mathcal{J}$ and $F_\alpha: \mathcal{C}_i \rightarrow \mathcal{C}_j$ the image of a morphism $\alpha: i \rightarrow j$ in $\mathcal{J}$, then $\textbf{Gro}(F)$ has
\begin{itemize}
	\item[($\bullet$)] objects the objects of $\mathcal{C}_i$ for $i \in \mathcal{J}$
	\item[($\rightarrow$)] morphisms between two objects $x \in \mathcal{C}_i$ and $y \in \mathcal{C}_j$ pairs $(\alpha, f)$ where $\alpha: i \rightarrow j$ is a morphism in $\mathcal{J}$ and $f : F_\alpha x \rightarrow y$ is a morphism in $\mathcal{C}(j)$
	\item[($\circ$)]  composition given by the formula
	$$(\beta, g)(\alpha, f) = (\beta\alpha, g\circ F_\beta(f))$$
	whenever it makes sense
\end{itemize}

The Grothendieck construction serves as a meeting point for double categories and \textbf{sSet}-categories because it can be formally understood both as 
\begin{itemize}
	\item[i)] a homotopy colimit
	\item[ii)] a double colimit
\end{itemize}

Roughly speaking, a simplicially enriched category is a category $\mathcal{C}$ in which the morhisms in $\mathcal{C}(x,y)$, for two objects $x,y$, are the vertices of a simplicial set. For example the category of categories \textbf{Cat} can be seen as \textbf{sSet}-enriched, with the mapping space between two categories $\mathcal{C}$ and $\mathcal{D}$ having $n$-simplicies the collection of functors
$$\mathcal{C} \times \Delta^n \rightarrow \mathcal{D}$$
where $\Delta^n$ is the linear category with objects $\{0, 1, \dots ,n  \}$ and a morphism $i \rightarrow j$ if $i<j$.

$\mathcal{C}$ is said to be \textit{tensored} over simplicial sets if it makes sense to "multiply" an object $x$ with a simplicial set $K$. More precisely if there is an object $K \odot x$ and a natural isomorphism 
$$\mathcal{C}(K \odot x, y) \cong \textbf{sSet}(K, \mathcal{C}(x,y))$$

In a tensored \textbf{sSet}-category $\mathcal{C}$ we may define the (local) \textit{homotopy colimit} of a diagram 
$$F: \mathcal{J} \rightarrow \mathcal{C}$$
via the coend formula
$$\text{hocolim}F = \int^{i \in \mathcal{J}} i/\mathcal{J} \odot F(i)$$
We consider coends confusing and this formula ad hoc, but it works. 
The Grothendieck construction is a model for the homotopy colimit of a diagram in \textbf{Cat}.

On the other hand the category \textbf{Cat} is part of a double category. In general, a double category $\mathbb{D}$ consists of
\begin{itemize}
	\item[($\bullet$)] objects
	\item[($\rightarrow$)] two types of morphisms between the objects, usually called vertical and horizontal, which compose within their type but not with each other
	\item[($\square$)] square-shaped cells
	\begin{center}
		\begin{tikzpicture}
		
		\node (v1) at (-2.8,3.2) {$A$};
		\node (v3) at (-1.2,3.2) {$B$};
		\node (v2) at (-2.8,1.8) {$C$};
		\node (v4) at (-1.2,1.8) {$D$};
		\draw[->]  (v1) edge node[left, font= \scriptsize]{$f$} (v2);
		\draw[->]  (v1) edge node[above, font= \scriptsize]{$M$} (v3);
		\draw[->]  (v3) edge node[right, font= \scriptsize]{$g$} (v4);
		\draw[->]  (v2) edge node[below, font= \scriptsize]{$N$} (v4);
		
		\node at (-2,2.5) {$\alpha$};
		\end{tikzpicture}
	\end{center}
	whose boundary is comprised of vertical and horizontal morphisms
	\item[($\circ$)] vertical and horizontal composition of cells (when they are adjecent to each other in the obvious manner) satisfying certain laws
\end{itemize}

 It is typical to use such cells in settings where $A,B,C,D$ are some sort of algebraic objects (monoids, groups, rings etc), $f,g$ homomorphisms, $M,N$ are bimodules and $\alpha : M \rightarrow N$ is a morphism which "respects" $f$ and $g$. We can define such double categories because bimodules can be tensored and we can consider tensoring as a composition operation. 

In particular, when regarding categories themselves as algebraic objects, for two categories $\mathcal{C}$, $\mathcal{D}$ we define a $(\mathcal{C}, \mathcal{D})$-bimodule, also called profunctor, to be a functor 
$$u : \mathcal{C}^{op} \times \mathcal{D} \rightarrow \text{\textbf{Set}}$$
where \textbf{Set} is the category of sets and functions. The double category we obtain, which we denote \textbf{Prof}, has important extra properties and is an example of an \textit{equipment}. 

The equipment structure formally captures the fact that for a functor $F: \mathcal{C} \rightarrow \mathcal{D}$ there is an associated profunctor $F^*$, called the \textit{companion} of $F$, with $$F^*(c,d) = \mathcal{D}(Fc,d)$$ for a pair of objects $c \in \mathcal{C}$, $d \in \mathcal{D}$. 
This construction defines a pseudofunctor
$$(\cdot)^* : \textbf{Cat} \rightarrow \textbf{Prof}_h$$
into the horizontal 2-category of the double category \textbf{Prof}. 

In general, every double category $\mathbb{D}$ contains a horizontal 2-category $\mathbb{D}_h$ with 2-morphisms cells of the form
\begin{center}
	\begin{tikzpicture}
	
	\node (v1) at (-2.8,3.2) {$A$};
	\node (v3) at (-1.2,3.2) {$B$};
	\node (v2) at (-2.8,1.8) {$A$};
	\node (v4) at (-1.2,1.8) {$B$};
	\draw[->]  (v1) edge node[left, font= \scriptsize]{$1_A$} (v2);
	\draw[->]  (v1) edge node[above, font= \scriptsize]{$u$} (v3);
	\draw[->]  (v3) edge node[right, font= \scriptsize]{$1_B$} (v4);
	\draw[->]  (v2) edge node[below, font= \scriptsize]{$v$} (v4);
	
	\node at (-2,2.5) {$\alpha$};
	\end{tikzpicture}
\end{center}
,and similarly a vertical one $\mathbb{D}_v$.

We can define the double colimit of a "horizontal" diagram 
$$F: \mathcal{J} \rightarrow \mathbb{D}_h$$
as follows:
for two horizontal diagrams $F$ and $G$ we first define "vertical" transformations between them in the obvious way. For example a cell
\begin{center}
	\begin{tikzpicture}
	
	\node (v1) at (-2.8,3.2) {$A$};
	\node (v3) at (-1.2,3.2) {$B$};
	\node (v2) at (-2.8,1.8) {$C$};
	\node (v4) at (-1.2,1.8) {$D$};
	\draw[->]  (v1) edge node[left, font= \scriptsize]{$f$} (v2);
	\draw[->]  (v1) edge node[above, font= \scriptsize]{$u$} (v3);
	\draw[->]  (v3) edge node[right, font= \scriptsize]{$g$} (v4);
	\draw[->]  (v2) edge node[below, font= \scriptsize]{$v$} (v4);
	
	\node at (-2,2.5) {$\alpha$};
	\end{tikzpicture}
\end{center}
is a vertical transformation between $u$ and $v$ when we regard them as diagrams $\Delta^1 \rightarrow \mathbb{D}_h$. 
Then let the \textit{double colimit} of $F$ be an object $\text{dcolim}F \in \mathbb{D}$ equipped with a vertical transformation to the contant diagram $$F \Rightarrow \text{dcolim}F$$ such that given any other object $C$ with a transormation $F \Rightarrow C$ there is a unique vertical arrow $\text{dcolim}F \rightarrow C$ making the triangle commute
\begin{center}
	\begin{tikzpicture}
	
	\node (v1) at (-5,5) {$F$};
	\node (v2) at (-5,3.5) {$\text{dcolim}F$};
	\node (v3) at (-3,3.5) {$C$};
	\draw[->]  (v1) edge[double] (v2);
	\draw[->]  (v1) edge[double] (v3);
	\draw[->, dashed]  (v2) edge node[below, font = \scriptsize]{$\exists !$} (v3);
	\end{tikzpicture}
\end{center}

We can prove the following result:
\begin{restatable*}{theorem}{grothdcolim}
	\label{thm1}
	
	Let $\mathcal{J}$ be a small category and $F : \mathcal{J} \rightarrow \text{\textbf{Cat}}$ be a diagram of categories. Then $$\textbf{Gro}(\mathcal{C}) \cong \text{dcolim} F^*$$ where the right hand side is the double colimit in \textbf{Prof} of the diagram  $F^*$ given by the composition $$ \mathcal{J} \xrightarrow{F} \text{\textbf{Cat}} \xrightarrow{(\cdot)^*} \text{\textbf{Prof}}_h$$ of $F$ with the companion functor $(\cdot)^*$.
	
\end{restatable*}

We will show that this is not a coincidence at all but double categories and \textbf{sSet}-categories are part of a deeper structural web.

\subsection*{A 2-fold structure for spaces}

It is generally desirable to have a double category of spaces. Let us focus on simplicial sets for convenience. The first question that arises is
\begin{quote}
	Q: What should a proarrow (bimodule) between simplicial sets be?
\end{quote}

To answer this question it is not so useful to think of profunctors algebraically as functors $u:\mathcal{C}^{op}\times \mathcal{D} \rightarrow \text{\textbf{Set}}$, but rather in terms of \textit{collages}. A collage between $\mathcal{C}$ and $\mathcal{D}$ is a category obtained by adding new morphisms between the objects of $\mathcal{C}$ and those of $\mathcal{D}$ (but not in the other direction), which looks something like this 
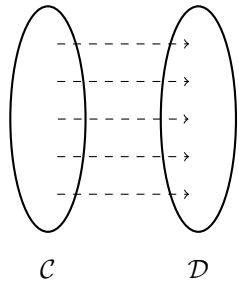
\begin{center}
	\begin{tikzpicture}
	
	\draw[thick]  (-4,3.5) node (v5) {} ellipse (0.5 and 1.5);
	\draw[thick]  (-2,3.5) node (v6) {} ellipse (0.5 and 1.5);
	\node (v1) at (-4,4.5) {};
	\node (v2) at (-2,4.5) {};
	\node (v3) at (-4,4) {};
	\node (v4) at (-2,4) {};
	\node (v7) at (-4,3) {};
	\node (v8) at (-2,3) {};
	\node (v9) at (-4,2.5) {};
	\node (v10) at (-2,2.5) {};
	\draw[->, dashed]  (v1) edge (v2);
	\draw[->, dashed]  (v3) edge (v4);
	\draw[->, dashed]  (v5) edge (v6);
	\draw[->, dashed]  (v7) edge (v8);
	\draw[->, dashed]  (v9) edge (v10);
	\node at (-4,1.5) {$\mathcal{C}$};
	\node at (-2,1.5) {$\mathcal{D}$};
	\end{tikzpicture}
\end{center}

Given a profunctor $u$ we form a collage $\text{col}(u)$ with morhisms between $c\in \mathcal{C}$ and $d\in \mathcal{D}$ the elements of the set $u(c,d)$. It is not difficult to see that we have a one-to-one correspondence between profunctors and collages.

Even more elegantly, we observe that a collage between $\mathcal{C}$ and $\mathcal{D}$ is simply a functor
$$p: U \rightarrow \Delta^1$$
with $p^{-1}(0) = \mathcal{C}$ and $p^{-1}(0) = \mathcal{D}$, and $U$ the collage of the corresponding profunctor. 

Then we can simply define a proarrow between simplicial sets $X$ and $Y$ to be a morphism
$$p : U \rightarrow \Delta^1$$
with $p^{-1}(0) = X$ and $p^{-1}(1) = Y$.
This is precisely what is called $0$-cylinder in \cite{joyal2008theory} and correspondence (in the sense of quasi-categories) in \cite{lurie2009higher}. The virtue of this notion of proarrow is that every profunctor constitutes an example if we think of categories as simplicial sets. The next question that arises is:
\begin{quote}
	Q: Can we organize arrows and proarrows defined as above in a double category?
\end{quote}
Unfortunately the answer is negative, the reason being that there does not seem to be a reasonable way to compose proarrows.

However we may expand the scope of the last question and ask for a \textit{2-fold structure} rather than a double category. The jargon term "2-fold" is used to refer to structures consisting of:
\begin{itemize}
	\renewcommand\labelitemi{--}
	\item objects
	\item vertical and horizontal arrows
	\item cells (whose shape varies from structure to structure) which relate the vertical and horizontal arrows
	\item various compositions of arrows and cells
\end{itemize}
There is a plethora of these structures found in the literature: double categories, generalized multicategories (see \cite{cruttwell2010unified} or \cite{leinster2004higher} where they are called \textbf{fc}-multicategories), hypervirtual double categories (see \cite{koudenburg2015double}), and in some sense bisimplicial sets (even though we have no compositions).

In this case we have a positive answer to our question.
Our main point is that simplicial categories, that is functors $$\mathbb{E} : \Delta^{op} \rightarrow \text{\textbf{Cat}}$$
where $\Delta$ is the usual category of finite ordinals and order preserving maps and \textbf{Cat} is the category of categories, provide a natural candidate for such a structure. 

First, simplicial categories, being bisimplicial sets, can be considered as two-fold structures with
\begin{itemize}
	\item[($\bullet$)] objects those of $\mathbb{E}_0$
	\item[($\downarrow$)] vertical morphisms those of $\mathbb{E}_0$
	\item[($\triangle$)] horizontal simplicies which are the objects of $\mathbb{E}_n$ for various $n$
	\item[($\square$)] cells the bisimplicies of $\mathbb{E}$, in this case the morphisms of $\mathbb{E}_n$ for various $n$
\end{itemize}
The only novelty in this structure is that we do not have morphisms in the horizontal direction but rather simplices. 

We will construct a simplicial category $\textbf{sSet}^\sharp$ with category of $n$-simplicies the slice category over $\Delta$
$$\textbf{sSet}^\sharp_n = \textbf{sSet}/\Delta^n$$
In particular $\textbf{sSet}^\sharp_0 = \textbf{sSet}$ and $\textbf{sSet}^\sharp_1$ is the category of collages. This simplicial category seems to be the natural way to organize what are referred to as "higher correspondences" in \cite{lurie2009higher}. However we are not concerned with quasi-categorical considerations in this work.

Moreover, there is enough data in a simplicial category $\mathbb{E}$ to do double category theory. If $\mathcal{J}$ is an indexing category we may regard it as a simplicial category and this way a "horizontal diagram" is simply a natural transformation
$$F: \mathcal{J} \Rightarrow \mathbb{E}$$
Then it is straightforward to define its double colimit $\text{dcolim}F$ in  analogy with double category theory.

We will also propose an equipment property. The equipment property we postulate guarantees the existence of a companion horizontal $n$-simplex  $\sigma^* \in E_n$ associated to a vertical one
$$\sigma : \xrightarrow{f_1} \xrightarrow{f_2} \dots \xrightarrow{f_n}$$
 i.e. to a chain of $n$ composable morphisms of $E_0$ (and hence the term \textit{higher equipment}). We obtain a transformation
 $$(\cdot)^* : \mathbb{E}_0 \rightarrow \mathbb{E}$$
 
While in the double category \textbf{Prof} the companion of a functor represents its mapping cylinder, in the simplicial category $\textbf{sSet}^\sharp$ the companion $\sigma^*$ of a chain of morphisms as above in \textbf{sSet} represents the homotopy colimit of $\sigma$. We also capture the fact that this homotopy colimit, which we could call higher mapping cylinder, is naturally equipped with a map to $\Delta^n$.

Every 2-fold object has a vertical and/or horizontal structure incorporated in it. Double categories have vertical and horizontal 2-categories, generalized multicategories with one object have a horizontal multicategory (coloured operad), bisimplicial sets have horizontal and vertical simplicial sets. 

In this vein, a simplicial category $\mathbb{E}$ has a vertical simplicially enriched category $\mathbb{E}_v$ incorporated in its structure. This is perhaps the critical observation of this work, not in sight when the author initiated this study. The analogy may be written as
$$\frac{\text{double categories}}{\text{2-categories}} = \frac{\text{simplicial categories}}{\text{\textbf{sSet}-categories}}$$
Moreover $\textbf{sSet}^\sharp_v$ recaptures the usual simplicial enrichment of \textbf{sSet}.

We prove the following theorem, which characterizes homotopy colimits in $\mathbb{E}_v$ as double colimits in $\mathbb{E}$ for higher equipments $\mathbb{E}$. We interpret it as a higher dimensional version of Theorem \ref{thm1}.

\begin{restatable*}{theorem}{hodcolim}
	\label{thm2}
	
	Let $\mathbb{E}$ be a higher equipment, $\mathcal{J}$ be an indexing category and 
$$F : \mathcal{J} \rightarrow \mathbb{E}_0$$
be a functor. Then we have 
$$\text{dcolim}{F^*} \cong \text{hocolim} F$$
where $F^*$ is the composite
$$\mathcal{J} \xrightarrow{F} \mathbb{E}_0 \xrightarrow{(\cdot)^*} \mathbb{E}$$
of $F$ with the companion map, and the homotopy colimit on the right is taken to be in $\mathbb{E}_v$.
	
\end{restatable*}

Of course everything we mentioned above has a dual. We can define double limits, a dual equipment property and obtain homotopy limits in $\mathbb{E}_v$ as double limits.

\subsection*{Organization}

We have divided this document in two parts. The first part is mainly expository. We start with a blitz introduction to simplicial sets in \ref{sSet} and homotopy colimits of spaces in \ref{hocolimsp} , then we describe the Grothendieck construction for categories in \ref{gro},  profunctors and collages in \ref{prof}, and then focus on double category theory in \ref{dblcat}. We prove that the Grothendieck construction is a double colimit, even though we hope to have reduced the statement to a mere observation at that point. 

In the second part we start by discussing simplicial categories  and the sort of examples we have in mind in detail in \ref{scat}. We postulate the equipment property in \ref{equip} and define companions of higher simplicies in \ref{companion}. Then we define double colimits in simplicial categories in \ref{dcolim}, which is rather straightforward in analogy with double categories. In \ref{ssetcat} we recall some elements of the theory of \textbf{sSet}-categories: tensoring, cotensoring and homotopy colimits. We then observe that there is a vertical \textbf{sSet}-category of a given simplicial category and hence establish a strong analogy. Then we prove our main theorem in \ref{thm}, which says that homotopy colimits in the vertical \textbf{sSet}-category are double colimits. We conclude with a brief discussion on duality and homotopy limits in \ref{dual}.

\subsection*{Prerequisites}

It is our objective to make this document as readable as possible. Hence we have accompanied a lot of our work with illustrations both for aesthetic purposes and in order to convey geometric intuition. The author strongly believes geometric intuition is fundamental to category theory.

Besides basic category theory and algebraic topology,
the ideal reader is familiar with 2-category theory but perhaps hasn't encountered double categories much. We refer to \cite{leinster2004higher} or the more concise \cite{leinster1998basic} for 2-categories. It has to be said that perfect knowledge of 2-category theory is not a must to grasp the content of what we present, even though formally it is.

We start with a brief introduction to simplicial sets (where we also establish some notation and drawing conventions), but the ideal reader is also familiar with their basic combinatorics. Our references for this subject are \cite{goerss2009simplicial} and \cite{riehl2011leisurely}. 

To the best of our capabilities anything beyond these is presented in expository style. Of course, the author is by no means an authority on the majority of the subjects treated here and the reader is advised to consult the references for better understanding of the topics at hand. 

\subsection*{Notation conventions}

We will typically denote known categories by boldface letters and name them after  the objects of the category. 
\begin{center}
\begin{tabular}{c | c}
	\text{Notation}	&  \text{Corresponding category}\\ 
	\hline
	\textbf{Set}	& \text{sets and function}  \\ 
	\textbf{Top}	& \text{topological spaces and continuous maps} \\ 
	\textbf{Cat}	& \text{categories and functors} \\ 
	\textbf{sSet}	&  \text{simplicial sets and simplicial maps}\\ 
	\textbf{CW}	& \text{CW-complexes and cellular maps}  \\
\end{tabular} 
\end{center}
Generic categories will be denoted by caligraphic letters $\mathcal{C}$, $\mathcal{D}$ etc and their opposites $\mathcal{C}^{op}$. If $x,y$ are objects in $\mathcal{C}$ we simply write $x,y \in \mathcal{C}$ and we denote the set of morphisms between them $\mathcal{C}(x,y)$. 
	
\section{The build up}

\epigraph{The uniting feature of all these structures is that they are purely algebraic
	in definition, yet near-impossible to understand without drawing or visualizing
	pictures. They are inherently geometrical.}{Tom Leinster \\ "Higher Operads, Higher Categories" \\ \cite{leinster2004higher}}

In this section we go through a series of topics and expose the ideas and observations that lead us to the study of higher equipments. We keep the discussion mainly expository although full discussions of each topics require a lot of space. Hence the focus will be on the elements which are most relevant to our purposes, and conveying as much geometric intuition as possible. 

\subsection{Simplicial sets}
\label{sSet}

Let $\Delta$ be the category with
\begin{itemize}
	\item[($\bullet$)] objects  finite ordinals $\{0, 1, \dots ,n  \}$ which we denote $[n]$
	\item[($\rightarrow$)] morphisms  order preserving maps
\end{itemize}
The category $\Delta$ plays an essential role in homotopy theory: it captures the essential combinatorial features of \textit{simplices}.

First, as shown in the above table, we may depict finite ordinals themselves as (oriented) simplices. We can also think of order preserving maps as assigning vertices to vertices, edges to edges and so on. Degenerate maps are allowed, where a higher face collapses into a smaller dimensional one. 
For example there is a map $[1] \rightarrow [0]$. 

\begin{table}
	\centering
	\begin{tabular}{|c|c|c | c|}
		\hline
		[0] & [1] & [2] & [3] \\
		\hline
		
		\begin{tikzpicture}
		
		\node at (-1.5,1.5) {$\bullet$};
		\node at (-1.5,1.8) {$0$};
		\end{tikzpicture}
		
		&
		\begin{tikzpicture}
		
		\node (v1) at (-1.5,1.5) {$\bullet$};
		\node at (-1.5,1.8) {$0$};
		
		\node (v2) at (0,1.5) {$\bullet$};
		
		\draw[->] (v1) edge (v2) ;
		
		\node at (0,1.8) {$1$};
		\end{tikzpicture} &
		\begin{tikzpicture}
		
		\node (v1) at (-1.5,1.5) {$\bullet$};
		\node at (-1.8,1.4) {$0$};
		
		\node (v3) at (0,1.5) {$\bullet$};
		
		\node at (0.2,1.4) {$2$};
		\node (v2) at (-0.8,2.6) {$\bullet$};
		
		\node at (-0.8,2.9) {$1$};
		
		\draw[->]  (v1) edge (v2);
		\draw[->]  (v2) edge (v3);
		\draw[->]  (v1) edge (v3);
		
		\end{tikzpicture} &
		\begin{tikzpicture}
		
		\node (v1) at (-1.5,1.5) {$\bullet$};
		\node at (-1.8,1.4) {$0$};
		
		\node (v3) at (0,1.5) {$\bullet$};
		
		\node at (0.2,1.4) {$2$};
		\node (v2) at (-0.8,2.6) {$\bullet$};
		
		\node at (-0.8,2.9) {$1$};
		
		\draw[->]  (v1) edge (v2);
		\draw[->,dashed]  (v2) edge (v3);
		\draw[->]  (v1) edge (v3);
		\node at (0.3,2.7) {$3$};
		
		\node (v4) at (0.3,2.4) {$\bullet$};
		\draw[->]  (v1) edge (v4);
		\draw[->]  (v2) edge (v4);
		\draw[->]  (v3) edge (v4);
		
		\end{tikzpicture}
		
		\\
		\hline
	\end{tabular}
\caption{The first four simplices}
\end{table}

A simplicial set is a functor 
$$X: \Delta^{op} \rightarrow \textbf{Set}$$
We briefly describe how to read the geometry of $X$ from its definition as a functor.

$X$ assigns to each ordinal $[n]$ a set which we denote $X_n$. The elements of $X_n$ are said to be the $n$-simplices of $X$ and visualized accordingly. An $n$-simplex has $(n+1)$ vertices, which are $0$-simplices, and $(n+1)$ faces, which are $(n-1)$-simplices. For a simplex $x \in X_n$ we denote its vertices $x_0$, $x_1$, \dots, $x_n$ and its faces $d_0x$, $d_1x$, \dots $d_nx$. The face $d_ix$ is obtained by "deleting" the vertex $x_i$ from $x$. This way for example we draw $x \in X_2$ as
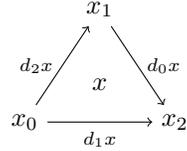
\begin{center}
	\begin{tikzpicture}
	
	\node (v1) at (-3.5,0.5) {$x_0$};
	\node (v3) at (-1.5,0.5) {$x_2$};
	\node (v2) at (-2.5,2) {$x_1$};
	\draw[->]  (v1) edge node[left, font = \scriptsize] {$d_2x$} (v2);
	\draw[->]  (v2) edge node[right, font = \scriptsize] {$d_0x$} (v3);
	\draw[->]  (v1) edge node[below, font = \scriptsize] {$d_1x$} (v3);
	\node at (-2.5,1) {$x$};
	\end{tikzpicture}
\end{center} 

We recover the above as follows: there is a map $v^i: [0] \rightarrow [n]$ in $\Delta$ which assigns $0 \mapsto i$. Since $X$ is a contravariant functor, $v^i$ induces a map 
$$v_i : X_n \rightarrow X_0$$
For $x \in X_n$ we obtain its vertices to be $x_i = v_i(x)$. For faces we consider the map $d^i: [n-1] \rightarrow [n]$ in $\Delta$ given by "skipping" $i$
$$d^i(j) = \begin{cases}
 j & , j<i \\
 j + 1 & , j \geq i
\end{cases}$$
This map induces $$d_i : X_n \rightarrow X_{n-1}$$
and we obtain the faces $d_ix$.

The speciality of simplicial sets consists in \textit{degeneracies}: we regard every simplex to be a degenerate higher simplex as well. Formally, we obtain degeneracies by considering the map $s^i : [n+1] \rightarrow [n]$ in $\Delta$ defined as 
$$s^i(j) = \begin{cases}
j & , j\leq i \\
j - 1 & , j > i
\end{cases}$$
This induces the degeneracy map
$$s_i : X_n \rightarrow X_{n+1}$$

Geometrically we may think of the simplex $s_ix$ as being obtained by "replicating" $x_i$. For a vertex $x \in X_0$ we draw the $1$-simplex $s_0x$ as
\begin{center}
	\begin{tikzpicture}

	\node (v1) at (-3.5,3.5) {$x$};
	\node (v2) at (-1.5,3.5) {$x$};
	\draw[-]  (v1) edge[double] node[above]{$s_0x$} (v2);
	\end{tikzpicture}
\end{center}
with the double line suggesting equality. For a $1$-simplex $x_0 \xrightarrow{x} x_1$ we obtain the degenerate $2$-simplicies $s_0x$ and $s_1x$ by replicating $x_0$ and $x_1$ respectively. We draw them as
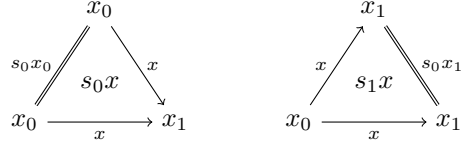
\begin{center}
	\begin{tikzpicture}
	
	\node (v1) at (-3.5,0.5) {$x_0$};
	\node (v3) at (-1.5,0.5) {$x_1$};
	\node (v2) at (-2.5,2) {$x_0$};
	\draw  (v1) edge[double] node[left, font = \scriptsize] {$s_0x_0$} (v2);
	\draw[->]  (v2) edge node[right, font = \scriptsize] {$x$} (v3);
	\draw[->]  (v1) edge node[below, font = \scriptsize] {$x$} (v3);
	\node at (-2.5,1) {$s_0x$};
	\end{tikzpicture} \ \ \ \ \ \ \ \ 
	\begin{tikzpicture}
	
	\node (v1) at (-3.5,0.5) {$x_0$};
	\node (v3) at (-1.5,0.5) {$x_1$};
	\node (v2) at (-2.5,2) {$x_1$};
	\draw[->]  (v1) edge node[left, font = \scriptsize] {$x$} (v2);
	\draw  (v2) edge[double] node[right, font = \scriptsize] {$s_0x_1$} (v3);
	\draw[->]  (v1) edge node[below, font = \scriptsize] {$x$} (v3);
	\node at (-2.5,1) {$s_1x$};
	\end{tikzpicture}
\end{center} 

As geometric intuition dictates we have for example $d_2s_0x = s_0x_0$ and $d_0s_0x = d_1s_0x = x$. Such identities relating faces and degeneracies are called \textit{simplicial identities}. We can imagine a large number of simplicial identities in higher dimensions, but quite amazingly we can deduce all simplicial identities from five elemental ones listed below:
\begin{align*}
	d_id_j &= d_{j-1}d_i, \ \ \ \ \ \ \  i<j \\
	s_is_j &= s_{j+1}s_i, \ \ \ \ \ \ \  i \leq j \\
	d_is_j &=\begin{cases}
	1, & i=j,j+1 \\
	s_{j-1}d_i, & i<j \\
	s_jd{i-1}, & i>j+1
	\end{cases}
	\end{align*}

Specifying a sequence of sets $X_n$, $n= 0, 1, \dots $, and functions $d_i, s_i$ satisfying the above relations is the same as specifying a functor $X$ from $\Delta^{op}$ to \textbf{Set} (see \cite{mac2013categories} for details). So the simplicial identities encapsulate all our geometric intuition and we have a purely combinatorial definition of simplicial sets. 

Simplicial sets form a category, \textbf{sSet}, whose morphisms are natural transformations. Alternatively, we may think of a morphism $f: X \rightarrow Y$ between simplicial sets as given by a sequence of functions $f_n: X_n \rightarrow Y_n$ which respect faces and degeneracies. We insist that both the functorial and the combinatorial perspective are important and choose not to prefer one over the other.

We started by thinking of ordinals themselves a simplices. This is formally justified due to the Yoneda embedding
$$\Delta \hookrightarrow \textbf{sSet}$$
We denote the simplicial set corresponding to $[n]$ by $\Delta^n$, and refer to it as the \textit{standard $n$-simplex}. Its $m$-simplices $\Delta^n_m$ are simply the order preserving maps $[m] \rightarrow [n]$. Also, intuitively a map from $\Delta^n$ to $X$ is just defined by picking an $n$-simplex of $X$. This is formally justified by the Yoneda Lemma which tells us 
$$X_n = \textbf{sSet}(\Delta^n, X)$$

Examples of simplicial sets in nature are too many to count and present here. First, we can just draw one. For instance
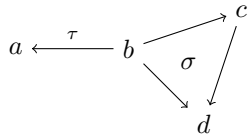
\begin{center}
	\begin{tikzpicture}
	
	\node (v2) at (-4,2) {$a$};
	\node (v1) at (-2.5,2) {$b$};
	\node (v3) at (-1,2.5) {$c$};
	\node (v4) at (-1.5,1) {$d$};
	\draw[->]  (v1) edge node[above, font = \scriptsize] {$\tau$} (v2);
	\draw[->]   (v1) edge (v3);
	\draw[->]   (v3) edge (v4);
	\draw[->]   (v1) edge (v4);
	\node at (-1.7,1.8) {$\sigma$};
	\end{tikzpicture}
\end{center}
This simplicial set has one $2$-simplex, four $1$-simplices and four vertices which are non-degenerate and the rest of its simplices are degenerate (meaning they are of the form $s_ix$ for some simplex $x$). It is common practice to hide the degenerate simplices when drawing pictures. 

The topological $n$-simplex, which we denote $|\Delta^n|$, is defined to be the space
$$|\Delta^n| = \{(x_0, \dots, x_n) \in \mathbb{R}^{n+1} : x_0 + \dots x_n = 1 \} $$
with subspace topology.
For a topological space $T$ we define its \textit{singular simplicial set}, $ST$, to have $n$-simplices 
$$ST_n = \textbf{Top}(|\Delta^n|, X)$$
with the obvious face and degeneracy maps.

The assignment $T \mapsto ST$ extends to define a functor
$$S: \textbf{Top} \rightarrow \textbf{sSet}$$
 We encounter this simplicial set when we define the singular homology of a space.

The singular simplicial set is a good example to illustrate the philosophy of degeneracies as well. The points of the space $T$ are the $0$-simplices of $ST$ and paths in $T$ are the $1$-simplices of $ST$. However we can think of a point as a constant path, which translates as degenerate simplex in our language. 

In the world of categories we may think of the ordinal $[n]$ as a poset and hence as a category. This category has objects $\{0, 1, \dots ,n  \}$ and a morphism $i \rightarrow j$ if $i<j$. We denote it $\Delta^n$ by abuse of notation and refer to it as the \textit{ categorical $n$-simplex}.

For a category $\mathcal{C}$ we define its \textit{nerve} $N\mathcal{C}$ to be the simplicial set with $n$-simplices given as
$$N\mathcal{C}_n = \textbf{Cat}(\Delta^n,\mathcal{C})$$
In other words an $n$-simplex in $\sigma \in N\mathcal{C}_n$ is just a chain of $n$ composable morphisms
$$\sigma : x_0 \xrightarrow{f_1} x_1 \xrightarrow{f_2} \dots \xrightarrow{f_n} x_n$$
in $\mathcal{C}$. The faces $d_i\sigma$ are given by composing
$$d_i\sigma : x_0 \xrightarrow{f_1} \dots x_{i-1} \xrightarrow{f_{i+1}f_i} x_{i+1} \dots \xrightarrow{f_n} x_n$$
for $0<i<n$ and the faces $d_0\sigma$, $d_n\sigma$ by deleting $x_0$ and $x_n$ respectively. Degeneracies are obtained by inserting identity morphisms in the chain
$$s_i\sigma : x_0 \xrightarrow{f_1} \dots x_i \xrightarrow{1_{x_i}} x_i \dots \xrightarrow{f_n} x_n$$
The nerve construction defines a functor, which as a matter of fact is a fully-faithful embedding
$$N : \textbf{Cat} \hookrightarrow \textbf{sSet}$$
For this reason, for the rest of this work we drop the symbol $N$ to indicate the nerve and treat categories as simplicial sets whenever necessary.

Again, the philosophy of degeneracies is reflected in the common practice of category theory: when drawing a commutative diagram we do not draw the infinitely many identity arrows. 

These examples can be seen to be obtained by a Yoneda Extension along a cosimplicial object. Meaning if we have a category $\mathcal{C}$ and a functor 
$$s: \Delta \rightarrow \mathcal{C}$$
we may interpret $s[n]$ to be the $n$-simplex in $\mathcal{C}$ and obtain a functor 
$$S: \mathcal{C} \rightarrow \textbf{sSet}$$
exactly as in the examples. We refer the reader to \cite{riehl2011leisurely} for this perspective.

Another source of examples is the new-from-old type of constructions in \textbf{sSet} itself. Products, coproducts and quotients are defined in the obvious way. Also, given a simplicial set and a subset of simplices we may define the simplicial subset generated by the subset. Important examples are 
\begin{itemize}
	\renewcommand\labelitemi{--}
	\item the \textit{simplicial sphere}  $\partial \Delta^n \subset \Delta^n$,  generated by the faces of $\Delta^n$
	\item the \textit{horn} $\Lambda^n_i \subset \Delta^n$, generated by the faces of $\Delta^n$ except the $i$-th one
\end{itemize}
We trust the reader to make the last paragraph precise (or see \cite{goerss2009simplicial}). 

The geometry of simplicial sets is rigid in comparison to topological spaces or CW-complexes, in the sense that there are fewer morphisms in \textbf{sSet} then in \textbf{Top}. However we have enough data to do homotopy theory. For example we may define a homotopy in \textbf{sSet} to be a morphism 
$$X \times \Delta^1 \rightarrow Y$$
We may define homotopy groups as well and create a category with weak homotopy equivalences whose localization is equivalent to the homotopy category of CW-complexes. So, at least homotopically, simplicial sets are a model for spaces. We refer the reader to \cite{goerss2009simplicial} for these results. 

There is much more to say but hopefully we have conveyed a few points. In particular the "deletion-replication" method of thinking about faces and degeneracies will be relevant in understanding the next section in which we study simplicial categories. 

\subsection{Homotopy colimits of spaces}
\label{hocolimsp}

Let \textbf{Top} be the category of topological spaces and continuous maps and let $\mathcal{J}$ be a small category. Given a diagram of spaces 
$$X : \mathcal{J} \rightarrow \textbf{Top}$$ 
we can construct a space $\text{colim}X$, the colimit of the diagram $X$. Let us denote $X_i = X(i)$ for $i\in \mathcal{J}$ and $f_\alpha : X_i \rightarrow X_j$ the image $X(\alpha)$ of a morphism $\alpha: i \rightarrow j$ in $\mathcal{J}$. We can define the space $\text{colim}X$ explicitly as follows: 

\begin{enumerate}
	\item As a set define $$\text{colim}X = \coprod_{i \in \mathcal{J}} X_i / \sim$$ where $\sim$ is the smallest equivalence relation such that for $x \in X_i$, $y \in X_j$ we have $x \sim y$ if there is a morphism $\alpha: i \rightarrow j$ in $\mathcal{J}$ such that $f_\alpha(x)=y$
	\item Give this set the smallest topology such that the induced maps $X_i \rightarrow \text{colim}X$ are continuous
\end{enumerate}

It is well known that the colimit construction is not well-behaved homotopically. We recall the standard example. Let $\mathcal{J} = (1 \leftarrow 0 \rightarrow 2)$ be the pushout category. Below we depict two diagrams $X$ and $X^\prime$, a transformation $ X \Rightarrow X^\prime$ and the induced map on their colimits: 

\begin{center}
	\begin{tikzpicture}

\node (v1) at (-4.5,3.5) {$X$};
\node (v2) at (-4.5,2) {$X^\prime$};
\node at (-4,3.5) {:};
\node at (-4,2) {:};
\node (v4) at (-3.5,3.5) {$D^{n+1}$};
\node (v3) at (-2,3.5) {$S^n$};
\node (v5) at (-0.5,3.5) {$D^{n+1}$};
\node (v7) at (-3.5,2) {$*$};
\node (v6) at (-2,2) {$S^n$};
\node (v8) at (-0.5,2) {$*$};
\node (v9) at (0,2.8) {};
\node (v10) at (1,2.8) {};
\node (v11) at (1.5,3.5) {$S^{n+1}$};
\node (v12) at (1.5,2) {$*$};
\draw[->]  (v1) edge[double] (v2);
\draw[->,]  (v3) edge node[above, font = \scriptsize] {$\partial$} (v4);
\draw[->]  (v3) edge node[above, font = \scriptsize] {$\partial$}(v5);
\draw[->]  (v6) edge (v7);
\draw[->]  (v6) edge (v8);
\draw[->]  (v4) edge (v7);
\draw[->]  (v5) edge (v8);
\draw[|->]  (v9) edge node[above, font = \scriptsize] {colim}  (v10);
\draw[->]  (v11) edge (v12);
\draw[->]  (v3) edge node[left, font = \scriptsize] {$1_{S^n}$}(v6);
\end{tikzpicture}
\end{center}
Here $S^n$ and $D^n$ denote the $n$-sphere and the $n$-disc as usual, $\partial$ the boundary inclusion and $*$ the one-point space. The components of this map of diagrams are homotopy equivalences but the resulting colimits are not homotopy equivalent. 

The homotopy colimit construction is a remedy for this problem. 
 Roughly speaking, instead of identifying things we make them homotopic. This proccess can be described as follows:

\begin{enumerate}
	\item For every chain of $n$ morphisms ($n$-simplex) in $\mathcal{J}$ $$i_0 \xrightarrow{\alpha_1} i_1 \rightarrow ... \rightarrow i_{n-1} \xrightarrow{\alpha_n} i_n$$
	and a point $x \in X_{i_0}$ we declare a (topological) $n$-simplex $\Delta^n(x)$ whose edges are $\{x, f_{\alpha_1}(x), f_{\alpha_2\alpha_1}(x),..., f_{\alpha_n \dots\alpha_1}(x) \}$. In particular for $n=0$ we just obtain the points of the spaces $X_i$.
	\item Make the obvious identifications. For example $\Delta^{n-1}(f_{\alpha_1}(x))$ is to be considered as a face of $\Delta^n(x)$. Then impose the appropriate topology.
\end{enumerate}

This time we could not be really precise in our recipe but we will come back to this topic in \ref{ssetcat} in the context of simplicially enriched categories. We refer the reader to \cite{dugger2008primer} for a broader viewpoint on the subject (and better pictures). Our aim is to provide some geometric intuition.  

Let $\mathcal{J} = \Delta^1 = (0 \rightarrow 1)$ be the interval category, so that a diagram $X : \mathcal{J} \rightarrow \textbf{Top}$ is just a continuous map $$f: X_0 \rightarrow X_1$$. In this case we join with an interval ($1$-simplex) every point $x \in X_0$ with its image $f(x) \in X_1$. The resulting space (after imposing the correct topology) is the mapping cylinder of $f$ (see Figure \ref{mapcyl}) which we denote $M_f$ as usual. 

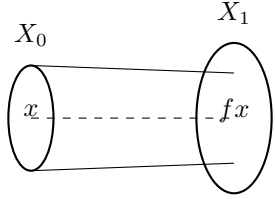
\begin{figure}
	\begin{center}
		\begin{tikzpicture}

		\draw[thick] (-4.1,3.3) ellipse (0.3 and 0.7);
		
		\draw[thick]  (-1.4,3.3) ellipse (0.5 and 1);
		
		\draw (-4.1,4) -- (-1.4,3.9);
		\draw (-4.1,2.6) -- (-1.4,2.7);
		
		\node at (-4.1,4.4) {$X_0$};
		\node at (-1.4,4.7) {$X_1$};
		
		\draw[dashed] (-4.1,3.3) -- (-1.4,3.3);
		\node at (-4.1,3.4) {$x$};
		\node at (-1.4,3.4) {$fx$};
		
		\end{tikzpicture}
	\end{center}
	
	\caption{The mapping cylinder of $f: X_0 \rightarrow X_1$}
		\label{mapcyl}
\end{figure}

If $\mathcal{J} = \Delta^2 = (0 \rightarrow 1 \rightarrow 2)$ then a diagram $X : \mathcal{J} \rightarrow \textbf{Top}$  is simply a commutative triangle in the category of spaces 
$$X_0 \xrightarrow{f} X_1 \xrightarrow{g} X_2$$
 We form its homotopy colimit by first attaching $1$-simplices according to $f$,$g$ and $gf$ to obtain the mapping cylinders $M_f$, $M_g$, and $M_{gf}$. Then we attach the $2$-simplices corresponding to the chain $(f,g)$. The $1$-faces of these $2$-simplicies will lie in $M_f$, $M_g$ and $M_{gf}$ (see Figure \ref{hocolim}).

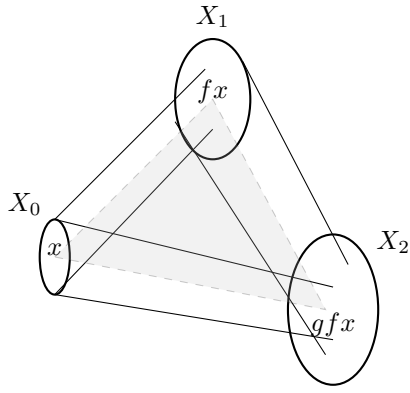
\begin{figure}
	\begin{center}
		\begin{tikzpicture}
		
		\draw[thick]  (-4.1,3.6) ellipse (0.2 and 0.5);
		\draw[thick]  (-2,5.7) ellipse (0.5 and 0.8);
		\draw[thick]  (-0.4,2.9) ellipse (0.6 and 1);
		
		\draw (-4.1,4.1) -- (-2.1,6.1);
		\draw (-4.1,3.1) -- (-2,5.3);
		
		\draw (-1.6,6.2) -- (-0.2,3.5);
		\draw (-2.5,5.4) -- (-0.5,2.3);
		
		\draw (-4.1,4.1) -- (-0.4,3.2);
		\draw (-4.1,3.1) -- (-0.4,2.5);
		
		\draw[dashed,  fill = lightgray, opacity=0.2] (-4.1,3.6)--(-2,5.7)--(-0.5,2.9)--cycle;
		
		\node at (-4.1,3.7) {$\tiny x$};
		\node at (-2,5.8) {$fx$};
		\node at (-0.4,2.7) {$gfx$};
		
		\node at (-4.5,4.3) {$X_0$};
		\node at (-2,6.8) {$X_1$};
		\node at (0.4,3.8) {$X_2$};
		\end{tikzpicture}
	\end{center}
	\caption{The homotopy colimit of $X_0 \xrightarrow{f} X_1 \xrightarrow{g} X_2$}
	\label{hocolim}
\end{figure}

 In general, if $$\sigma = (X_0 \xrightarrow{f_1} X_1 \rightarrow ... \rightarrow X_{n-1} \xrightarrow{f_n} X_n)$$ is an $n$-simplex in the category of spaces and we regard it as a diagram $$X : \Delta^n \rightarrow \textbf{\text{Top}}$$
  ,its homotopy colimit of this diagram should be considered as the mapping cylinder of $\sigma$. We may call it \textit{higher mapping cylinder} and denote it $M_\sigma$. 
  
  In the example of the two pushout diagrams above the homotopy colimit of both rows is homotpy equivalent to $S^{n+1}$. We invite the reader to contemplate this fact in light of the above figures.

\subsection{The Grothendieck construction for categories}
\label{gro}

Just as the notion of colimit makes sense in every category if described in the proper language, the notion of homotopy colimit as described above makes sense in a variety of categories. Our recipe depends on the presence of simplices in the category of spaces. In the world of categories, as mentioned in \ref{sSet}, we can regard the poset $\Delta^n = (0 \rightarrow 1 \rightarrow ... \rightarrow n-1 \rightarrow n)$ as the \textit{categorical $n$-simplex}.

Let \textbf{Cat} be the category of small categories and functors between them.  Now for a small category $\mathcal{J}$ and a diagram 
$$F : \mathcal{J} \rightarrow \textbf{\text{Cat}}$$ 
we can reiterate our recipe and produce a category serving as the homotopy colimit of this diagram. In this case we can decribe the resulting category explicitely as the \textit{Grothendieck construction} of the diagram $\mathcal{C}$, which we denote $\text{\textbf{Gro}}(\mathcal{C})$.

Let $\mathcal{C}_i = F(i)$ for $i \in \mathcal{J}$ and $F_\alpha: \mathcal{C}_i \rightarrow \mathcal{C}_j$ denote the image of a morphism $\alpha: i \rightarrow j$ in $\mathcal{J}$. Then the objects of $\textbf{Gro}(\mathcal{C})$ are  objects of the categories $\mathcal{C}_i$ for $i \in \mathcal{J}$. 
$$ob(\textbf{Gro}(\mathcal{C})) = \coprod_{i \in \mathcal{J}} ob(\mathcal{C}_i)$$
Given two objects $x \in \mathcal{C}_i$ and $y \in \mathcal{C}_j$ a morphism between them is a pair $(\alpha,f)$ where $\alpha: i \rightarrow j$ is an arrow in $\mathcal{J}$ and $f :F_\alpha(x) \rightarrow y$ is an arrow in $\mathcal{C}_j$. It helps to visualise $(\alpha,f)$ as 

\begin{center}
	\begin{tikzpicture}

\node (v1) at (-3,3.5) {$x$};
\node (v2) at (-1.5,3.5) {$F_\alpha x$};
\node (v3) at (-1.5,2) {$y$};
\draw[|->]  (v1) edge node[above, font = \scriptsize]{$F_\alpha$} (v2);
\draw[->]  (v2) edge node[right, font = \scriptsize]{$f$} (v3);
\end{tikzpicture}
\end{center}
We can think of the "assignment" arrow $F_\alpha : x \mapsto F_\alpha x$ as the $1$-simplex attached to join $x$ and $F_\alpha x$ corresponding to the morphism $\alpha: i \rightarrow j$ in $\mathcal{J}$. As a matter of fact this morphism is precisely $(\alpha,1_{F_\alpha x})$. 

Now let $z \in \mathcal{C}_k$ be another object, $\beta : j \rightarrow k$ be an arrow in $\mathcal{J}$ and $(\beta,g) : y \rightarrow z$ be a morphism in $\textbf{Gro}(\mathcal{C})$. We define its composition with $(F,f)$ via the formula 
$$(\beta,g)(\alpha,f) = (\beta\alpha,g \circ F_\beta(f))$$
The following picture may help in understanding this composition

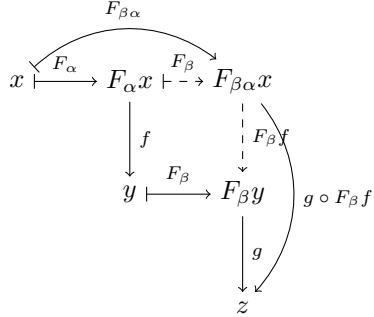
\begin{center}
	\begin{tikzpicture}
	
	\node (v1) at (-3,3.5) {$x$};
	\node (v2) at (-1.5,3.5) {$F_\alpha x$};
	\node (v3) at (-1.5,2) {$y$};
	\draw[|->]  (v1) edge node[above, font = \scriptsize]{$F_\alpha$} (v2);
	\draw[->]  (v2) edge node[right, font = \scriptsize]{$f$} (v3);
	
	\node (v6) at (0,3.5) {$F_{\beta\alpha}x$};
	\node (v4) at (0,2) {$F_\beta y$};
	\node (v5) at (0,0.5) {$z$};
	\draw[|->]  (v3) edge node[above, font = \scriptsize]{$F_\beta$} (v4);
	\draw[->]  (v4) edge node[right, font = \scriptsize]{$g$} (v5);
	\draw[|->, dashed]  (v2) edge node[above, font = \scriptsize]{$F_\beta$} (v6);
	\draw[->, dashed]  (v6) edge node[right, font = \scriptsize]{$F_\beta f$} (v4);
	\draw[|->]  (v1) edge[bend left=40] node[above, font = \scriptsize]{$F_{\beta\alpha}$} (v6);
	\draw[->]  (v6) edge[bend left=40] node[right, font = \scriptsize]{$g \circ F_\beta f$} (v5);
	\end{tikzpicture}
\end{center}
We may think of the composable pair $(F_\alpha,F_\beta)$ in the picture as the $2$-simplex attached at $x$ corresponding to the chain $i \xrightarrow{\alpha} j \xrightarrow{\beta} k$ in $\mathcal{J}$. And so on in higher dimension for longer chains.

\subsection{Profunctors}
\label{prof}

Let $\mathcal{C}$, $\mathcal{D}$ be two categories and let \textbf{Set} denote the category of sets and functions. A \textit{profunctor} $u$ from $\mathcal{C}$ to $\mathcal{D}$ is a functor $$u: \mathcal{C}^{\text{op}} \times \mathcal{D} \rightarrow \text{\textbf{Set}}$$
where $\mathcal{C}^{\text{op}}$ denotes the opposite category of $\mathcal{C}$.

\subsubsection{Collages}

 We may understand profunctors in a few ways. In analogy with algebra a profunctor $u$ can be thought of as a $(\mathcal{C}, \mathcal{D})$-\textit{bimodule}. If $\mathcal{C}$ is a group $G$, meaning it has one object and all arrows are invertible, then a functor $\mathcal{C}^{\text{op}} \rightarrow \text{\textbf{Set}}$ is a left $G$-set and a functor $\mathcal{C} \rightarrow \text{\textbf{Set}}$ is a right $G$-set. If $\mathcal{D}$ is also a group $H$ then a profunctor $u$ is a $(G,H)$-biset, meaning it is a set with left $G$-action and right $H$-action. 
 
 The other perspective is to understand profunctors in terms of collages.
 
 \begin{definition}
 	Let $\mathcal{C}$ and $\mathcal{D}$ be categories. A collage between $\mathcal{C}$ and $\mathcal{D}$ is a category $U$ such that:
 	\begin{itemize}
 		\item[i)] the objects of $U$ are precisely those of $\mathcal{C}$ and $\mathcal{D}$
 		$$\text{ob}(U) = \text{ob}(\mathcal{C}) \coprod \text{ob}(\mathcal{D})$$
 		\item[ii)] if $c,c^\prime \in \mathcal{C}$ then $U(c,c^\prime) = \mathcal{C}(c,c^\prime)$ and if $d,d^\prime \in \mathcal{D}$ then $U(d,d^\prime) = \mathcal{D}(d,d^\prime)$
 		\item[iii)] if $c \in \mathcal{C}$ and $d \in \mathcal{D}$ then $U(d,c) = \emptyset$
 	\end{itemize}
 \end{definition}

In other words a collage is a category constructed by adding new morphisms from objects of $\mathcal{C}$ to objects of $\mathcal{D}$, which we abstractly depict as

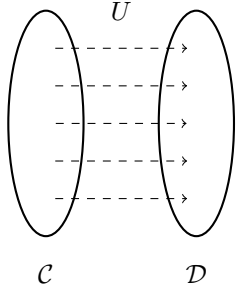
\begin{center}
	\begin{tikzpicture}
	
	\draw[thick]  (-4,3.5) node (v5) {} ellipse (0.5 and 1.5);
	\draw[thick]  (-2,3.5) node (v6) {} ellipse (0.5 and 1.5);
	\node (v1) at (-4,4.5) {};
	\node (v2) at (-2,4.5) {};
	\node (v3) at (-4,4) {};
	\node (v4) at (-2,4) {};
	\node (v7) at (-4,3) {};
	\node (v8) at (-2,3) {};
	\node (v9) at (-4,2.5) {};
	\node (v10) at (-2,2.5) {};
	\draw[->, dashed]  (v1) edge (v2);
	\draw[->, dashed]  (v3) edge (v4);
	\draw[->, dashed]  (v5) edge (v6);
	\draw[->, dashed]  (v7) edge (v8);
	\draw[->, dashed]  (v9) edge (v10);
	\node at (-4,1.5) {$\mathcal{C}$};
	\node at (-2,1.5) {$\mathcal{D}$};
	\node at (-3,5) {$U$};
	\end{tikzpicture}
\end{center}

Given a collage $U$ we may define a profunctor
$$\bar{U} : \mathcal{C}^{op} \times \mathcal{D} \rightarrow \text{\textbf{Set}}$$
simply by $$\bar{U}(c,d) = U(c,d)$$ for $(c,d) \in \mathcal{C}^{op} \times \mathcal{D}$. The action of $\bar{U}$ on morphisms is prescribed by the composition in $U$. More precisely, given morphisms $f: c^\prime \rightarrow c$ in $\mathcal{C}$ and $g: d \rightarrow d^\prime$ in $\mathcal{D}$, so that we have a morphism $(f,g) : (c,d) \rightarrow (c^\prime, d^\prime)$ in $\mathcal{C}^{op} \times \mathcal{D}$, we define the function $\bar{U}(f,g) : \bar{U}(c,d) \rightarrow \bar{U}(c^\prime, d^\prime)$ as 
$$\bar{U}(f,g)(x) = g \circ x \circ f$$
where $\circ$ is the composition in $U$ and $x \in U(c,d)$.

Given a profunctor 
$$u: \mathcal{C}^{op} \times \mathcal{D} \rightarrow \text{\textbf{Set}}$$
we may define a collage between $\mathcal{C}$ and $\mathcal{D}$, called the collage of $u$ and denoted $\text{col}(u)$, by letting the set of morphisms from $c \in \mathcal{C}$ to $d \in \mathcal{D}$ to be $u(c,d)$. Composition in $\text{col}(u)$ is given by the action of $u$. Given $f: c^\prime \rightarrow c$ in $\mathcal{C}$, $g: d \rightarrow d^\prime$ in $\mathcal{D}$ and $x \in u(c,d)$ we define their composite to be
$$g \circ x \circ f = u(f,g)(x)$$

\begin{proposition}
	The above constructions are mutual inverses of each other and they define a one-to-one correspondence between the collection of profunctors from $\mathcal{C}$ to $\mathcal{D}$ to the collection of collages between $\mathcal{C}$ and $\mathcal{D}$ for any two categories $\mathcal{C}$, $\mathcal{D}$.
\end{proposition}

The proof of the proposition is obvious at this point but it is its statement we are interested in. We have another interpretation of profunctors: they are simply collages. While our previous interpretation was algebraic, we may arguably consider the latter to be \textit{topological}. 

There is a third way to look at profunctors: observe that a collage $U$ between $\mathcal{C}$ and $\mathcal{D}$ is simply given by a functor 
$$p: U \rightarrow \Delta^1$$
with $p^{-1}(0) = \mathcal{C}$ and $p^{-1}(1) = \mathcal{D}$. We elegantly pack the conditions in the definition of collages and we may label this perspective \textit{combinatorial}. 

Profunctors are found a lot in nature. For example given a functor $F: \mathcal{C} \rightarrow \mathcal{D}$ there is a distinguished profunctor associated to $F$, denoted $F^* : \mathcal{C}^{\text{op}} \times \mathcal{D} \rightarrow \text{\textbf{Set}}$, given by $$F^*(c,d) = \mathcal{D}(Fc,d)$$
This profunctor is called the companion of $F$. If we regard the functor $F$ as a diagram
$$F : \Delta^1 \rightarrow \text{\textbf{Cat}}$$
where $\Delta^1$ is the interval category, then we have (by definition)
$$\text{\textbf{Gro}}(F) = \text{col} (F^*)$$
This observation will become more interesting once we have interpreted the companion and the collage construction formally through double category theory.

\subsubsection{Tensoring as formal composition}

Let $\mathcal{C}$, $\mathcal{D}$, $\mathcal{E}$ be categories and $u: \mathcal{C}^{\text{op}} \times \mathcal{D} \rightarrow \text{\textbf{Set}}$, $v: \mathcal{D}^{\text{op}} \times \mathcal{E} \rightarrow \text{\textbf{Set}}$ be profunctors. Since profunctors are bimodules we should be able to tensor $u$ and $v$ over $\mathcal{D}$ and produce a profunctor 
$$v \otimes_\mathcal{D} u : \mathcal{C}^{\text{op}} \times \mathcal{E} \rightarrow \text{\textbf{Set}}$$  
For a pair of objects $c\in \mathcal{C}$ and $e \in \mathcal{E}$ we define the set $(v \otimes_\mathcal{D} u)(c,e)$ to be given as a coequalizer
$$\coprod_{f: d \rightarrow d^\prime \in \mathcal{D}} v(d^\prime,e)  \times u(c,d) \rightrightarrows \coprod_{d \in \mathcal{D}} v(d,e) \times u(c,d) \dashrightarrow (v \otimes_\mathcal{D} u) (c,e) $$
The first coproduct is indexed over all morphisms in $\mathcal{D}$ and the second over all objects of $\mathcal{D}$. The two parallel morphisms forming the above fork are defined, for a fixed $f : d \rightarrow d^\prime$, via $(y,x) \mapsto (y, fx)$ and $(y,x) \mapsto (yf, x)$ for $y \in v(d^\prime,e)$ and $x \in u(c,d)$.

This definition has the virtue that it applies to enriched categories as well and recovers the tensor product of bimodules over rings as a special case (see \cite{riehl2009weighted}). 

But profunctors are collages as well and their tensor product has a meaningful interpretation from this perspective too.
Consider the collages of $u$ and $v$ and juxtapose them along $\mathcal{D}$. 
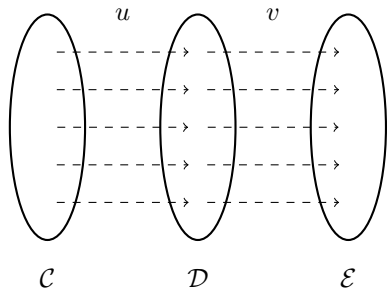
\begin{center}
	\begin{tikzpicture}
	
	\draw[thick]  (-4,3.5) node (v5) {} ellipse (0.5 and 1.5);
	\draw[thick]  (-2,3.5) node (v6) {} ellipse (0.5 and 1.5);
	\draw[thick]  (0,3.5) node (v13) {} ellipse (0.5 and 1.5);
	
	\node (v1) at (-4,4.5) {};
	\node (v2) at (-2,4.5) {};
	\node (v3) at (-4,4) {};
	\node (v4) at (-2,4) {};
	\node (v7) at (-4,3) {};
	\node (v8) at (-2,3) {};
	\node (v9) at (-4,2.5) {};
	\node (v10) at (-2,2.5) {};
	\draw[->, dashed]  (v1) edge (v2);
	\draw[->, dashed]  (v3) edge (v4);
	\draw[->, dashed]  (v5) edge (v6);
	\draw[->, dashed]  (v7) edge (v8);
	\draw[->, dashed]  (v9) edge (v10);
	\node at (-4,1.5) {$\mathcal{C}$};
	\node at (-2,1.5) {$\mathcal{D}$};
	
	\node (v11) at (0,4.5) {};
	\node (v12) at (0,4) {};
	\node (v14) at (0,3) {};
	\node (v15) at (0,2.5) {};
	\draw[->, dashed]   (v2) edge (v11);
	\draw[->, dashed]   (v4) edge (v12);
	\draw[->, dashed]   (v6) edge (v13);
	\draw[->, dashed]   (v8) edge (v14);
	\draw[->, dashed]   (v10) edge (v15);
	\node at (0,1.5) {$\mathcal{E}$};
	
	\node at (-3,5) {$u$};
	\node at (-1,5) {$v$};
	\end{tikzpicture}
\end{center} 
What results is not a category because we do not have composites of morphisms in $u(c,d)$ with those in $v(d,e)$. What we do in this case is freely generate a category out of the given data. That is, we add a formal composite $y \otimes x$ for all $x \in u(c,d)$ and $y \in v(d,e)$. Since we want our formal composites to be coherent with existing compositions, for all $f : d \rightarrow d^\prime$, $x \in u(c,d)$ and $y \in v(d^\prime,e)$ we impose the relation 
$$y \otimes (fx) = (yf) \otimes x$$
We obtain a category that looks something like this
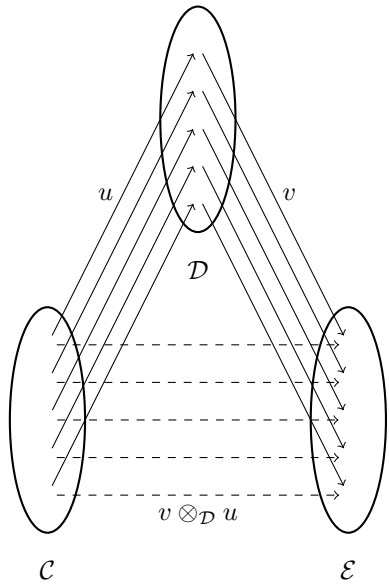
\begin{center}
	\begin{tikzpicture}
	
	\draw[thick]  (-4,3.5) node (v5) {} ellipse (0.5 and 1.5);
	\draw[thick]  (-2,7.5) node (v6) {} ellipse (0.5 and 1.5);
	\draw[thick]  (0,3.5) node (v13) {} ellipse (0.5 and 1.5);
	
	\node (v1) at (-4,4.5) {};
	\node (v2) at (-2,8.5) {};
	\node (v3) at (-4,4) {};
	\node (v4) at (-2,8) {};
	\node (v7) at (-4,3) {};
	\node (v8) at (-2,7) {};
	\node (v9) at (-4,2.5) {};
	\node (v10) at (-2,6.5) {};
	\draw[->]  (v1) edge node[left]{$u$} (v2);
	\draw[->]  (v3) edge (v4);
	\draw[->]  (v5) edge (v6);
	\draw[->]  (v7) edge (v8);
	\draw[->]  (v9) edge (v10);
	\node at (-4,1.5) {$\mathcal{C}$};
	\node at (-2,5.5) {$\mathcal{D}$};
	
	\node (v11) at (0,4.5) {};
	\node (v12) at (0,4) {};
	\node (v14) at (0,3) {};
	\node (v15) at (0,2.5) {};
	\draw[->]   (v2) edge node[right]{$v$} (v11);
	\draw[->]   (v4) edge (v12);
	\draw[->]   (v6) edge (v13);
	\draw[->]   (v8) edge (v14);
	\draw[->]   (v10) edge (v15);
	\node at (0,1.5) {$\mathcal{E}$};
	
	\draw[->, dashed]   (v1) edge  (v11);
	\draw[->, dashed]   (v3) edge (v12);
	\draw[->, dashed]   (v5) edge (v13);
	\draw[->, dashed]   (v7) edge (v14);
	\draw[->, dashed]   (v9) edge node[below]{$v \otimes_\mathcal{D} u$} (v15);
	\end{tikzpicture}
\end{center} 
Then we discard the objects of $\mathcal{D}$ to obtain a collage between $\mathcal{C}$ and $\mathcal{E}$ which corresponds precisely to the tensor $v \otimes_\mathcal{D} u$.

It is easy to see how this coincides with what the above coequalizer presents. We trust the reader to make the latter interpretation more precise (see for example reflexive graphs with relations in \cite[p.~142]{riehl2017category}). What we intend to draw attention to is the similarity between the above picture and the one we used to illustrate the homotopy colimit of a commutative triangle of spaces (Figure \ref{hocolim}). 

If we look at profunctors as functors to $\Delta^1$ then the above process can be described as follows: 
\begin{enumerate}
	\item Forst consider the commutative diagram which describes a transformation of pushout diagrams
	\begin{center}
		\begin{tikzpicture}

		\node (v4) at (-3.5,3.5) {$\text{col}(u)$};
		\node (v3) at (-2,3.5) {$\mathcal{D}$};
		\node (v5) at (-0.5,3.5) {$\text{col}(v)$};
		\node (v7) at (-3.5,2) {$\Delta^1$};
		\node (v6) at (-2,2) {$\Delta^0$};
		\node (v8) at (-0.5,2) {$\Delta^1$};

		\draw[left hook->]  (v3) edge  (v4);
		\draw[right hook->]  (v3) edge (v5);
		\draw[left hook->]  (v6) edge node[above, font = \scriptsize] {$d^1$} (v7);
		\draw[right hook->]  (v6) edge node[above, font = \scriptsize] {$d^0$} (v8);
		\draw[->]  (v4) edge (v7);
		\draw[->]  (v5) edge (v8);

		\draw[->]  (v3) edge (v6);
		\end{tikzpicture}
	\end{center}
	where the maps in the top row are the obvious inclusions.
	
	\item Take the pushout of both rows to obtain an induced functor
	$$p: \text{col}(u) \coprod_\mathcal{D} \text{col}(v) \rightarrow \Delta^2$$
	
	\item Let the collage of $v \otimes_\mathcal{D} u$ be given by pulling back along the first face of $\Delta^2$
	\begin{center}
		\begin{tikzpicture}
		
		\node (v4) at (-2.5,3) {$\text{col}(v \otimes_\mathcal{D} u)$};
		\node (v3) at (-2.5,1) {$\Delta^1$};
		\node (v1) at (1,3) {$\text{col}(u) \coprod_\mathcal{D} \text{col}(v)$};
		\node (v2) at (1,1) {$\Delta^2$};
		\draw[->]  (v1) edge node[right] {$p$} (v2);
		\draw[right hook ->]  (v3) edge node[above] {$d^1$} (v2);
		\draw[->, dashed]  (v4) edge (v1);
		\draw[->, dashed]  (v4) edge (v3);
		\end{tikzpicture}
	\end{center}
\end{enumerate}

\begin{remark}
	We are taking huge advantage of the fact that in the category of categories the pushout of $$\Delta^1 \xleftarrow{d^1} \Delta^0 \xrightarrow{d^0} \Delta^1$$
	gives us $\Delta^2$. All of the above work for the category of simplicial sets except this pushout would be the horn $\Lambda^2_1$ and hence we cannot define the tensor product of collages between simplicial sets. This happens because we create the category of simplicial sets as a free cocompletion of $\Delta$ and while we add new formal colimits we destroy the ones already present in $\Delta$.
\end{remark}

\subsection{Double categories and equipments}
\label{dblcat}

We may consider profunctors $\mathcal{C}^{\text{op}} \times \mathcal{D} \rightarrow \text{\textbf{Set}}$ as some type of morphism $\mathcal{C} \rightarrow \mathcal{D}$ with composition given by tensoring. However we do not obtain a category because composition is not strictly associative and unital but only so up to invertible natural transformation. This way categories, profunctors and transformations between them are organized in a (weak) 2-category (see \cite{leinster2004higher}). 

On the other hand we have functors as morphisms between categories. Categories, functors and natural transformations are themselves organized in a 2-category. Functors and profunctors are not unrelated objects. We saw for example that we have the companion profunctor $F^*$ associated to a functor $F$ and the collage of $F^*$ gives us the mapping cylinder of $F$. 

Double categories are a natural organizing principle for such situations. Their category theory is studied systematically in the the series of papers \cite{grandis1999limits,grandis2004adjoint,grandis2008kan} (which we follow). In particular we are interested in double colimits. 

The double category of functors and profunctors has nice extra properties and constitutes the paradigmatic example of an equipment. We will follow \cite{shulman2008framed} and \cite{shulman2009equip} for the theory of equipments.  
Finally we will observe that the Grothendieck construction for categories can be recovered from  this equipment as a double colimit. This in turn will serve as a paradigm for the rest of this work.

\subsubsection{Double categories} 

(Strict) double categories are category objects in the category of categories (i.e. internal categories). We will not need this perspective here since most examples are weak and it is not illuminating to think of double categories as internal objects. 
In what follows we pack definition, notation and pictures. We insist in depicting categorical structures as it enables us to recognize them when they are in front of us.

A double category $\mathbb{D}$ consists of:

\begin{itemize}
	\renewcommand\labelitemi{--}
	\item Objects $a,b,c...$ which form a class we denote $ob\mathbb{D}$. We may abuse and simply write $a \in \mathbb{D}$
	\item A category $\mathbb{D}_0$ with objects $ob\mathbb{D}$ called the \textit{vertical category} of $\mathbb{D}$. The morphisms in $\mathbb{D}_0$ are called \textit{vertical morphisms}
	\item For any two objects $a,b \in \mathbb{D}$ a collection $\mathbb{D}_h(a,b)$ whose elements are called \textit{horizontal morphisms} 
	\item 2-cells of the form 
	\begin{center}
		\begin{tikzpicture}
		
		\node (v1) at (-2.8,3.2) {$a$};
		\node (v3) at (-1.2,3.2) {$b$};
		\node (v2) at (-2.8,1.8) {$c$};
		\node (v4) at (-1.2,1.8) {$d$};
		\draw[->]  (v1) edge node[left, font= \scriptsize]{$f$} (v2);
		\draw[->]  (v1) edge node[above, font= \scriptsize]{$u$} (v3);
		\draw[->]  (v3) edge node[right, font= \scriptsize]{$g$} (v4);
		\draw[->]  (v2) edge node[below, font= \scriptsize]{$v$} (v4);
		
		\node at (-2,2.5) {$\alpha$};
		\end{tikzpicture}
	\end{center}
	where $f$,$g$ are vertical morphisms and $u$, $v$ are horizontal morphisms.  $f$ and $g$ are the vertical domain and codomain of the 2-cell $\alpha$, $u$ and $v$ are the horizontal domain and codomain of $\alpha$.
	
	We denote $\alpha : u \Rightarrow v$ a 2-cell $\alpha$  of the form
	\begin{center}
		\begin{tikzpicture}
		
		\node (v1) at (-2.8,3.2) {$a$};
		\node (v3) at (-1.2,3.2) {$b$};
		\node (v2) at (-2.8,1.8) {$a$};
		\node (v4) at (-1.2,1.8) {$b$};
		\draw[->]  (v1) edge node[left, font= \scriptsize]{$1_a$} (v2);
		\draw[->]  (v1) edge node[above, font= \scriptsize]{$u$} (v3);
		\draw[->]  (v3) edge node[right, font= \scriptsize]{$1_b$} (v4);
		\draw[->]  (v2) edge node[below, font= \scriptsize]{$v$} (v4);
		
		\node at (-2,2.5) {$\alpha$};
		\end{tikzpicture}
	\end{center}
	\item For all $u \in \mathbb{D}_h(a,b)$ a 2-cell
	\begin{center}
		\begin{tikzpicture}
		
		\node (v1) at (-2.8,3.2) {$a$};
		\node (v3) at (-1.2,3.2) {$b$};
		\node (v2) at (-2.8,1.8) {$a$};
		\node (v4) at (-1.2,1.8) {$b$};
		\draw[->]  (v1) edge node[left, font= \scriptsize]{$1_a$} (v2);
		\draw[->]  (v1) edge node[above, font= \scriptsize]{$u$} (v3);
		\draw[->]  (v3) edge node[right, font= \scriptsize]{$1_b$} (v4);
		\draw[->]  (v2) edge node[below, font= \scriptsize]{$u$} (v4);
		
		\node at (-2,2.5) {$1_u$};
		\end{tikzpicture}
	\end{center}
	called the unit 2-cell of $u$
	\item Vertical composition of 2-cells
	
	\begin{center}
		\begin{tikzpicture}
			\node (v1) at (-2.8,3.2) {$a$};
			\node (v3) at (-1.2,3.2) {$b$};
			\node (v2) at (-2.8,1.8) {$a^\prime$};
			\node (v4) at (-1.2,1.8) {$b^\prime$};
			\draw[->]  (v1) edge node[left, font= \scriptsize]{$f$} (v2);
			\draw[->]  (v1) edge node[above, font= \scriptsize]{$u$} (v3);
			\draw[->]  (v3) edge node[right, font= \scriptsize]{$g$} (v4);
			\draw[->]  (v2) edge node[above, font= \scriptsize]{$v$} (v4);
			\node at (-2,2.5) {$\alpha$};
			
			\node (v5) at (-2.8,0.4) {$a^{\prime\prime}$};
			\node (v6) at (-1.2,0.4) {$b^{\prime\prime}$};
			\node at (-2,1) {$\beta$};
			\draw[->]  (v2) edge node[left, font= \scriptsize]{$f^\prime$} (v5);
			\draw[->]  (v4) edge node[right, font= \scriptsize]{$g^\prime$} (v6);
			\draw[->]  (v5) edge node[below, font= \scriptsize]{$w$} (v6);
			
			\node (v7) at (-0.4,1.8) {};
			\node (v8) at (0.8,1.8) {};
			\draw[|->]  (v7) edge node[above, font= \scriptsize]{compose} (v8);
			
			\node (v9) at (1.4,2.5) {$a$};
			\node (v10) at (1.4,1.1) {$a^{\prime\prime}$};
			\node (v11) at (2.8,2.5) {$b$};
			\node (v12) at (2.8,1.1) {$b^{\prime\prime}$};
			\draw[->]  (v9) edge node[left, font= \scriptsize]{$f^\prime f$} (v10);
			\draw[->]  (v9) edge node[above, font= \scriptsize]{$u$} (v11);
			\draw[->]  (v11) edge node[right, font= \scriptsize]{$g\prime g$} (v12);
			\draw[->]  (v10) edge node[below, font= \scriptsize]{$w$} (v12);
			\node[->] at (2.1,1.8) {$\beta\alpha$};
		\end{tikzpicture}
	\end{center}
	which is unital and associative. We might also write $\beta \circ \alpha$ for this composition.
	\item For any three objects $a,b,c \in \mathbb{D}$ a composition  operation $\otimes$ assigning to a pair $u \in \mathbb{D}_h(a,b)$, $v \in \mathbb{D}_h(b,c)$ a horizontal morphism $v \otimes u \in \mathbb{D}_h(a,c)$. 
	\item For each object $a$ a horizontal morphism $1_a^* \in \mathbb{D}_h(a,a)$ called the \textit{horizontal unit} of $a$ which we draw as 
	\begin{center}
		\begin{tikzpicture}
		\node (v1) at (-4,3.6) {$a$};
		\node (v2) at (-2.8,3.6) {$a$};
		\draw  (v1) edge[double] (v2);
		\end{tikzpicture}
	\end{center}
	\item Horizontal conposition of 2-cells
	\begin{center}
		\begin{tikzpicture}

		\node (v1) at (-4.2,4.2) {$a$};
		\node (v2) at (-4.2,2.8) {$a^\prime$};
		\node (v3) at (-2.8,4.2) {$b$};
		\node (v4) at (-2.8,2.8) {$b^\prime$};
		\node (v5) at (-1.4,4.2) {$c$};
		\node (v6) at (-1.4,2.8) {$c^\prime$};
		\node (v7) at (-0.7969,3.5059) {};
		\node (v8) at (0.6315,3.5059) {};
		\node (v9) at (1.2,4.2) {$a$};
		\node (v10) at (1.2,2.8) {$b$};
		\node (v11) at (2.6,4.2) {$c$};
		\node (v12) at (2.6,2.8) {$c^\prime$};
		\node at (-3.4731,3.4731) {$\alpha$};
		\node at (-2.0447,3.4895) {$\beta$};
		\node at (1.9614,3.4731) {$\beta * \alpha$};
		\draw[->]  (v1) edge node[left, font = \scriptsize]{$f$} (v2);
		\draw[->]  (v1) edge node[above, font = \scriptsize]{$u$} (v3);
		\draw[->]  (v3) edge node[left, font = \scriptsize]{$g$} (v4);
		\draw[->]  (v2) edge node[below, font = \scriptsize]{$u^\prime$} (v4);
		\draw[->]  (v3) edge node[above, font = \scriptsize]{$v$} (v5);
		\draw[->]  (v5) edge node[right, font = \scriptsize]{$h$} (v6);
		\draw[->]  (v4) edge node[below, font = \scriptsize]{$v^\prime$} (v6);
		\draw[|->]  (v7) edge node[above, font = \scriptsize]{compose}(v8);
		\draw[->]  (v9) edge node[left, font = \scriptsize]{$f$} (v10);
		\draw[->]  (v9) edge node[above, font = \scriptsize]{$v \otimes u$} (v11);
		\draw[->]  (v11) edge node[right, font = \scriptsize]{$h$} (v12);
		\draw[->]  (v10) edge node[below, font = \scriptsize]{$v^\prime \otimes u^\prime$} (v12);
		\end{tikzpicture}
	\end{center}
\end{itemize}
	 
	 This definition would end here if $\otimes$ was associative and unital, but as our example suggests we must account for weak associativity and unital laws. So we postulate unitors and associators for the horizontal direction of $\mathbb{D}$.
	 \begin{itemize}
	 \renewcommand\labelitemi{--}
	
	\item For each horizontal morphism $u \in \mathbb{D}_h(a,b)$ invertible 2-cells $r_u: 1_b^* \otimes u \rightarrow u$ and $l_u : u \otimes 1_a^* \rightarrow u$ called the \textit{right} and \textit{left unitor} respectively. We may depict them as 
	
	\begin{center}
		\begin{tikzpicture}
	\node (v1) at (-4,3.6) {$a$};
	\node (v2) at (-2.8,3.6) {$b$};
	\draw[->]  (v1) edge node[above, font= \scriptsize]{$u$} (v2);
	\node (v3) at (-1.6,3.6) {$b$};
	\node (v4) at (-4,2.4) {$a$};
	\node (v5) at (-1.6,2.4) {$b$};
	\draw  (v2) edge[double] (v3);
	\draw[->]   (v1) edge node[left, font= \scriptsize]{$1_a$} (v4);
	\draw[->]   (v3) edge node[right, font= \scriptsize]{$1_b$} (v5);
	\draw[->]   (v4) edge node[below, font= \scriptsize]{$u$} (v5);
	\node at (-2.7835,2.9641) {$r_u$};
	\end{tikzpicture} \ \ \ \ \ \ \
	\begin{tikzpicture}
	\node (v1) at (-4,3.6) {$a$};
	\node (v2) at (-2.8,3.6) {$a$};
	\draw (v1) edge[double]  (v2);
	\node (v3) at (-1.6,3.6) {$b$};
	\node (v4) at (-4,2.4) {$a$};
	\node (v5) at (-1.6,2.4) {$b$};
	\draw[->]  (v2) edge node[above, font= \scriptsize]{$u$} (v3);
	\draw[->]   (v1) edge node[left, font= \scriptsize]{$1_a$} (v4);
	\draw[->]   (v3) edge node[right, font= \scriptsize]{$1_b$} (v5);
	\draw[->]   (v4) edge node[below, font= \scriptsize]{$u$} (v5);
	\node at (-2.7835,2.9641) {$l_u$};
	\end{tikzpicture}
	\end{center}
	\item For each vertical morphism $f \in \mathbb{D}_v(a,b)$ a 2-cell
	\begin{center}
		\begin{tikzpicture}
		
		\node (v1) at (-2.8,3.2) {$a$};
		\node (v3) at (-1.2,3.2) {$a$};
		\node (v2) at (-2.8,1.8) {$b$};
		\node (v4) at (-1.2,1.8) {$b$};
		\draw[->]  (v1) edge node[left, font= \scriptsize]{$f$} (v2);
		\draw  (v1) edge[double]  (v3);
		\draw[->]  (v3) edge node[right, font= \scriptsize]{$f$} (v4);
		\draw (v2) edge [double] (v4);
		
		\node at (-2,2.5) {$1_u$};
		\end{tikzpicture}
	\end{center}
	called the \textit{unit} of $f$.
	\item For every triple $u,v,w$ of composable horizontal morphisms an invertible 2-cell $a_{u,v,w}: (u \otimes v) \otimes w \rightarrow u \otimes (v \otimes w)$ called \textit{associator}
\end{itemize} 
The above data is subject to coherence axioms analogous to those of weak 2-categories. The axiomas guarantee that we have a (strict) vertical 2-category $\mathbb{D}_v$ with vertical morphisms as 1-morphisms and cells of the form 
\begin{center}
	\begin{tikzpicture}
	
	\node (v1) at (-2.8,3.2) {$a$};
	\node (v3) at (-1.2,3.2) {$a$};
	\node (v2) at (-2.8,1.8) {$b$};
	\node (v4) at (-1.2,1.8) {$b$};
	\draw[->]  (v1) edge node[left, font= \scriptsize]{$f$} (v2);
	\draw  (v1) edge[double]  (v3);
	\draw[->]  (v3) edge node[right, font= \scriptsize]{$g$} (v4);
	\draw (v2) edge [double] (v4);
	
	\node at (-2,2.5) {$\alpha$};
	\end{tikzpicture}
\end{center}
as 2-morphisms. Similarly the axioms guarantee the existenece of a (weak) horizontal 2-category $\mathbb{D}_h$ whose 2-morphisms are cells of the form 
\begin{center}
	\begin{tikzpicture}
	
	\node (v1) at (-2.8,3.2) {$a$};
	\node (v3) at (-1.2,3.2) {$b$};
	\node (v2) at (-2.8,1.8) {$a$};
	\node (v4) at (-1.2,1.8) {$d$};
	\draw[->]  (v1) edge node[left, font= \scriptsize]{$1_a$} (v2);
	\draw[->]  (v1) edge node[above, font= \scriptsize]{$u$} (v3);
	\draw[->]  (v3) edge node[right, font= \scriptsize]{$1_b$} (v4);
	\draw[->]  (v2) edge node[below, font= \scriptsize]{$v$} (v4);
	
	\node at (-2,2.5) {$\alpha$};
	\end{tikzpicture}
\end{center}
We refer the reader to \cite{leinster2004higher} or \cite{leinster1998basic} as we are more interested in the spirit of these structures rather then in the delicate art of defining them.

Examples of double categories are numerous (see \cite{grandis1999limits}, \cite{leinster2004higher} or \cite{shulman2008framed}) because it occurs a lot to have two kinds of morphisms related as above. For example we have homomorphisms between groups or rings and bimodules, functions between sets and relations, continuous maps between manifolds and cobordisms, functors between categories and adjunctions. 

Here we are interested in the example of categories, functors and profunctors. 
	There is a double category $\textbf{\text{Prof}}$ with
	\begin{itemize}
		\item[($\bullet$)] Categories as objects
		\item[($\downarrow$)] The category $\textbf{\text{Cat}}$ of categories and functors as vertical category
		\item[($\rightarrow$)] Profunctors $\mathcal{C}^{\text{op}} \times \mathcal{D} \rightarrow \text{\textbf{Set}}$ as horizontal morphisms
		\item[($=$)] The profunctor $\mathcal{C}(\_,\_): \mathcal{C}^{\text{op}} \times \mathcal{C} \rightarrow \text{\textbf{Set}}$ ,which assigns to a pair of objects $c,d \in \mathcal{C}$ the morphism set $\mathcal{C}(c,d)$, as horizontal unit
		\item[($\otimes$)] Tensor product of profunctors as horizontal composition
		\item[($\square$)] 2-cells 
		\begin{center}
			\begin{tikzpicture}
			
			\node (v1) at (-2.8,3.2) {$\mathcal{C}$};
			\node (v3) at (-1.2,3.2) {$\mathcal{D}$};
			\node (v2) at (-2.8,1.8) {$\mathcal{C}^\prime$};
			\node (v4) at (-1.2,1.8) {$\mathcal{D}^\prime$};
			\draw[->]  (v1) edge node[left, font= \scriptsize]{$F$} (v2);
			\draw[->]  (v1) edge node[above, font= \scriptsize]{$u$} (v3);
			\draw[->]  (v3) edge node[right, font= \scriptsize]{$G$} (v4);
			\draw[->]  (v2) edge node[below, font= \scriptsize]{$v$} (v4);
			
			\node at (-2,2.5) {$\alpha$};
			\end{tikzpicture}
		\end{center}
		given by a family of functions $\alpha_{c,d}: u(c,d) \rightarrow v(Fc,Gd)$ indexed by pairs $c \in \mathcal{C}$, $d \in \mathcal{D}$ which respect the action of the arrows in $\mathcal{C}$ and $\mathcal{D}$ in the obvious way. 
		
		Equivalently, the 2-cell $\alpha$ can be described as a functor $\alpha : \text{col}(u) \rightarrow \text{col}(v)$ between collages which respects $F$ and $G$, meaning $\alpha|_\mathcal{C} = F$ and $\alpha|_\mathcal{D}=G$.
	\end{itemize}

The double category \textbf{Prof} elegantly organizes the 2-category of categories functors and natural transformations as vertical 2-category and the weak 2-category of categories, profunctors and transformations between them as horizontal 2-category. As we will see next, we can read information from \textbf{Prof} which we cannot read from its vertical or horizontal 2-categories alone. 

Another interesting example is the double category of cospans associated to a category $\mathcal{C}$ with pushouts. Recall that a cospan in $\mathcal{C}$ is a diagram of the form $$x_0 \longrightarrow x \longleftarrow x_1$$ in $\mathcal{C}$. We may think of such a cospan as a morphism with source $x_0$ and target $x_1$. Given another cospan $(x_1 \rightarrow y \leftarrow x_2)$ we define their composite by forming the pushout of the top square as indicated in the diagram:
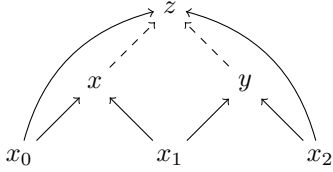
\begin{center}
	\begin{tikzpicture}
	
	\node (v1) at (-4,1.5) {$x_0$};
	\node (v2) at (-3,2.5) {$x$};
	\node (v3) at (-2,1.5) {$x_1$};
	\node (v4) at (-1,2.5) {$y$};
	\node (v5) at (0,1.5) {$x_2$};
	\node (v6) at (-2,3.5) {$z$};
	\draw[->]  (v1) edge (v2);
	\draw[->]  (v3) edge (v2);
	\draw[->]  (v3) edge (v4);
	\draw[->]  (v5) edge (v4);
	\draw[->, dashed]  (v2) edge (v6);
	\draw[->, dashed]  (v4) edge (v6);
	\draw[->, bend left]  (v1) edge (v6);
	\draw[->, bend right]  (v5) edge (v6);
	\end{tikzpicture}
\end{center}

This way we may form a double category $\text{\textbf{coSpan}}(\mathcal{C})$ with
\begin{itemize}
	\renewcommand\labelitemi{--}
	\item objects those of $\mathcal{C}$ and vertical category $\mathcal{C}$
	\item cospans in as horizontal morphisms, with composition as prescribed above
	\item 2-cells diagrams of the form 
	
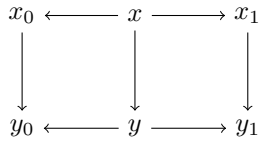
\begin{center}
		\begin{tikzpicture}

		\node (v1) at (-4.5,5) {$x_0$};
		\node (v3) at (-3,5) {$x$};
		\node (v4) at (-1.5,5) {$x_1$};
		\node (v2) at (-4.5,3.5) {$y_0$};
		\node (v5) at (-3,3.5) {$y$};
		\node (v6) at (-1.5,3.5) {$y_1$};
		\draw[->]  (v1) edge (v2);
		\draw[->]  (v3) edge (v1);
		\draw[->]  (v3) edge (v4);
		\draw[->]  (v3) edge (v5);
		\draw[->]  (v4) edge (v6);
		\draw[->]  (v5) edge (v2);
		\draw[->]  (v5) edge (v6);
		\end{tikzpicture}
	\end{center}
	with the obvious composition induced by taking pushouts
\end{itemize}

\subsubsection{Equipments, companions and cojoints}

The notion of equipment has been discussed in a few variants. In this work an \textit{equipment} is a double category $\mathbb{D}$ satisfying the following universal filling condition: given a niche 
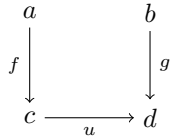
\begin{center}
	\begin{tikzpicture}
	
	\node (v1) at (-2.8,3.2) {$a$};
	\node (v3) at (-1.2,3.2) {$b$};
	\node (v2) at (-2.8,1.8) {$c$};
	\node (v4) at (-1.2,1.8) {$d$};
	\draw[->]  (v1) edge node[left, font= \scriptsize]{$f$} (v2);
	\draw[->]  (v3) edge node[right, font= \scriptsize]{$g$} (v4);
	\draw[->]  (v2) edge node[below, font= \scriptsize]{$u$} (v4);
	
	\end{tikzpicture}
\end{center}
there is a horizontal morphism $u^{f,g}: a \rightarrow b$ and a 2-cell
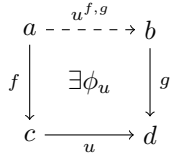
\begin{center}
	\begin{tikzpicture}
	
	\node (v1) at (-2.8,3.2) {$a$};
	\node (v3) at (-1.2,3.2) {$b$};
	\node (v2) at (-2.8,1.8) {$c$};
	\node (v4) at (-1.2,1.8) {$d$};
	\draw[->]  (v1) edge node[left, font= \scriptsize]{$f$} (v2);
	\draw[->, dashed]  (v1) edge node[above, font= \scriptsize]{$u^{f,g}$} (v3);
	\draw[->]  (v3) edge node[right, font= \scriptsize]{$g$} (v4);
	\draw[->]  (v2) edge node[below, font= \scriptsize]{$u$} (v4);
	
	\node at (-2,2.5) {$\exists \phi_u$};
	\end{tikzpicture}
\end{center}
which is universal, meaning every 2-cell $\psi$ with target $u$ whose horizontal domain and codomain factor through $f$ and $g$

\begin{center}
	\begin{tikzpicture}
				\node (v1) at (-4.2,4) {$a^\prime$};
				\node (v4) at (-2.8,4) {$b^\prime$};
				\node (v2) at (-4.2,2.8) {$a$};
				\node (v5) at (-2.8,2.8) {$b$};
				\node (v3) at (-4.2,1.6) {$c$};
				\node (v6) at (-2.8,1.6) {$d$};
				
				\node at (-3.4895,2.7999) {$\psi$};
				
		\draw[->]  (v1) edge node[left, font= \scriptsize]{$f^\prime$} (v2);
		\draw[->]  (v2) edge node[left, font= \scriptsize]{$f$} (v3);
		\draw[->]  (v1) edge node[above, font= \scriptsize]{$v$} (v4);
		\draw[->]  (v4) edge node[right, font= \scriptsize]{$g^\prime$} (v5);
		\draw[->]  (v5) edge node[right, font= \scriptsize]{$g$} (v6);
		\draw[->]  (v3) edge node[below, font= \scriptsize]{$u$} (v6);
	\end{tikzpicture}
\end{center}
 factors uniquely through $\phi_u$, meaning $\psi = \phi_u \psi^\prime$ for a unique $\psi^\prime$
\begin{center}
	\begin{tikzpicture}
	
	\node (v1) at (-4.2,4) {$a^\prime$};
	\node (v4) at (-2.8,4) {$b^\prime$};
	\node (v2) at (-4.2,2.8) {$a$};
	\node (v5) at (-2.8,2.8) {$b$};
	\node (v3) at (-4.2,1.6) {$c$};
	\node (v6) at (-2.8,1.6) {$d$};
	
	\node at (-2.2969,2.7999) {$=$};
	
	\node (v7) at (-1.8,4) {$a^\prime$};
	\node (v10) at (-0.4,4) {$b^\prime$};
	\node (v8) at (-1.8,2.8) {$a$};
	\node (v11) at (-0.4,2.8) {$b$};
	\node (v9) at (-1.8,1.6) {$c$};
	\node (v12) at (-0.4,1.6) {$d$};
	\node at (-3.4895,2.7999) {$\psi$};
	\node at (-1.076,2.1596) {$\phi_u$};
	\draw[->]  (v1) edge node[left, font= \scriptsize]{$f^\prime$} (v2);
	\draw[->]  (v2) edge node[left, font= \scriptsize]{$f$} (v3);
	\draw[->]  (v1) edge node[above, font= \scriptsize]{$v$} (v4);
	\draw[->]  (v4) edge node[right, font= \scriptsize]{$g^\prime$} (v5);
	\draw[->]  (v5) edge node[right, font= \scriptsize]{$g$} (v6);
	\draw[->]  (v3) edge node[below, font= \scriptsize]{$u$} (v6);
	\draw[->]  (v7) edge node[left, font= \scriptsize]{$f^\prime$} (v8);
	\draw[->]  (v8) edge node[left, font= \scriptsize]{$f$} (v9);
	\draw[->]  (v7) edge node[above, font= \scriptsize]{$v$} (v10);
	\draw[->]  (v10) edge node[right, font= \scriptsize]{$g^\prime$} (v11);
	\draw[->]  (v11) edge node[right, font= \scriptsize]{$g$} (v12);
	\draw[->]  (v8) edge node[above, font= \scriptsize]{$u^{f,g}$} (v11);
	\draw[->]  (v9) edge node[below, font= \scriptsize]{$u$} (v12);
	
	\node at (-1.076,3.4074) {$\exists ! \psi^\prime$};
	\end{tikzpicture}
\end{center}

\begin{proposition}
	The double category \textbf{Prof} is an equipment.
\end{proposition}
\begin{proof}
	Let $\mathcal{A}, \mathcal{B}, \mathcal{C}, \mathcal{D}$ be categories, $F: \mathcal{A} \rightarrow \mathcal{C}$ and $G: \mathcal{B} \rightarrow \mathcal{D}$ be functors and $u: \mathcal{C}^{\text{op}} \times \mathcal{D} \rightarrow \text{\textbf{Set}}$ be a profunctor so that we have a niche
	\begin{center}
		\begin{tikzpicture}
		
		\node (v1) at (-2.8,3.2) {$\mathcal{A}$};
		\node (v3) at (-1.2,3.2) {$\mathcal{B}$};
		\node (v2) at (-2.8,1.8) {$\mathcal{C}$};
		\node (v4) at (-1.2,1.8) {$\mathcal{D}$};
		\draw[->]  (v1) edge node[left, font= \scriptsize]{$F$} (v2);
		\draw[->]  (v3) edge node[right, font= \scriptsize]{$G$} (v4);
		\draw[->]  (v2) edge node[below, font= \scriptsize]{$u$} (v4);
		
		\end{tikzpicture}
	\end{center}
	Define the profunctor $u^{F,G}: \mathcal{A}^{\text{op}} \times \mathcal{B} \rightarrow \text{\textbf{Set}}$ via $$u^{F,G}(a,b)= u(Fa,Gb)$$ for all $a \in \mathcal{A}$, $b \in \mathcal{B}$ with action of the arrows defined in the obvious way, meaning that for $f: a^\prime \rightarrow a$ in $\mathcal{A}$ and $g : b \rightarrow b^\prime$ the induced map will be just $$u^{F,G}(f, g) = u(Ff,Gg): u(Fa,Gb) \rightarrow u(Fa^\prime, Gb^\prime) $$
	 $\phi_u$ will be given by the trivial components $\phi_u(a,b) = id : u^{F,G}(a,b) \rightarrow u(Fa,Gb)$. Verifying the universal property is an easy exercise and we leave it to the reader. 
\end{proof}

Let $\mathbb{D}$ be an equipment, $A,B \in \mathbb{D}$ two objects and $f: A \rightarrow B$ a vertical morphism. There are two distinguished horizontal morphisms we denote $f^*: A \nrightarrow B$ and $f_* : B \nrightarrow A$ obtained by filling the following niches: 

\begin{center}
	\begin{tikzpicture}

\node (v1) at (-2.8,3.2) {$A$};
\node (v3) at (-1.2,3.2) {$B$};
\node (v2) at (-2.8,1.8) {$B$};
\node (v4) at (-1.2,1.8) {$B$};
\draw[->]  (v1) edge node[left, font= \scriptsize]{$f$} (v2);
\draw[->, dashed]  (v1) edge node[above, font= \scriptsize]{$f^*$} (v3);
\draw[->]  (v3) edge node[right, font= \scriptsize]{$1_B$} (v4);
\draw (v2) edge [double]  (v4);

\end{tikzpicture} \ \ \ \ \ 
\begin{tikzpicture}

\node (v1) at (-2.8,3.2) {$B$};
\node (v3) at (-1.2,3.2) {$A$};
\node (v2) at (-2.8,1.8) {$B$};
\node (v4) at (-1.2,1.8) {$B$};
\draw[->]  (v1) edge node[left, font= \scriptsize]{$1_B$} (v2);
\draw[->, dashed]  (v1) edge node[above, font= \scriptsize]{$f_*$} (v3);
\draw[->]  (v3) edge node[right, font= \scriptsize]{$f$} (v4);
\draw (v2) edge [double]  (v4);

\end{tikzpicture}
\end{center}
$f^*$ is called the \textit{companion} of $f$ and $f_*$ the \textit{cojoint} of $f$. 

If $F: \mathcal{C} \rightarrow \mathcal{D}$ is a functor between categories then its companion  in \textbf{Prof} is given by $$F^*(c,d) = \mathcal{D}(Fc,d)$$
and its cojoint by $$F_*(c,d) = \mathcal{D}(d,Fc)$$ 
It is not our intention to give a full exposition of this topic. The main point we want to bring home is that the companion of a functor can be understood formally as satisfying a universal property in \textbf{Prof}! We will mention a few facts.

Comapanions and cojoints can be formulated independent of the equipment property (see \cite{shulman2008framed}). Then a double category with companions and cojoints satisfies the equipment property, because given a niche 
		\begin{center}
		\begin{tikzpicture}
		
		\node (v1) at (-2.8,3.2) {$A$};
		\node (v3) at (-1.2,3.2) {$B$};
		\node (v2) at (-2.8,1.8) {$C$};
		\node (v4) at (-1.2,1.8) {$D$};
		\draw[->]  (v1) edge node[left, font= \scriptsize]{$f$} (v2);
		\draw[->]  (v3) edge node[right, font= \scriptsize]{$g$} (v4);
		\draw[->]  (v2) edge node[below, font= \scriptsize]{$u$} (v4);
		
		\end{tikzpicture}
	\end{center}
	we have the filler $u^{F,G} = g^* \otimes u \otimes f_*$ and it can be shown to be universal. The moral is that for a double category, satisfying the equipment property and having companions and cojoints are equivalent.

The companion and cojoint constructions are well-behaved in terms of functoriality. However we should note that we do not have $(gf)^* = g^*\otimes f^*$ as a strict equality but we have an isomorphism $$(gf)^* \cong g^*\otimes f^*$$
which can easily be seen as follows: for composable maps $A \xrightarrow{f} B \xrightarrow{g} C$ consider the niche 
\begin{center}
	\begin{tikzpicture}
	
	\node (v1) at (-2.8,3.2) {$A$};
	\node (v3) at (-1.2,3.2) {$C$};
	\node (v2) at (-2.8,1.8) {$C$};
	\node (v4) at (-1.2,1.8) {$C$};
	\draw[->]  (v1) edge node[left, font= \scriptsize]{$gf$} (v2);
	\draw[->]  (v3) edge node[right, font= \scriptsize]{$1$} (v4);
	\draw (v2) edge [double]  (v4);

	\end{tikzpicture}
\end{center}
and use the equipment property to fill it in two steps
\begin{center}
	\begin{tikzpicture}
	\node (v1) at (-4.2,4) {$A$};
	\node (v4) at (-2.8,4) {$C$};
	\node (v2) at (-4.2,2.8) {$B$};
	\node (v5) at (-2.8,2.8) {$C$};
	\node (v3) at (-4.2,1.6) {$C$};
	\node (v6) at (-2.8,1.6) {$C$};
	
	\draw[->]  (v1) edge node[left, font= \scriptsize]{$f$} (v2);
	\draw[->]  (v2) edge node[left, font= \scriptsize]{$g$} (v3);
	\draw[->, dashed]  (v1) edge node[above, font= \scriptsize]{$g^* \otimes f^*$} (v4);
	\draw[->, dashed]  (v2) edge node[above, font= \scriptsize]{$g^*$} (v5);
	\draw[->]  (v4) edge node[right, font= \scriptsize]{$1_C$} (v5);
	\draw[->]  (v5) edge node[right, font= \scriptsize]{$1_C$} (v6);
	\draw  (v3) edge [double] (v6);
	\end{tikzpicture}
\end{center}
by first filling the bottom niche and then the upper one. By the above  the filler will be $g^* \otimes f^*$. But composition of universal squares is universal and hence the desired isomorphism. The same discussion is valid for cojoints as well. So these constructions define weak functors 
$$(\cdot)^* : \mathbb{D}_0 \rightarrow \mathbb{D}_h$$
and
$$(\cdot)_* : \mathbb{D}_0 \rightarrow \mathbb{D}_h$$
from the vertical category $\mathbb{D}_0$ of $\mathbb{D}$ to its horizontal 2-category.

Lastly it is worth discussing the dual of the equipment property postulated above. We can ask of a double category $\mathbb{D}$ to have a universal filler $u_{f,g}$ for any configuration of solid arrows $u,f,g$ as in the diagram:
\begin{center}
	\begin{tikzpicture}
	
	\node (v1) at (-3.5,3.5) {$A$};
	\node (v3) at (-2,3.5) {$B$};
	\node (v2) at (-3.5,2) {$C$};
	\node (v4) at (-2,2) {$D$};
	\draw[->]  (v1) edge node[left, font= \scriptsize]{$f$} (v2);
	\draw[->]  (v1) edge node[above, font= \scriptsize]{$u$} (v3);
	\draw[->, dashed]  (v2) edge node[below, font= \scriptsize]{$u_{f,g}$} (v4);
	\draw[->]  (v3) edge node[right, font= \scriptsize]{$g$} (v4);
	\end{tikzpicture}
\end{center}
Luckily for us the equipment property postulated above and its dual are equivalent for a double category. This is the content of \cite[Theorem 4.1]{shulman2008framed}. Roughly speaking this is because the dual equipment property allows us to define companions and cojoints (in the obvious dual manner) and hence the equivalence.

\subsubsection{Colimits in double categories}

If $\mathcal{C}$ is a category, $\mathcal{J}$ a small category and $$F : \mathcal{J} \rightarrow \mathcal{C}$$ a diagram then the colimit of this diagram is an object $\text{colim} D \in \mathcal{C}$ which represents $D$. More precisely, $\text{colim} F$ is equipped with a natural transformation to the constant functor $F \Rightarrow \text{colim} F$ such that given another object $c \in \mathcal{C}$ and a natural transformation $F \Rightarrow c$ there is a unique map $\text{colim} F \rightarrow c$ such that the triangle
\begin{center}
	\begin{tikzpicture}

\node (v1) at (-5,5) {$F$};
\node (v2) at (-5,3.5) {$\text{colim} F$};
\node (v3) at (-3,3.5) {$c$};
\draw[->]  (v1) edge[double] (v2);
\draw[->]  (v1) edge[double] (v3);
\draw[->, dashed]  (v2) edge node[below, font = \scriptsize]{$\exists !$} (v3);
\end{tikzpicture}
\end{center}
commutes. It is common practice to depict the universal property as
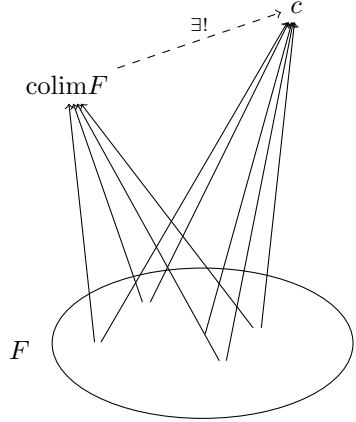
\begin{center}
	\begin{tikzpicture}

	\draw  (-3.5,4.5) node (v7) {} ellipse (2 and 1);
	\node (v3) at (-4.2564,4.9229) {};
	\node (v4) at (-3.2089,4.1419) {};
	\node (v5) at (-2.7328,4.5801) {};
	\node (v1) at (-4.9229,4.3897) {};
	\node (v2) at (-5.3038,7.932) {$\text{colim}F$};
	\node (v6) at (-2.2567,8.9698) {$c$};
	\draw[->]  (v1) edge (v2);
	\draw[->]  (v3) edge (v2);
	\draw[->]  (v4) edge (v2);
	\draw[->]  (v5) edge (v2);
	\draw[->]  (v1) edge (v6);
	\draw[->]  (v3) edge (v6);
	\draw[->]  (v7) edge (v6);
	\draw[->]  (v4) edge (v6);
	\draw[->]  (v5) edge (v6);
	\draw[->, dashed]  (v2) edge node[above, font = \scriptsize]{$\exists !$} (v6);
	
	\node at (-5.9323,4.4086) {$F$};
	
	\end{tikzpicture}
\end{center}

In the context of double categories we are interested in the case where the digram $F$ is laying in the horizontal direction and the rest of the maps are vertical, and we may replace the corresponding commutative triangles with cells. We obtain this way the notion of double colimit.

Let $\mathbb{D}$ be a double category and 
$$F : \mathcal{J} \rightarrow \mathbb{D}_h$$ be a diagram in the horizontal 2-category of $\mathbb{D}$. Recall that functors whose target is a weak 2-category assign to composable pairs $\xrightarrow{f} \xrightarrow{g}$ an invertible 2-cell $\alpha_{f,g}: F(g)F(f) \Rightarrow F(gf)$. 

\begin{center}
	
\begin{tikzpicture}
\node (v1) at (-4,3) {$i$};
\node (v2) at (-3,4.5) {$j$};
\node (v3) at (-2,3) {$k$};
\draw[->]  (v1) edge node[left, font = \scriptsize]{$f$} (v2);
\draw[->]  (v2) edge node[right, font = \scriptsize]{$g$} (v3);
\draw[->]  (v1) edge node[below, font = \scriptsize]{$gf$} (v3);

\node (v4) at (-1.2,3.8) {};
\node (v5) at (0.4,3.8) {};

\draw[|->]  (v4) edge node[above, font = \scriptsize]{$F$} (v5);
\node (v6) at (1.5,4.5) {$F(i)$};
\node (v8) at (3.5,4.5) {$F(j)$};
\node (v9) at (5.5,4.5) {$F(k)$};
\node (v7) at (1.5,3) {$F(i)$};
\node (v10) at (5.5,3) {$F(k)$};
\draw[->]  (v6) edge node[left, font = \scriptsize]{$1$} (v7);
\draw[->]  (v6) edge node[above, font = \scriptsize]{$Ff$} (v8);
\draw[->]  (v8) edge node[above, font = \scriptsize]{$Fg$} (v9);
\draw[->]  (v9) edge node[right, font = \scriptsize]{$1$} (v10);
\draw[->]  (v7) edge node[below, font = \scriptsize]{$F(gf)$} (v10);

\node at (3.5,3.7) {$\alpha_{f,g}$};
\end{tikzpicture}
\end{center}
For emphasis we will denote such weak functors $(F,\alpha)$.
\begin{definition}
	Let $(F,\alpha)$, $(G,\beta)$ be two horizontal diagrams in a double category $\mathbb{D}$ indexed by a category $\mathcal{J}$. A vertical transformation $\tau: F \Rightarrow G$ consists of:
	\begin{itemize}
		\renewcommand\labelitemi{--}
			\item  A vertical morphism $\tau_i : F(i) \rightarrow G(i)$ or each $i \in \mathcal{J}$
			
			\item A 2-cell
			\begin{center}
				\begin{tikzpicture}
				
				\node (v1) at (-2.8,3.2) {$F(i)$};
				\node (v3) at (-1.2,3.2) {$F(j)$};
				\node (v2) at (-2.8,1.8) {$G(i)$};
				\node (v4) at (-1.2,1.8) {$G(j)$};
				\draw[->]  (v1) edge node[left, font= \scriptsize]{$\tau_i$} (v2);
				\draw[->]  (v1) edge node[above, font= \scriptsize]{$Ff$} (v3);
				\draw[->]  (v3) edge node[right, font= \scriptsize]{$\tau_j$} (v4);
				\draw[->] (v2) edge node[below, font= \scriptsize]{$Gf$} (v4);
				
				\node at (-2,2.5) {$\tau_f$};
				\end{tikzpicture}
			\end{center}
			for each morphism $f: i \rightarrow j$ in $\mathcal{J}$ 
	\end{itemize}
such that everything commutes, meaning for all composable pairs $i \xrightarrow{f} j \xrightarrow{g} k$ we have 
\begin{center}
	\begin{tikzpicture}

	\node (v3) at (-4.5,4) {$F(i)$};
	\node (v4) at (-3,4) {$F(j)$};
	\node (v5) at (-1.5,4) {$F(k)$};
	\node (v6) at (-4.5,2.5) {$G(i)$};
	\node (v7) at (-3,2.5) {$G(j)$};
	\node (v8) at (-1.5,2.5) {$G(k)$};
	\node (v1) at (-4.5,5.5) {$F(i)$};
	\node (v2) at (-1.5,5.5) {$F(k)$};
	\node (v9) at (-4.5,1) {$G(i)$};
	\node (v10) at (-1.5,1) {$G(k)$};
	\node at (-3,4.8) {$\alpha_{f,g}^{-1}$};
	\node at (-3.7,3.2) {$\tau_f$};
	\node at (-2.2,3.2) {$\tau_g$};
	\node at (-3,1.6) {$\beta_{f,g}$};
	\node at (-0.7,3.2) {$=$};
	\node (v11) at (0,4) {$F(i)$};
	\node (v12) at (0,2.5) {$ G(i)$};
	\node (v13) at (1.5,4) {$F(k)$};
	\node (v14) at (1.5,2.5) {$G(k)$};
	\node at (0.8,3.2) {$\tau_{gf}$};
	\draw[->]  (v1) edge node[above, font= \scriptsize]{$F(gf)$} (v2);
	\draw[->]  (v1) edge node[left, font= \scriptsize]{$1$} (v3);
	\draw[->]  (v3) edge node[above, font= \scriptsize]{$Ff$} (v4);
	\draw[->]  (v4) edge node[above, font= \scriptsize]{$Fg$} (v5);
	\draw[->]  (v2) edge node[right, font= \scriptsize]{$1$} (v5);
	\draw[->]  (v3) edge node[left, font= \scriptsize]{$\tau_i$} (v6);
	\draw[->]  (v4) edge node[left, font= \scriptsize]{$\tau_j$} (v7);
	\draw[->]  (v5) edge node[left, font= \scriptsize]{$\tau_k$} (v8);
	\draw[->]  (v6) edge node[above, font= \scriptsize]{$Gf$} (v7);
	\draw[->]  (v7) edge node[above, font= \scriptsize]{$Gg$} (v8);
	\draw[->]  (v6) edge node[left, font= \scriptsize]{$1$} (v9);
	\draw[->]  (v8) edge node[left, font= \scriptsize]{$1$} (v10);
	\draw[->]  (v9) edge node[below, font= \scriptsize]{$G(gf)$} (v10);
	\draw[->]  (v11) edge  (v12);
	\draw[->]  (v11) edge  (v13);
	\draw[->]  (v13) edge  (v14);
	\draw[->]  (v12) edge  (v14);
	\end{tikzpicture}
\end{center}
or equationally $$\beta_{f,g}(\tau_g * \tau_f)\alpha_{f,g}^{-1} = \tau_{gf} $$
\end{definition}

For an object $c \in \mathbb{D}$ we have a constant diagram, which by abuse we denote $c : \mathcal{J} \rightarrow \mathbb{D}_h$ which send all objects to $c$, maps to $1_c$ and with structure isomorphisms the unitors of $\mathbb{D}_h$.

Now we define the double colimit of a horizontal diagram $F$ to be an object $\text{dcolim} F \in \mathbb{D}$ equipped with a vertical transformation $F \Rightarrow \text{dcolim} F$ which is universal, meaning any other vertical transformation to some object $F \Rightarrow c$ factors through a unique vertical arrow $\text{dcolim} D \rightarrow c$. 
\begin{center}
	\begin{tikzpicture}
	
	\node (v1) at (-5,5) {$F$};
	\node (v2) at (-5,3.5) {$\text{dcolim}F$};
	\node (v3) at (-3,3.5) {$c$};
	\draw[->]  (v1) edge[double] (v2);
	\draw[->]  (v1) edge[double] (v3);
	\draw[->, dashed]  (v2) edge node[below, font = \scriptsize]{$\exists !$} (v3);
	\end{tikzpicture}
\end{center}

\begin{remark}
	More general colimits are discussed in \cite{grandis1999limits}, in which $\mathcal{J}$ is allowed to be a double category. For our purposes it is enough to consider double colimits of horizontal diagrams.
\end{remark}

We have two important examples in \textbf{Prof}. First let $\Delta^1 = (0 \rightarrow 1)$ be the interval category and let the functor 
$$F: \Delta^1 \rightarrow \text{\textbf{Prof}}_h$$
 pick a profunctor $u : \mathcal{C}^{op} \times \mathcal{D} \rightarrow
  \text{\textbf{Set}}$. Then $$\text{dcolim}F = \text{col}(u)$$
So besides comapnion and cojoints, the collage construction of a profunctor satisfies a universal property in the double category \textbf{Prof}, which if unpacked says that the colloage $\text{col}(u)$ is equipped with a 2-cell $\phi_u$
\begin{center}
	\begin{tikzpicture}
	
	\node (v1) at (-2.8,3.2) {$\mathcal{C}$};
	\node (v3) at (-1.2,3.2) {$\mathcal{D}$};
	\node (v2) at (-2.8,1.8) {$\text{col}(u)$};
	\node (v4) at (-1.2,1.8) {$\text{col}(u)$};
	\draw[left hook->]  (v1) edge  (v2);
	\draw[->]  (v1) edge node[above, font= \scriptsize]{$u$} (v3);
	\draw[left hook->]  (v3) edge  (v4);
	\draw (v2) edge [double]  (v4);
	
	\node at (-2,2.5) {$\phi$};
	\end{tikzpicture}
\end{center}
 such that any other 2-cell of the form 
\begin{center}
	\begin{tikzpicture}
	
	\node (v1) at (-2.8,3.2) {$\mathcal{C}$};
	\node (v3) at (-1.2,3.2) {$\mathcal{D}$};
	\node (v2) at (-2.8,1.8) {$\mathcal{E}$};
	\node (v4) at (-1.2,1.8) {$\mathcal{E}$};
	\draw[->]  (v1) edge node[left, font= \scriptsize]{$F$} (v2);
	\draw[->]  (v1) edge node[above, font= \scriptsize]{$u$} (v3);
	\draw[->]  (v3) edge node[right, font= \scriptsize]{$G$} (v4);
	\draw (v2) edge [double]  (v4);
	
	\node at (-2,2.5) {$\psi$};
	\end{tikzpicture}
\end{center}
for some category $\mathcal{E}$ and functors $F,G$ factors uniquely through $\phi_u$, meaning there is a functor $H: \text{col}(u) \rightarrow \mathcal{E}$ such that 
\begin{center}
	\begin{tikzpicture}
	
	\node (v1) at (-2.8,3.2) {$\mathcal{C}$};
	\node (v3) at (-1.2,3.2) {$\mathcal{D}$};
	\node (v2) at (-2.8,1.8) {$\mathcal{E}$};
	\node (v4) at (-1.2,1.8) {$\mathcal{E}$};
	\draw[->]  (v1) edge node[left, font= \scriptsize]{$F$} (v2);
	\draw[->]  (v1) edge node[above, font= \scriptsize]{$u$} (v3);
	\draw[->]  (v3) edge node[right, font= \scriptsize]{$G$} (v4);
	\draw (v2) edge [double]  (v4);
	
	\node at (-2,2.5) {$\psi$};
	
\node at (-0.4,2.5) {$=$};
\node (v5) at (0.5,4) {$\mathcal{C}$};
\node (v7) at (2,4) {$\mathcal{D}$};
\node (v6) at (0.5,2.5) {$\text{col}(u)$};
\node (v8) at (2,2.5) {$\text{col}(u)$};
\node (v9) at (0.5,1) {$\mathcal{E}$};
\node (v10) at (2,1) {$\mathcal{E}$};
\draw[->]  (v5) edge node[left, font= \scriptsize]{$i$} (v6);
\draw[->]  (v5) edge node[above, font= \scriptsize]{$u$}(v7);
\draw[->]  (v7) edge node[right, font= \scriptsize]{$j$}(v8);
\draw (v6) edge [double](v8);
\draw[->]  (v6) edge node[left, font= \scriptsize]{$H$} (v9);
\draw[->]  (v8) edge node[right, font= \scriptsize]{$H$} (v10);
\draw  (v9) edge [double] (v10);
\node at (1.2,3.2) {$\phi_u$};
\node at (1.2,1.7) {$1_H$};
	
	\end{tikzpicture}
\end{center}
Here we take the components $\phi_u(c,d)$ to be simply the identity map $u(c,d) \rightarrow \text{col}(u)(c,d)$ and it is straightforward to verify the universality of $\phi_u$.  

We conclude this discussion with our second example of double colimits which is the Grothendieck construction for categories. We present this in the form of a theorem.

\grothdcolim

\begin{proof}
	We prove the theorem directly by verifying the universal property of double colimits. By definition of the Grothendieck construction we have inclusions $\iota_i : \mathcal{C}_i \rightarrow \textbf{Gro}(F)$ for $i \in \mathcal{J}$ given by identity on objects and on maps $(a \xrightarrow{f} b) \in \mathcal{C}_i$ by $\iota_i(f) = (1_{\mathcal{C}_i}, f)$. 
	
	Next, for all functors $F_\alpha : \mathcal{C}_i \rightarrow \mathcal{C}_j$ in the image of $F$ we have a 2-cell 
	\begin{center}
		\begin{tikzpicture}
		
		\node (v1) at (-3.4,3.2) {$\mathcal{C}_i$};
		\node (v3) at (-0.8,3.2) {$\mathcal{C}_j$};
		\node (v2) at (-3.4,1.8) {$\textbf{Gro}(F)$};
		\node (v4) at (-0.8,1.8) {$\textbf{Gro}(F)$};
		\draw[->]  (v1) edge node[left, font= \scriptsize]{$\iota_i$} (v2);
		\draw[->]  (v1) edge node[above, font= \scriptsize]{$F_\alpha^*$} (v3);
		\draw[->]  (v3) edge node[right, font= \scriptsize]{$\iota_j$} (v4);
		\draw (v2) edge [double]  (v4);
		
		\node at (-2,2.5) {$\phi_\alpha$};
		\end{tikzpicture}
	\end{center}
with components for a pair of objects $x \in \mathcal{C}_i$, $y \in \mathcal{C}_j$ given by $\phi_\alpha(x,y)(f) = (\alpha,f)$ where $f: F_\alpha x \rightarrow y$ is a map in $\mathcal{C}_j$.
The above assemble to a vertical transformation 
$$\phi: F^* \Rightarrow \textbf{Gro}(F)$$ 

To verify universality assume we have a category $\mathcal{D}$ and a vertical transformation $\psi: F^* \Rightarrow \mathcal{D}$ in \textbf{Prof}. Let us denote the components of $\psi$ by $\psi_i : \mathcal{C}_i \rightarrow \mathcal{D}$ and 
\begin{center}
	\begin{tikzpicture}
	
	\node (v1) at (-3.4,3.2) {$\mathcal{C}_i$};
	\node (v3) at (-0.8,3.2) {$\mathcal{C}_j$};
	\node (v2) at (-3.4,1.8) {$\mathcal{D}$};
	\node (v4) at (-0.8,1.8) {$\mathcal{D}$};
	\draw[->]  (v1) edge node[left, font= \scriptsize]{$\psi_i$} (v2);
	\draw[->]  (v1) edge node[above, font= \scriptsize]{$F_\alpha^*$} (v3);
	\draw[->]  (v3) edge node[right, font= \scriptsize]{$\psi_j$} (v4);
	\draw (v2) edge [double]  (v4);
	
	\node at (-2,2.5) {$\psi_\alpha$};
	\end{tikzpicture}
\end{center}
for a morphism $\alpha: i \rightarrow j$ in $\mathcal{J}$.

 We have an induced functor
$$G: \textbf{Gro}(F) \rightarrow \mathcal{D}$$ 
which on objects is given by $G(x) = \psi_i(x)$ for $x\in \mathcal{C}_i$. For a morphism $(\alpha,f) : x \rightarrow y$ in $\textbf{Gro}(F)$, $x\in \mathcal{C}_i$, $y\in \mathcal{C}_j$, we define 
$$G(\alpha,f) = \psi_\alpha(f)$$
It is easy to see that $G \circ \phi = \psi$.

\end{proof}

\begin{note}
It could be possible to prove the above theorem using the properties of companions and the already known fact that the Grothendieck construction is a two dimensional lax colimit (see \cite{gepner1501lax}). However we find this treatment simpler and more meaningful. What the theorem is telling us is that the Grothendieck construction can be regarded as a collage of the whole diagram, and this is formally obtained by first transposing the diagram in the realm of profunctors via the companion functor and then taking its double colimit. To the best of the author's knowledge the Grothendieck construction is not presented in this fashion anywhere in the literature.
\end{note}

\section{Higher equipments}

Our objects of interest will be (possibly weak) simplicial categories
$$\mathbb{E} : \Delta^{op} \rightarrow \text{\textbf{Cat}}$$ 
In this paper the term simplicial category means what it is supposed to: a simplicial object in the category of categories. It is our intention in this section to show that: 
\begin{enumerate}
	\item Simplicial categories are categorical structures in themselves. We will treat them as 2-fold structures and explain how we can implement some of the categorical concepts coming from double category theory. In particular we propose a definition of equipment.
	\item The double category theory we develop serves as an organizing principle in homotopy theory. Notions such as higher cylinders, the mapping cylinder of a simplex and the homotopy colimit of a diagram of spaces are naturally captured as double colimits.
	\item They unify double category theory and the theory of simplicialy enriched categories.
\end{enumerate}

\subsection{Simplicial categories}
\label{scat}

Before postulating the equipment property we comment a little on simplicial categories themselves. If we think of categories as simplicial sets then a simplicial category is some kind of bisimplicial set. The reader who is familiar with the latter may immediately understand the sense in which such objects are 2-fold structures: a bisimplicial set has vertical and horizontal simplices tied together via bisimplices. However, since we are dealing with simplicial categories and not general bisimplicial sets the geometry is easier to grasp.

Let $$\mathbb{E}: \Delta^{op} \rightarrow \text{\textbf{Cat}}$$ be a simplicial category. The functor $\mathbb{E}$ assigns to each ordinal $[n] \in \Delta$ a category we denote $\mathbb{E}_n$, and specifies face and degeneracy functors $d_i$ and $s_i$. We will think of the objects of $\mathbb{E}_n$ as $n$-simplicies lying in the horizontal direction. For example an object $x \in \mathbb{E}_2$ may be depicted as 
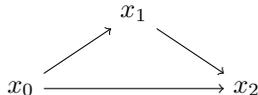
\begin{center}
	\begin{tikzpicture}

	\node (v1) at (-4,2.5) {$x_0$};
	\node (v2) at (-2.5,3.5) {$x_1$};
	\node (v3) at (-1,2.5) {$x_2$};
	\draw[->]  (v1) edge (v2);
	\draw[->]  (v2) edge (v3);
	\draw[->]  (v1) edge (v3);
	\end{tikzpicture}
\end{center}
This way a morphism in $\mathbb{E}_n$ may be thought of as a prism (bisimplex). For example a map $f: x \rightarrow y$ in $\mathbb{E}_2$ is drawn as 
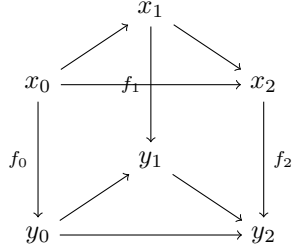
\begin{center}
	\begin{tikzpicture}

	\node (v1) at (-4,2.5) {$x_0$};
	\node (v2) at (-2.5,3.5) {$x_1$};
	\node (v3) at (-1,2.5) {$x_2$};
	\node (v4) at (-4,0.5) {$y_0$};
	\node (v5) at (-2.5,1.5) {$y_1$};
	\node (v6) at (-1,0.5) {$y_2$};
	\draw[->]  (v1) edge (v2);
	\draw[->]  (v2) edge (v3);
	\draw[->]  (v1) edge (v3);
	\draw[->]  (v1) edge node[left, font= \scriptsize]{$f_0$} (v4);
	\draw[->]  (v2) edge node[left, font= \scriptsize]{$f_1$} (v5);
	\draw[->]  (v3) edge node[right, font= \scriptsize]{$f_2$} (v6);
	\draw[->]  (v4) edge (v5);
	\draw[->]  (v5) edge (v6);
	\draw[->]  (v4) edge (v6);
	\end{tikzpicture}
\end{center}
We may think of composition in $\mathbb{E}_n$ as being an operation which produces a prism when we have two prisms on top of each other. This will be the vertical composition for this structure. 

Notice in particular that maps in $\mathbb{E}_1$ will look just like 2-cells in double categories
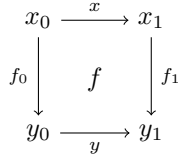
\begin{center}
	\begin{tikzpicture}
	
	\node (v1) at (-3.5,3.5) {$x_0$};
	\node (v3) at (-2,3.5) {$x_1$};
	\node (v2) at (-3.5,2) {$y_0$};
	\node (v4) at (-2,2) {$y_1$};
	\draw[->]  (v1) edge node[left, font= \scriptsize]{$f_0$} (v2);
	\draw[->]  (v1) edge node[above, font= \scriptsize]{$x$} (v3);
	\draw[->]  (v2) edge node[below, font= \scriptsize]{$y$} (v4);
	\draw[->]  (v3) edge node[right, font= \scriptsize]{$f_1$} (v4);
	
	\node at (-2.7505,2.7145) {$f$};
	\end{tikzpicture}
\end{center}
except we may compose them vertically but not horizontally. In the horizontal direction instead of composition we have a simplicial structure and we will take advantage of that. After all, categories are just special simplicial sets. 

Now we turn to examples of the sort we are have in mind. We will construct examples of simplicial categories $\mathbb{E}$ with $\mathbb{E}_0 = \textbf{Cat}, \textbf{sSet}, \textbf{Top}$. These will be named $\textbf{Cat}^\sharp, \textbf{sSet}^\sharp$ and $\textbf{Top}^\sharp$ respectively. We also construct for a category $\mathcal{C}$ a simplicial category $\textbf{coSpan}(\mathcal{C})^\sharp$ analogous to the double category of cospans associated to $\mathcal{C}$. 

Inspired by profunctors and their combinatorial interpretation we define the objects of $\textbf{Cat}^\sharp_n$ to be pairs $(\mathcal{C}, p)$ where $\mathcal{C}$ is a category and $p$ is a functor 
$$p: \mathcal{C} \rightarrow \Delta^n$$
A morphism $(\mathcal{C}, p) \rightarrow (\mathcal{D}, q)$ in $\textbf{Cat}^\sharp_n$ is a commutative triangle
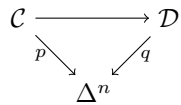
\begin{center}
	\begin{tikzpicture}

	\node (v1) at (-3,3) {$\mathcal{C}$};
	\node (v3) at (-1,3) {$\mathcal{D}$};
	\node (v2) at (-2,2) {$\Delta^n$};
	\draw[->]   (v1) edge node[left, font= \scriptsize]{$p$} (v2);
	\draw[->]   (v3) edge node[right, font= \scriptsize]{$q$} (v2);
	\draw[->]   (v1) edge (v3);
	\end{tikzpicture}
\end{center}
In other words this is the slice category so we just write
$$\textbf{Cat}^\sharp_n = \textbf{Cat}/ \Delta^n$$
We abuse and write $\mathcal{C} \in \textbf{Cat}^\sharp_n$ and leave $p$ understood. 

We define faces of $\mathcal{C}$ by pulling back along the face inclusions $\Delta^{n-1} \hookrightarrow \Delta^n$ so that the $i$-th face is given by the pullback square
\begin{center}
	\begin{tikzpicture}
	
	\node (v4) at (-3.5,4) {$d_i\mathcal{C}$};
	\node (v1) at (-2,4) {$\mathcal{C}$};
	\node (v3) at (-3.5,2.5) {$\Delta^{n-1}$};
	\node (v2) at (-2,2.5) {$\Delta^n$};
	\draw[->]  (v1) edge node[right, font= \scriptsize]{$p$} (v2);
	\draw[right hook->]  (v3) edge node[above, font= \scriptsize]{$d^i$} (v2);
	\draw[->, dashed]  (v4) edge (v3);
	\draw[->, dashed]  (v4) edge (v1);
	\end{tikzpicture}
\end{center}
Degeneracies are again defined by pulling back, as indicated in the diagram
\begin{center}
	\begin{tikzpicture}
	
	\node (v4) at (-5,4) {$s_i\mathcal{C}$};
	\node (v1) at (-2,4) {$\mathcal{C} \times \Delta^1$};
	\node (v3) at (-5,2.5) {$\Delta^{n+1}$};
	\node (v2) at (-2,2.5) {$\Delta^n \times \Delta^1$};
	\draw[->]  (v1) edge node[right, font= \scriptsize]{$p \times \Delta^1$} (v2);
	\draw[right hook->]  (v3) edge node[above, font= \scriptsize]{$\iota_i$} (v2);
	\draw[->, dashed]  (v4) edge (v3);
	\draw[->, dashed]  (v4) edge (v1);
	\end{tikzpicture}
\end{center}
where the map $\iota_i$ is given by components $s^i: \Delta^{n+1} \rightarrow \Delta^n$ and $\chi_{>i} : \Delta^{n+1} \rightarrow \Delta^1$, the latter being the characteristic function of the subset $\{i+1, \dots, n\} \subseteq [n]$
$$\chi_{>i}(j) = \begin{cases}
0  &\text{if} \ \ j \leq i \\
1  &\text{if} \ \ j>i 
\end{cases}$$

The simplicial identities are not obvious from the above description of faces and degeneracies. They will once we rephraze them in terms of collages. 

\begin{definition}
	An $n$-collage is an $(n+2)$-tuples of categories $(\mathcal{C}_0, ..., \mathcal{C}_n , \mathcal{C})$ such that  $\mathcal{C}$ is a category satisfying:
	\begin{itemize}
		\item[i)] $\text{ob}(\mathcal{C}) = \coprod_{i=1}^n \text{ob}(\mathcal{C}_i)$
		\item[ii)] $\mathcal{C}(a,b) = \mathcal{C}_i(a,b)$ if $a,b \in \mathcal{C}_i$
		\item[iii)] $\mathcal{C}(a,b) = \emptyset$ if $a \in \mathcal{C}_i$, $b \in \mathcal{C}_j$ and $i>j$
	\end{itemize}
\end{definition} 

In other words an $n$-collage $\mathcal{C}$ is obtained by adding new arrows from the objects of $\mathcal{C}_i$ to the objects of $\mathcal{C}_j$ for $i<j$. We saw an example of a $2$-collage when discussing the composition of profunctors. Let a morphism of $n$-collages $F: \mathcal{C} \rightarrow \mathcal{D}$ be a functor with $F(a) \in \mathcal{D}_i$ if $a \in \mathcal{C}_i$ and denote $F|_{\mathcal{C}_i} = F_i$. We obtain a category $\textbf{col}_n$. 

It is easy to see that we have an isomorphism of categories
$$\textbf{col}_n \cong \textbf{Cat}/\Delta^n$$ 
Indeed given a functor $p: \mathcal{C} \rightarrow \Delta^n$ we let $\mathcal{C}_i = p^{-1}(i)$ and we see that $\mathcal{C}$ satisfies the above properties. We have chosen to distinguish the notion of $n$-collage for heuristic purposes.  

Having the above identification in mind we may think of the $i$-th face of $\mathcal{C}$ as obtained by deleting the objects of $\mathcal{C}_i$ 
$$d_i\mathcal{C} = \mathcal{C} - \mathcal{C}_i$$
The $i$-th degeneracy of $\mathcal{C}$ to be the category obtained by replicating $\mathcal{C}_i$. More precisely, first form the category $\mathcal{C} \times \Delta^1$, where $\Delta^1 = \{0\rightarrow 1 \}$ is the category with two objects and a single morphism between them. For each $i$ the category $\mathcal{C} \times \Delta^1$ contains two copies of $\mathcal{C}_i$ denoted $(\mathcal{C}_0, 0)$ and $(\mathcal{C}_i, 1)$. Then we have 
$$s_i\mathcal{C} = \mathcal{C} \times \Delta^1 - (\mathcal{C}_0, 1) - \dots -(\mathcal{C}_{i-1}, 1) - (\mathcal{C}_{i+1}, 0) - \dots - (\mathcal{C}_n, 0)$$

We illustrate with an example for clarity. Let $\mathcal{C}_0 = (a)$ be a one object category with just the identity morphism and object labelled $a$ and $\mathcal{C}_1 = (b \rightarrow c)$ be the interval category with objects labelled $b$ and $c$. Consider an object $(\mathcal{C}_0, \mathcal{C}_1, \mathcal{C}) \in \text{\textbf{Cat}}^\sharp_1$ (which is simply a profunctor)
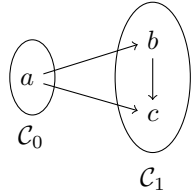
\begin{center}
	\begin{tikzpicture}
	
	\draw  (-4.1,3.5) ellipse (0.3 and 0.5);
	\draw  (-2.5,3.5) ellipse (0.5 and 1);
	\node (v2) at (-2.5,4) {$b$};
	\node (v3) at (-2.5,3) {$c$};
	\node (v1) at (-4.1705,3.4784) {$a$};
	\draw[->]  (v1) edge (v2);
	\draw[->]  (v1) edge (v3);
	\draw[->]  (v2) edge (v3);
	\node at (-4.1165,2.7043) {$\mathcal{C}_0$};
	\node at (-2.4963,2.1462) {$\mathcal{C}_1$};
	\end{tikzpicture}
\end{center}
given by a commutative triangle. We first form the cylinder $\mathcal{C} \times \Delta^1$ 
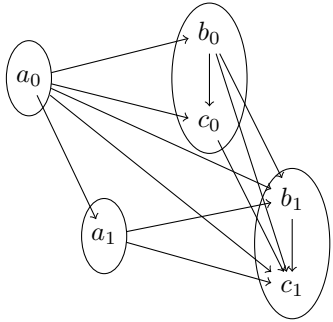
\begin{center}
	\begin{tikzpicture}
	\draw  (-4.8,5.2) ellipse (0.3 and 0.5);
	\draw  (-2.4,5.2) ellipse (0.5 and 1);
	\node (v2) at (-2.4,5.8) {$b_0$};
	\node (v3) at (-2.4,4.6) {$c_0$};
	\node (v1) at (-4.8,5.2) {$a_0$};
	\draw[->]  (v1) edge (v2);
	\draw[->]  (v1) edge (v3);
	\draw[->] (v2) edge (v3);
	
	\draw  (-1.3,3) ellipse (0.5 and 1);
	\node (v4) at (-1.3,3.6) {$b_1$};
	\node (v5) at (-1.3,2.4) {$c_1$};
	\draw[->]  (v2) edge (v4);
	\draw[->]  (v4) edge (v5);
	\draw[->]  (v3) edge (v5);
	\draw[->]  (v2) edge (v5);
	\draw[->]  (v1) edge (v4);
	\draw[->]  (v1) edge (v5);

	\draw  (-3.8,3.1) ellipse (0.3 and 0.5);
	\node (v6) at (-3.8,3.1) {$a_1$};
	\draw[->]   (v1) edge (v6);
	\draw[->]   (v6) edge (v4);
	\draw[->]   (v6) edge (v5);
	\end{tikzpicture}
\end{center}
Then we obtain $s_0\mathcal{C}$ and $s_1\mathcal{C}$ by the prescribed deletions
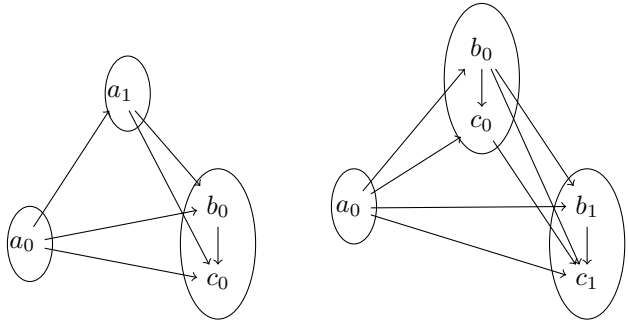
\begin{center}
	
	\begin{tikzpicture}
	
	\draw  (-4.5,3.5) ellipse (0.3 and 0.5);
	\draw  (-2,3.5) ellipse (0.5 and 1);
	\node (v2) at (-2,4) {$b_0$};
	\node (v3) at (-2,3) {$c_0$};
	\node (v1) at (-4.6,3.5) {$a_0$};
	\draw[->]  (v1) edge (v2);
	\draw[->]  (v1) edge (v3);
	\draw[->]  (v2) edge (v3);

	\draw  (-3.2,5.5) ellipse (0.3 and 0.5);
	\node (v4) at (-3.3,5.5) {$a_1$};
	\draw[->]  (v1) edge (v4);
	\draw[->]  (v4) edge (v2);
	\draw[->]  (v4) edge (v3);
	\end{tikzpicture}\ \ \ \ \ \ \ \
	\begin{tikzpicture}
	
	\draw  (-4.1,3.5) ellipse (0.3 and 0.5);
	\draw  (-2.4,5.2) ellipse (0.5 and 1);
	\node (v2) at (-2.4,5.6) {$b_0$};
	\node (v3) at (-2.4,4.6) {$c_0$};
	\node (v1) at (-4.1705,3.4784) {$a_0$};
	\draw[->]  (v1) edge (v2);
	\draw[->]  (v1) edge (v3);
	\draw[->] (v2) edge (v3);
	
	\draw  (-1,3) ellipse (0.5 and 1);
	\node (v4) at (-1,3.5) {$b_1$};
	\node (v5) at (-1,2.5) {$c_1$};
	\draw[->]  (v2) edge (v4);
	\draw[->]  (v4) edge (v5);
	\draw[->]  (v3) edge (v5);
	\draw[->]  (v2) edge (v5);
	\draw[->]  (v1) edge (v4);
	\draw[->]  (v1) edge (v5);
	\end{tikzpicture}
\end{center}
We take advantage of this example to illustrate an important point. We obtain $d_1s_0\mathcal{C}$ by deleting $a_1$ from $s_0\mathcal{C}$  
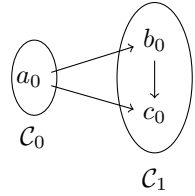
\begin{center}
	\begin{tikzpicture}
	
	\draw  (-4.1,3.5) ellipse (0.3 and 0.5);
	\draw  (-2.5,3.5) ellipse (0.5 and 1);
	\node (v2) at (-2.5,4) {$b_0$};
	\node (v3) at (-2.5,3) {$c_0$};
	\node (v1) at (-4.1705,3.4784) {$a_0$};
	\draw[->]  (v1) edge (v2);
	\draw[->]  (v1) edge (v3);
	\draw[->]  (v2) edge (v3);
	\node at (-4.1165,2.7043) {$\mathcal{C}_0$};
	\node at (-2.4963,2.1462) {$\mathcal{C}_1$};
	\end{tikzpicture}
\end{center}
The simplicial identities dictate $d_1s_0 = 1$ but as it happens too often in category theory strict equalities in two-dimensional environments appear as coherent natural isomorphisms. Indeed we have a natural isomorphism $d_1s_0 \cong 1$. 

In conclusion we have constructed a weak simplicial category
$$\textbf{Cat}^\sharp : \Delta^{op} \rightarrow \textbf{Cat}$$
Verifying the simplcial identities (up to isomorphism) and coherence laws is an easy but tedious exercice and is clear by intuition from the collage perspective. 

In a similar fashion we construct the (weak) simplicial category $\textbf{sSet}^\sharp$. Simply define 
$$\textbf{sSet}^\sharp_n = \textbf{sSet} / \Delta^n$$
with faces and degeneracies exactly as above. To observe the weak simplicial identities we can put things in terms of collages again.

Because of the unfortunate conflict in notation we have denoted by $X(n)$ the set of $n$-simplicies of a simplicial set $X$. For two $0$-simplices $x$ and $y$ in $X$ we denote $X(1)(x,y)$ the set of $1$-simplicies with source $x$ and target $y$. Now we define an $n$-collage of simplicial sets to be an $(n+2)$-tuple of simplicial sets $(X_0, ... X_n, X)$ such that 
\begin{itemize}
	\item[i)] $X(0) = \coprod_{i=0}^n X_i(0)$
	\item[ii)] $X(1)(x,y) = X_i(1)(x,y)$ if $x,y \in X_i(0)$
	\item[iii)] $X(1)(x,y) = \emptyset$ if $x \in X_i(0)$, $y \in X_j(0)$ and $i>j$
\end{itemize}
A morphism $f: X \rightarrow Y$ will be a map of simplicial sets such that $f(x) \in Y_i(0)$ for all $x \in X_i(0)$. 

Just as we did for categories we define faces by deletion, meaning $$d_iX = X - X_i$$ 
$X - X_i$ is simply the simplicial subset of $X$ generated by the simplicies of $X$ whose vertices are not in $X_i$. The degeneracy $s_iX$
 is defined as 
 $$s_iX = X \times \Delta^1 - (X_0, 1) - \dots -(X_{i-1}, 1) - (X_{i+1}, 0) - \dots - (X_n, 0)$$

Interesting spaces are examples of simplices in $\text{\textbf{sSet}}^\sharp$. For two simplicial sets $X$ and $Y$ we have the triples $(X,Y, X \coprod Y)$ and $(X,Y, X * Y)$ in $\text{\textbf{sSet}}^\sharp_1$, where $X*Y$ is the join of $X$ and $Y$. In particular the degeneracy $s_0X$ is presented by the triple $(X,X, X\times \Delta^1)$. 

Conceptualising the cylinder as a degeneracy is a practice we observe throughout this work. This is so because the space $X\times \Delta^n$, which may be regarded as a \textit{higher cylinder}, is naturally equipped with the projection map $X\times \Delta^n \rightarrow \Delta^n$ and hence is an $n$-collage.

Also for a map $f: X \rightarrow Y$ the triple $(X,Y,M_f)$ is present in $\text{\textbf{sSet}}^\sharp_1$, where $M_f$ is the mapping cylinder of $f$. This triple will turn out to be the companion of the vertical morphism $f$, once companions are defined appropriately. Most importantly a cell of the form
\begin{center}
	\begin{tikzpicture}
	
	\node (v1) at (-3,3) {$X$};
	\node (v2) at (-1.5,3) {$X$};
	\node (v3) at (-3,1.5) {$Y$};
	\node (v4) at (-1.5,1.5) {$Y$};
	\draw[-]  (v1) edge[double] node[above, font = \scriptsize]{$s_0X$} (v2);
	\draw[->]  (v1) edge node[left, font = \scriptsize]{$f$} (v3);
	\draw[->]  (v2) edge node[right, font = \scriptsize]{$g$} (v4);
	\draw[-]  (v3)  edge[double] node[below, font = \scriptsize]{$s_0Y$} (v4);
	\end{tikzpicture}
\end{center}
represents precisely a homotopy $f \Rightarrow g$ in \textbf{sSet}. So we see in this example a simplicial category full of homotopical meaning. We will show that the double category theory in these simplicial categories recovers a lot of interesting homotopical information.

Following the same pattern we may define $\textbf{Top}^\sharp$ with 
$$\textbf{Top}^\sharp_n = \textbf{Top}/ |\Delta^n|$$
where \textbf{Top} is the category of topological spaces and $|\Delta^n|$ is the topological $n$-simplex. This example is somehow pathological. First, there is no nice description of the objects of $\textbf{Top}^\sharp_n$ in terms of collages. Second we have counterintuitive objects, for example if $X$ is a space we can consider the constant map $X \rightarrow |\Delta^1|$ at some point in the interior of $|\Delta^1|$. Then $X$ would be a "proarrow" between empty spaces. 

It is perhaps more interesting to consider only CW-complexes and cellular maps. We can reiterate the above by considering $|\Delta^n|$ equipped with the standard CW-structure: one $k$-cell for every subset of order $k$ of $[n]$. The above pathology does not appear because we force structure maps to be cellular. We can also define collages to be $(n+2)$-tuples of CW-complexes $(X_0, ..., X_n, X)$ such that 
\begin{itemize}
	\item[i)] $\text{sk}_0X = \coprod_{i=0}^n \text{sk}_0X_i$
	\item[ii)] The spaces $X_i$ are cellularly embedded in $X$
	\item[iii)] A cell in $X$ whose boundary lies in one of the $X_i$ is a cell of $X_i$
\end{itemize}
where $sk_nX$ denotes the $n$-skeleton of $X$. However such (natural) collages do not correspond precisely to cellular maps $X \rightarrow |\Delta^n|$. 

We used this notion of deletion in the previous examples to define faces and degeneracies. Deletion of one of the $X_i$'s has to be understood in the CW-context, meaning we delete the $1$-cells of $X_i$ right from the first step of constructing $X$ which means we also delete the $1$-cells whose boundary touches one of these cells. We keep deleting higher cells this way.

Hopefully the reader can now define more examples himself. We describe a last one.
Let $\mathcal{C}$ be any category. Let the category $\text{\textbf{cst}}_n$, which we call the \textit{costar} category, have $(n+1)$ objects $\{o_0, ... o_n, o\}$ and just a single morphism $o_i \rightarrow o$ for each $i$. Define an $n$-cospan in $\mathcal{C}$ to be a functor $$\text{cst}_n \rightarrow \mathcal{C}$$ and a map of $n$-cospans is just a natural transformation. Then let $\text{\textbf{coSpan}}(\mathcal{C})^\sharp_n$ be the category of $n$-cospans in $\mathcal{C}$. A map in $\text{\textbf{coSpan}}(\mathcal{C})^\sharp_2$ is just a diagram of the shape
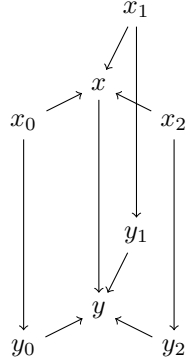
\begin{center}
	\begin{tikzpicture}
	
	\node (v2) at (-2.5,2) {$x$};
	\node (v1) at (-3.5,1.5) {$x_0$};
	\node (v3) at (-2,3) {$x_1$};
	\node (v4) at (-1.5,1.5) {$x_2$};
	\node (v6) at (-2.5,-1) {$y$};
	\node (v7) at (-2,0) {$y_1$};
	\node (v5) at (-3.5,-1.5) {$y_0$};
	\node (v8) at (-1.5,-1.5) {$y_2$};
	\draw[->]  (v1) edge (v2);
	\draw[->]  (v3) edge (v2);
	\draw[->]  (v4) edge (v2);
	\draw[->]  (v5) edge (v6);
	\draw[->]  (v7) edge (v6);
	\draw[->]  (v8) edge (v6);
	\draw[->]  (v1) edge (v5);
	\draw[->]  (v2) edge (v6);
	\draw[->]  (v3) edge (v7);
	\draw[->]  (v4) edge (v8);
	\end{tikzpicture}
\end{center}
with the functors $x, y: \text{cst}_2 \rightarrow \mathcal{C} $ depicted on top and bottom of the picture. And so on in higher dimension. Of course the objects of $\text{\textbf{coSpan}}(\mathcal{C})^\sharp_1$ are the usual cospans.

Let $x : \text{cst}_n \rightarrow \mathcal{C}$ be a functor with $x(o_i) = x_i$. We define faces of $x$ by deleting one of the $x_i$'s. More precisely let $d^i : \text{cst}_{n-1} \rightarrow \text{cst}_n$ be the inclusion functor with $$d^i(o_j) = \begin{cases}
o_j &\text{if} \ \  0 \leq j \leq i \\
o_{j+1} &\text{if} \ \ i+1 \leq j \leq n
\end{cases}$$
Let $d_ix = x\circ d^i$ be obtained via postcomposition. 

We define degeneracies by repetition of one of the $x_i$'s. More precisely let $s^i : \text{cst}_n \rightarrow \text{cst}_{n+1}$ be the projection functor with 
$$s^i(o_j) = \begin{cases}
o_j &\text{if} \ \  0 \leq j \leq i \\
o_{j-1} &\text{if} \ \ i+1 \leq j \leq n+1
\end{cases}$$
Let $s_ix = x\circ s^i$ be obtained via postcomposition. Again the simplicial identities hold but this time strictly.

\begin{note}
	In the simplicial category $\textbf{Cat}^\sharp$ an $n$-simplex $\mathcal{C}$ corresponds to a lax 2-functor
	$$\Delta^n \rightarrow \textbf{Prof}_h$$
	to the horizontal bicategory of \textbf{Prof}, and a morphism in $\textbf{Cat}^\sharp_n$ corresponds to a vertical transformation of such functors. This indicates that $\textbf{Cat}^\sharp$ is some kind of nerve of the double category \textbf{Prof}. This nerve should be analogous to the Duskin nerve for weak 2-categories as defined in \cite{duskin2002simplicial}. We expect to characterize double categories as certain simplicial categories and also expect nerve of equipments to satisfy the equipment property we define in the next section. Such a result would show that $\textbf{sSet}^\sharp$ is not the nerve of a double category and would resonate with the fact that we do not have a double category of simplicial sets. Perhaps we will make these claims more precise in future works.
\end{note}
\begin{note}
	Regrettably we did not construct a simplicial category $\textbf{Set}^\sharp$ with vertical category \textbf{Set}. Such a simplicial category ought to provide a Yoneda theory similar to that discussed in \cite{pare2011yoneda}. We also expect it to provide a free cocompletion with respect to the double colimits we define in \ref{dcolim}.
\end{note}

\subsection{The equipment property}
\label{equip}

\begin{definition}
	Let $\mathbb{E}$ be a simplicial category. $\mathbb{E}$ is said to be an equipment if given any $x \in \mathbb{E}_n$ and maps in $\mathbb{E}_{n-1}$ $$f^i : d_ix \rightarrow y^i$$  such that  $d_if^j = d_{j-1}f^i$ for all $i<j$, then there is an object $y \in \mathbb{E}_n$ equipped with a map $f: x \rightarrow y$ satisfying the following:
	\begin{itemize}
		\item[i)] $d_iy = y^i$ and $d_if = f^i$ for $i= 0, \dots , n$
		\item[ii)] Given any map $h: x \rightarrow z$ in $\mathbb{E}_n$ such that  $d_ih = g^if$ for some $g^i: y^i \rightarrow d_iz$ in $\mathbb{E}_{n-1}$ then there is a unique $g : y \rightarrow z$ such that $h=gf$ and $d_i(g) = g^i$
	\end{itemize}
\end{definition}

The above definition needs some unpacking (and packing). The first thing to notice is that the collection of maps $f^i : d_ix \rightarrow y^i$ which satisfy the prescribed condition can be thought of as a morphism in $E$ from the boundary $\partial x$ of $x$. This is in virtue of the following elementary proposition (which can be found in \cite[p.~11]{goerss2009simplicial})
\begin{proposition}
	Let $S$ be a simplicial set. A map $\sigma: \partial\Delta^n \rightarrow S$ corresponds to a correction of $(n-1)$-simplices $\sigma^0, \sigma^1, \dots ,\sigma^n \in S_{n-1}$ satisfying $d_i\sigma^j = d_{j-1}\sigma^i$ for all $i<j$. 
\end{proposition}
In light of this fact we will denote a map $\partial\Delta^n \rightarrow \mathbb{E}$ given by $y^i \in E_{n-1}$ by $y^\bullet$, and a transformation between such maps $f^\bullet$. This way we can paraphrase the equipment property more elegantly as:
\begin{quote}
	A simplicial category $\mathbb{E}$ satisfies the \textit{equipment property} if for all $x \in \mathbb{E}_n$ every map from the boundary of $x$
	$$f^\bullet : \partial x \rightarrow y^\bullet$$
	can be universally extended to a map $f : x \rightarrow y$ in $\mathbb{E}_n$ such that $\partial f = f^\bullet$ (and as a consequence $\partial y = y^\bullet$ ). Universality is understood to mean: any map $h: x \rightarrow z$ in $\mathbb{E}_n$ whose faces factor through $f^i$ factors uniquely through f. 
\end{quote} 

Our main example of a higher equipment is $\textbf{sSet}^\sharp$. Indeed assume we have $X \in \text{\textbf{sSet}}^\sharp_n$,  $Y^i \in \text{\textbf{sSet}}^\sharp_{n-1}$ satisfying the above coherent condition and a morphism $f^\bullet: \partial X \rightarrow Y^\bullet$. Then we can take the universal filler $Y = X(f^\bullet)$ to be given as a pushout in \textbf{sSet}
\begin{center}
	\begin{tikzpicture}
	
	\node (v1) at (-3,3) {$\partial X$};
	\node (v2) at (-3,1) {$Y^\bullet$};
	\node (v3) at (-1,3) {$X$};
	\node (v4) at (-1,1) {$Y$};
	\draw[->]  (v1) edge node[left, font= \scriptsize]{$f^\bullet$} (v2);
	\draw[right hook ->]  (v1) edge (v3);
	\draw[->, dashed]  (v3) edge (v4);
	\draw[->, dashed]  (v2) edge (v4);
	\end{tikzpicture}
\end{center} 
Similarly $\textbf{Cat}^\sharp$, $\textbf{Top}^\sharp$ and $\textbf{coSpan}(\mathcal{C})^\sharp$ satisfy the equipment property. 

Simplicial categories form a 2-category themselves (because \textbf{Cat} is a 2-category, see \cite{leinster2004higher}), which we denote \textbf{sCat}. The friend of 2-categories may understand the equipment property depicted as:
\begin{center}
	\begin{tikzpicture}
	
	\node (v1) at (-2.5,3) {$\partial \Delta^n$};
	\node (v2) at (-2.5,1) {$\Delta^n$};
	\node (v3) at (0,3) {$\mathbb{E}$};
	\draw[->]  (v1) edge (v2);
	
	\draw[->, bend right = 30]  (v1) edge  node[below, font= \scriptsize]{$y^\bullet$} (v3);
	\draw[->, bend left = 30]  (v1) edge node[above, font= \scriptsize]{$\partial x$} (v3);
	\node (v4) at (-1.2,3.4) {};
	\node (v5) at (-1.2,2.6) {};
	\draw[->]  (v4) edge[double] node[left, font= \scriptsize]{$f^\bullet$} (v5);
	
	\draw[->]  (v2) edge node[below, font= \scriptsize]{$x$} (v3);
	\draw[->, bend right = 50, dashed]  (v2) edge node[right, font= \scriptsize]{$\exists y$} (v3);
	
	\node (v6) at (-1.2,2) {};
	\node (v7) at (-0.5,1.4) {};
	\draw[->]  (v6) edge[double] node[above, font= \scriptsize]{$f$} (v7);
	\end{tikzpicture}
\end{center}
We have not drawn the full universal property in order to not overload the picture. It seems from the diagram that the equipment property is some sort of fibrancy condition (we have no idea in what sense though). 

The equipment property is analogous to the one for double categories presented in the previous section. If we postulate it for $x \in \mathbb{E}_1$ we have a universal filler for every niche in solid arrows
\begin{center}
	\begin{tikzpicture}
	
	\node (v1) at (-3,3) {$x_0$};
	\node (v2) at (-1.5,3) {$x_1$};
	\node (v3) at (-3,1.5) {$y_0$};
	\node (v4) at (-1.5,1.5) {$y_1$};
	\draw[->]  (v1) edge node[above, font = \scriptsize]{$x$} (v2);
	\draw[->]  (v1) edge node[left, font = \scriptsize]{$f^0$} (v3);
	\draw[->]  (v2) edge node[right, font = \scriptsize]{$f^1$} (v4);
	\draw[->, dashed]  (v3)  edge (v4);
	\end{tikzpicture}
\end{center}

However we can entail more. Assume we are given $x \in \mathbb{E}_n$ and some morphisms $f_i: x_i \rightarrow y_i$ in $\mathbb{E}_0$.
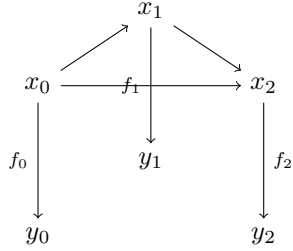
\begin{center}
	\begin{tikzpicture}

	\node (v1) at (-4,2.5) {$x_0$};
	\node (v2) at (-2.5,3.5) {$x_1$};
	\node (v3) at (-1,2.5) {$x_2$};
	\node (v4) at (-4,0.5) {$y_0$};
	\node (v5) at (-2.5,1.5) {$y_1$};
	\node (v6) at (-1,0.5) {$y_2$};
	\draw[->]  (v1) edge (v2);
	\draw[->]  (v2) edge (v3);
	\draw[->]  (v1) edge (v3);
	\draw[->]  (v1) edge node[left, font= \scriptsize]{$f_0$} (v4);
	\draw[->]  (v2) edge node[left, font= \scriptsize]{$f_1$} (v5);
	\draw[->]  (v3) edge node[right, font= \scriptsize]{$f_2$} (v6);
	\end{tikzpicture}
\end{center}
Then we may apply the equipment property to fill in the $1$-simplicies in the bottom first, then the $2$-simplicies and so on to construct a universal filler $y$ with a map $f: x \rightarrow y$ such that $v_if = f_i$
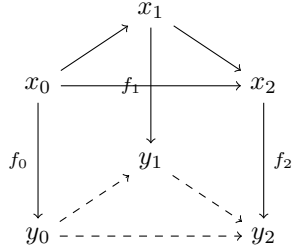
\begin{center}
	\begin{tikzpicture}

	\node (v1) at (-4,2.5) {$x_0$};
	\node (v2) at (-2.5,3.5) {$x_1$};
	\node (v3) at (-1,2.5) {$x_2$};
	\node (v4) at (-4,0.5) {$y_0$};
	\node (v5) at (-2.5,1.5) {$y_1$};
	\node (v6) at (-1,0.5) {$y_2$};
	\draw[->]  (v1) edge (v2);
	\draw[->]  (v2) edge (v3);
	\draw[->]  (v1) edge (v3);
	\draw[->]  (v1) edge node[left, font= \scriptsize]{$f_0$} (v4);
	\draw[->]  (v2) edge node[left, font= \scriptsize]{$f_1$} (v5);
	\draw[->]  (v3) edge node[right, font= \scriptsize]{$f_2$} (v6);
	\draw[->, dashed]  (v4) edge (v5);
	\draw[->, dashed]  (v5) edge (v6);
	\draw[->, dashed]  (v4) edge (v6);
	\end{tikzpicture}
\end{center}
Similarly we may obtain universal fillers for other configurations of maps from faces of $x$ which are coherent (agree on subfaces). We make this observation precise in the following proposition.

\begin{proposition}[Strong equipment property]
	\label{strong}
	Suppose $\mathbb{E}$ is an equipment and $x \in \mathbb{E}_n$ an $n$-simplex. Let $A_\alpha$, for $\alpha \in A$, be a family of subsets of $[n] = \{ 0, 1, \dots ,n\}$ that cover $[n]$. Denote $x_\alpha$ be the subsimplex of $x$ with vertices $\{x_i | i\in A_\alpha\}$, $A_{\alpha\beta} = A_\alpha \cap A_\beta$ and the corresponding subsimplex $x_{\alpha\beta}$. Then given a family of morphisms in $E$
	$$f_\alpha : x_\alpha \rightarrow y_\alpha$$
	such that for all $\alpha, \beta \in A$ the restrictions of $f_\alpha$ and $f_\beta$ on $x_{\alpha\beta}$ agree, there is an object $y \in \mathbb{E}_n$ and a morphism $f: x \rightarrow y$ such that:
	\begin{itemize}
		\item[i)] $y_\alpha$ is the subsimplex of $y$ with vertices $\{y_i | i\in A_\alpha\}$ and the restriction of $f$ to $x_\alpha$ is $f_\alpha$
		\item[ii)] Any map $x \rightarrow z$ in $\mathbb{E}_n$ whose restrictions to $x_\alpha$ factor through $f_\alpha$ factors uniquely through $f$
	\end{itemize}
\end{proposition}

The proof is not difficult but we isolate the following lemma:

\begin{lemma}
	Let $\mathbb{E}$ be an equipment and $x \in \mathbb{E}_n$ be an object. Given morphisms $f^n : d_nx \rightarrow y^i$ in $\mathbb{E}_{n-1}$ and $f_n: x_n \rightarrow y_n$ in $\mathbb{E}_0$ then there is a object $y \in \mathbb{E}_n$ and a map $f: x \rightarrow y$ such that:
	\begin{itemize}
		\item[i)] $d_nf = f^n$ and $v_nf = f_n$
		\item[ii)] Any morphism $h: x \rightarrow z$ such that $d_nh$ factors through $f^n$ and $v_nf = f_n$ factors though $f$
	\end{itemize} 
\end{lemma}
\begin{proof}
	We construct $1$-simplicies joining $y_i$ with $y_n$ using the equipment property
	\begin{center}
		\begin{tikzpicture}
		
		\node (v1) at (-3,3) {$x_i$};
		\node (v2) at (-1.5,3) {$x_n$};
		\node (v3) at (-3,1.5) {$y_i$};
		\node (v4) at (-1.5,1.5) {$y_n$};
		\draw[->]  (v1) edge (v2);
		\draw[->]  (v1) edge node[left, font = \scriptsize]{$v_if^n$} (v3);
		\draw[->]  (v2) edge node[right, font = \scriptsize]{$f^n$} (v4);
		\draw[->, dashed]  (v3)  edge (v4);
		\end{tikzpicture}
	\end{center}
	This way we create boundaries of $2$-simplicies so that we use the equipment property until we have a map from $\partial x$. Then the equipment property gives the desired simplex $y$ satisfying the universal property. 
\end{proof}

\begin{proof}(of the proposition)
	We proceed by induction on $n$. If $n=0$ there is nothing to do. For $n>0$ use the induction hypothesis to construct a universal map $f^n : d_nx \rightarrow y^n$ and then use the lemma to complete the construction of the desired $y$. 
\end{proof}

We will denote universal constructions such as the above $x(f^\bullet)$, $x(f^1, \dots, f^n)$, $x(f_\alpha)$ or $x(f_1,\dots f_n)$ in trust that it is clear from the context which universal extension is being talked about.

\subsection{The companion construction for a vertical $n$-simplex}
\label{companion}

For a simplicial category $\mathbb{E}$ and an object $x \in \mathbb{E}_0$ we denote by $s^nx$ the image of $x$ under the map 
$$\mathbb{E}_0 \xrightarrow{s_0} \mathbb{E}_1 \xrightarrow{s_0} \dots \xrightarrow{s_0} \mathbb{E}_n$$

Let $\mathbb{E}$ be a higher equipment and let 
$$\sigma = x_0 \xrightarrow{f_1} x_1 \xrightarrow{f_2} \dots \xrightarrow{f_n} x_n$$
be an $n$-simplex in the category $E_0$. We define a companion simplex  $\sigma^* \in E_n$ together with a morphism $\phi_\sigma : s^nx_0 \rightarrow \sigma^*$ via the following recursion
$$\sigma^*= \begin{cases}
\sigma &\text{if} \ \  n=0 \\
s^nx_0(\phi_{d_n\sigma}, \dots, \phi_{d_0\sigma}) &\text{if} \ \ n>0
\end{cases}$$
We define $\phi_\sigma = 1_\sigma $ for $\sigma \in E_0$ and we define $\phi_\sigma$ to be the induced map $s^nx_0 \rightarrow s^nx_0(\phi_{d_n\sigma}, \dots, \phi_{d_0\sigma})$.

As mysterious as this recursion might look it has a pretty basic idea behind it. If $x \in \mathbb{E}_0$ is just an object then its companion in $\mathbb{E}_0$ is just itself. If $f: x_0 \rightarrow x_1$ is a $1$-simplex in $\mathbb{E}_0$ we obtain the companion $f^* \in \mathbb{E}_1$ as a universal filler  
\begin{center}
	\begin{tikzpicture}
	
	\node (v1) at (-3,3) {$x_0$};
	\node (v2) at (-1.5,3) {$x_0$};
	\node (v3) at (-3,1.5) {$x_0$};
	\node (v4) at (-1.5,1.5) {$x_1$};
	\draw  (v1) edge[double] node[above, font = \scriptsize]{$s_0x_0$} (v2);
	\draw[->]  (v1) edge node[left, font = \scriptsize]{$1_{x_0}$} (v3);
	\draw[->]  (v2) edge node[right, font = \scriptsize]{$f$} (v4);
	\draw[->, dashed]  (v3)  edge node[above, font = \scriptsize]{$f^*$} (v4);
	\end{tikzpicture}
\end{center}
We automatically obtain the map $\phi_f : s_0x_0 \rightarrow f^*$. As we may intuit from the examples the degeneracy $s_0x_0$ represents the cylinder of $x_0$ and extending the cylinder along $f: x_0 \rightarrow x_1$ universally to form $f^*$ represents the formation of the mapping cylinder of $f$. 

If we had a $2$-simplex $\sigma : x_0 \xrightarrow{f} x_1 \xrightarrow{g} x_2$ we would first form the companions of $f$, $g$ and $gf$ as above and then form the universal extension 
\begin{center}
	\begin{tikzpicture}

	\node (v1) at (-4,2.5) {$x_0$};
	\node (v2) at (-2.5,3.5) {$x_0$};
	\node (v3) at (-1,2.5) {$x_0$};
	\node (v4) at (-4,0.5) {$x_0$};
	\node (v5) at (-2.5,1.5) {$x_1$};
	\node (v6) at (-1,0.5) {$x_2$};
	\draw[-]  (v1) edge[double] (v2);
	\draw[-]  (v2) edge[double] (v3);
	\draw[-]  (v1) edge[double]  (v3);
	\draw[->]  (v1) edge node[left, font= \scriptsize]{$1_{X_0}$} (v4);
	\draw[->]  (v2) edge node[left, font= \scriptsize]{$g$} (v5);
	\draw[->]  (v3) edge node[right, font= \scriptsize]{$gf$} (v6);
	\draw[->]  (v4) edge node[above, font = \scriptsize]{$f^*$} (v5);
	\draw[->]  (v5) edge node[above, font = \scriptsize]{$g^*$} (v6);
	\draw[->]  (v4) edge node[below, font = \scriptsize]{$(gf)^*$} (v6);
	\node at (-2.5,0.9) {$\sigma^*$};
	\end{tikzpicture}
\end{center}
Again we obtain the induced map $\phi_\sigma$. We also observe that our recursion step is well defined, in the sense that it is guaranteed by the recursion itself that the morphisms $\phi_{d_i\sigma}$ agree on subfaces.

$\sigma^*$ may be thought as representing the mapping cylinder of $\sigma$, or what we may call \textit{higher mapping cylinder}. We will make this precise in a little bit. 

The companion construction 
$$\sigma \mapsto \sigma^*$$
associates to an $n$-simplex of $\mathbb{E}_0$ an object of $\mathbb{E}_n$. Since the latter are horizontal $n$-simplicies of $E$ we expect this construction to define a transformation
$$(\cdot)^* : \mathbb{E}_0 \rightarrow \mathbb{E}$$
First observe that it follows from the recursion that 
$$d_i\sigma^* = (d_i\sigma)^*$$
However the axioms do not seem to guarantee $s_i\sigma^* = (s_i\sigma)^*$. Nonetheless they guarantee the next best thing which is a comparison map 
$$\alpha_i : (s_i\sigma)^* \rightarrow s_i\sigma^*$$
We may construct $\alpha_i$ recursively. For an object $x \in \mathbb{E}_0$ we have $x^* = x$ and clearly $s_0x^* = (s_0x)^*$ so we choose $\alpha_0$ in this case to be the identity map. For $\sigma$ an $n$-simplex as above we let $\alpha_i$ be the map induced by the equipment property according to the diagram:
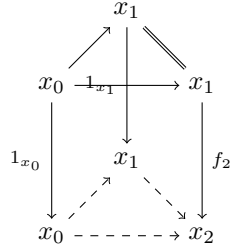
\begin{center}
	\begin{tikzpicture}
	
	\node (v1) at (-3,3) {$s^{n+1}x_0$};
	\node (v2) at (0,3) {$s_i\sigma^*$};
	\node (v3) at (-3,1) {$(s_i\sigma)^*$};
	\draw[->]  (v1) edge node[above, font= \scriptsize]{$s_i\phi_\sigma$} (v2);
	\draw[->] (v1) edge node[left, font= \scriptsize]{$\phi_{s_i\sigma}$} (v3);
	\draw[->, dashed]  (v3) edge node[below, font= \scriptsize]{$\exists \alpha_i$} (v2);
	\end{tikzpicture}
\end{center}
What the recursion guarantees is that the faces of $s_i\phi_\sigma$ factor through the faces of $\phi_{s_i\sigma}$ so that we can use the universal property. Hence we obtain an oplax transformation $(\cdot)^*$.

We conclude with an important observation. Let us illustrate with an example, say 
$$\sigma : x_0 \xrightarrow{f_1} x_1 \xrightarrow{f_2} x_2 $$
is a $2$-simplex in $E_0$. Then we can think of $\sigma^*$ as produced by two steps of universal extensions. Given the tower 
\begin{center}
	\begin{tikzpicture}

	\node (v1) at (-3.5,2) {$x_0$};
	\node (v2) at (-2.5,3) {$x_0$};
	\node (v3) at (-1.5,2) {$x_0$};
	\node (v6) at (-3.5,0) {$x_0$};
	\node (v4) at (-2.5,1) {$x_1$};
	\node (v5) at (-1.5,0) {$x_1$};
	\node (v7) at (-3.5,-2) {$x_0$};
	\node (v8) at (-2.5,-1) {$x_1$};
	\node (v9) at (-1.5,-2) {$x_2$};
	
	\draw  (v1) edge[double] (v2);
	\draw  (v2) edge[double] (v3);
	\draw  (v1) edge[double] (v3);
	
	\draw  (v4) edge[double] (v5);

	\draw[->]  (v1) edge node[left, font= \scriptsize]{$1_{x_0}$} (v6);
	\draw[->]  (v2) edge node[left, font= \scriptsize]{$f_1$} (v4);
	\draw[->]  (v3) edge node[right, font= \scriptsize]{$f_1$} (v5);
	\draw[->]  (v6) edge node[left, font= \scriptsize]{$1_{x_0}$} (v7);
	\draw[->]  (v4) edge node[left, font= \scriptsize]{$1_{x_1}$} (v8);
	\draw[->]  (v5) edge node[right, font= \scriptsize]{$f_2$} (v9);
	\end{tikzpicture}
\end{center}
we first construct $s^3x_0(1_{x_0}, s_0f_1) = (s_1f_1)^*$
\begin{center}
	\begin{tikzpicture}

	\node (v1) at (-3.5,2) {$x_0$};
	\node (v2) at (-2.5,3) {$x_0$};
	\node (v3) at (-1.5,2) {$x_0$};
	\node (v6) at (-3.5,0) {$x_0$};
	\node (v4) at (-2.5,1) {$x_1$};
	\node (v5) at (-1.5,0) {$x_1$};

	\draw  (v1) edge[double] (v2);
	\draw  (v2) edge[double] (v3);
	\draw  (v1) edge[double] (v3);
	
	\draw  (v4) edge[double] (v5);

	\draw[->]  (v1) edge node[left, font= \scriptsize]{$1_{x_0}$} (v6);
	\draw[->]  (v2) edge node[left, font= \scriptsize]{$f_1$} (v4);
	\draw[->]  (v3) edge node[right, font= \scriptsize]{$f_1$} (v5);
	\draw[->,dashed] (v6) edge (v4);
	\draw[->,dashed] (v6) edge (v5);
	\end{tikzpicture}
\end{center}
and then $(s_1f_1)(1_{x_0}, 1_{x_1}, f_2)$
\begin{center}
	\begin{tikzpicture}

	\node (v1) at (-3.5,2) {$x_0$};
	\node (v2) at (-2.5,3) {$x_1$};
	\node (v3) at (-1.5,2) {$x_1$};
	\node (v6) at (-3.5,0) {$x_0$};
	\node (v4) at (-2.5,1) {$x_1$};
	\node (v5) at (-1.5,0) {$x_2$};

	\draw[->]   (v1) edge (v2);
	\draw  (v2) edge[double] (v3);
	\draw [->]  (v1) edge (v3);
	
	\draw[->,dashed]   (v4) edge (v5);

	\draw[->]  (v1) edge node[left, font= \scriptsize]{$1_{x_0}$} (v6);
	\draw[->]  (v2) edge node[left, font= \scriptsize]{$1_{x_1}$} (v4);
	\draw[->]  (v3) edge node[right, font= \scriptsize]{$f_2$} (v5);
	\draw[->,dashed] (v6) edge (v4);
	\draw[->,dashed] (v6) edge (v5);
	\end{tikzpicture}
\end{center}
This is true because of this trivial lemma:
\begin{lemma}
	The composition of universal extensions is universal. 
\end{lemma}
\begin{proof}
	This follows directly from the definition of universality. 
\end{proof}
In general we have the following proposition:
\begin{proposition}[Tower representation]
	\label{tower}
	Let $\mathbb{E}$ be a higher equipment and $\sigma$ be an $n$-simplex as above. Then we have 
	$$\sigma^* = s^nx_0(1_{x_0}, s^{n-1}f_1)(1_{x_0}, 1_{x_1}, s^{n-2}f_2) \dots (1_{x_0}, \dots, 1_{x_{n-1}}, f_n)$$
\end{proposition}
\begin{proof}
	The proposition follows from the fact that faces of companions are companions mentioned above and the lemma. 
\end{proof}

Dually we can define the cojoint $\sigma_*$ of an $n$-simplex $\sigma$ in $\mathbb{E}_0$. Let $\sigma_*$ be equipped with a structure morphism $\psi_\sigma: s^nx_0 \rightarrow \sigma_*$ and be given recursively by
$$\sigma_*= \begin{cases}
\sigma &\text{if} \ \  n=0 \\
s^nx_0(\psi_{d_0\sigma}, \dots, \psi_{d_n\sigma}) &\text{if} \ \ n>0
\end{cases}$$
Again let $\psi_\sigma = 1_\sigma $ for $\sigma \in \mathbb{E}_0$ and we define $\psi_\sigma$ to be the induced map $s^nx_0 \rightarrow s^nx_0(\phi_{d_0\sigma}, \dots, \phi_{d_n\sigma})$.

For example for a $1$-simplex $f: x_0 \rightarrow x_1$ we obtain the cojoint $f_*$ as the universal extension 
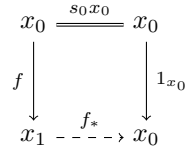
\begin{center}
	\begin{tikzpicture}
	
	\node (v1) at (-3,3) {$x_0$};
	\node (v2) at (-1.5,3) {$x_0$};
	\node (v3) at (-3,1.5) {$x_1$};
	\node (v4) at (-1.5,1.5) {$x_0$};
	\draw  (v1) edge[double] node[above, font = \scriptsize]{$s_0x_0$} (v2);
	\draw[->]  (v1) edge node[left, font = \scriptsize]{$f$} (v3);
	\draw[->]  (v2) edge node[right, font = \scriptsize]{$1_{x_0}$} (v4);
	\draw[->, dashed]  (v3)  edge node[above, font = \scriptsize]{$f_*$} (v4);
	\end{tikzpicture}
\end{center}
All of the above statements have their duals for cojoints.

\begin{note}
	In the examples it seems that the companion construction preserves  degeneracies up to isomorphism so that we have more than a mere comparison map. However the axioms seem to imply simply a lax transformation. We leave the investigation of the conditions that would guarantee invertible comparison maps for future work, as this does not influence what comes next.
\end{note}

\subsection{Double colimits in a simplicial category}
\label{dcolim}

Let $\mathbb{E}$ be a simplicial category. In analogy with double category theory we are interested in defining the double colimit of a horizontal diagram of shape $\mathcal{J}$, where $\mathbb{J}$ is some indexing category. As pointed out previously, while in double categories we have a horizontal 2-category in simplicial categories we have a simplicial structure in the horizontal direction. So for our purposes a horizontal diagram in $E$ will be a transformation $$F: \mathcal{J} \rightarrow E$$
$F$ will assign an object in $\mathbb{E}_n$ to each $n$-simplex of $\mathcal{J}$. 

Because of the observations in the previous section we are also interested in the case where $F$ is oplax. This means for any $n$-simplex $\sigma$ in $\mathcal{J}$ and any map $\theta : [m] \rightarrow [n]$ in $\Delta$ we have a structure map in $\mathbb{E}_m$
$$\psi_\theta : F(\theta(\sigma)) \rightarrow \theta(F(\sigma))$$
satisfying coherence laws. In such case we may write $(F,\psi)$ instead of just $F$. 

\begin{definition}
	Let $(F,\psi)$ be as above. The double colimit of $F$ is an object in $\mathbb{E}_0$, denoted $\text{dcolim}F$, equipped with a structure map 
	$$\eta: F \Rightarrow \text{dcolim}F$$
	such that given any other map $\eta^\prime :F \Rightarrow x$ for some $x\in \mathbb{E}_0$ there is a unique morphism $f: \text{dcolim}F \rightarrow x$ in $E_0$ with $f\circ \eta = \eta^\prime$. 
	\begin{center}
		\begin{tikzpicture}
		
		\node (v1) at (-5,5) {$F$};
		\node (v2) at (-5,3.5) {$\text{dcolim}F$};
		\node (v3) at (-3,3.5) {$x$};
		\draw[->]  (v1) edge[double] node[left, font= \scriptsize]{$\eta$} (v2);
		\draw[->]  (v1) edge[double] node[above, font= \scriptsize]{$\eta^\prime$} (v3);
		\draw[->, dashed]  (v2) edge node[below, font = \scriptsize]{$\exists ! f$} (v3);
		\end{tikzpicture}
	\end{center}
\end{definition}

We have abused notation and denoted the constant functor $$\mathcal{J} \rightarrow \Delta^0 \xrightarrow{x} E$$ by simply $x$. The composition $f\circ \eta $ makes sense if we think of $f$ as a transformation between functors $\Delta^0 \rightarrow E$. 

In order to grasp this colimit better we distinguish a special case. Let $x \in \mathbb{E}_n$ be an object. For simplicity we draw a picture in $\mathbb{E}_2$
\begin{center}
	\begin{tikzpicture}

	\node (v1) at (-4,2.5) {$x_0$};
	\node (v2) at (-2.5,3.5) {$x_1$};
	\node (v3) at (-1,2.5) {$x_2$};
	\draw[->]  (v1) edge (v2);
	\draw[->]  (v2) edge (v3);
	\draw[->]  (v1) edge (v3);
	\end{tikzpicture}
\end{center}
We can think of $x$ as a diagram $\Delta^n \xrightarrow{x} \mathbb{E}$ and hence consider its double colimit. We will call this the \textit{cotabulator} of $x$ and denote it $\bot_x$ (this terminology and notation is borrowed from \cite{grandis1999limits}). $\bot_x$ comes equipped with a map $\eta : x \rightarrow s^{(n)}\bot_x$ in $\mathbb{E}_n$, 
which we can depict $\eta$ as a bisimplex 
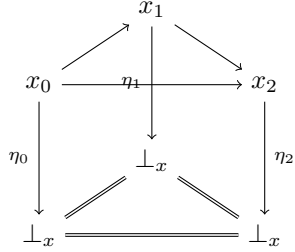
\begin{center}
	\begin{tikzpicture}

	\node (v1) at (-4,2.5) {$x_0$};
	\node (v2) at (-2.5,3.5) {$x_1$};
	\node (v3) at (-1,2.5) {$x_2$};
	\node (v4) at (-4,0.5) {$\bot_x$};
	\node (v5) at (-2.5,1.5) {$\bot_x$};
	\node (v6) at (-1,0.5) {$\bot_x$};
	\draw[->]  (v1) edge (v2);
	\draw[->]  (v2) edge (v3);
	\draw[->]  (v1) edge (v3);
	\draw[->]  (v1) edge node[left, font= \scriptsize]{$\eta_0$} (v4);
	\draw[->]  (v2) edge node[left, font= \scriptsize]{$\eta_1$} (v5);
	\draw[->]  (v3) edge node[right, font= \scriptsize]{$\eta_2$} (v6);
	\draw[-]  (v4) edge[double] (v5);
	\draw[-]  (v5) edge[double] (v6);
	\draw[-]  (v4) edge[double] (v6);
	\end{tikzpicture}
\end{center}
or simply as 
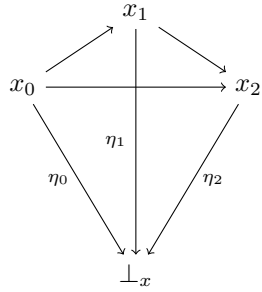
\begin{center}
	\begin{tikzpicture}

	\node (v1) at (-4,2.5) {$x_0$};
	\node (v2) at (-2.5,3.5) {$x_1$};
	\node (v3) at (-1,2.5) {$x_2$};
	\node (v4) at (-2.5,0) {$\bot_x$};

	\draw[->]  (v1) edge (v2);
	\draw[->]  (v2) edge (v3);
	\draw[->]  (v1) edge (v3);
	\draw[->]  (v1) edge node[left, font= \scriptsize]{$\eta_0$} (v4);
	\draw[->]  (v2) edge node[left, font= \scriptsize]{$\eta_1$} (v4);
	\draw[->]  (v3) edge node[right, font= \scriptsize]{$\eta_2$} (v4);

	\end{tikzpicture}
\end{center}
depending on taste. The universal property then is understood as
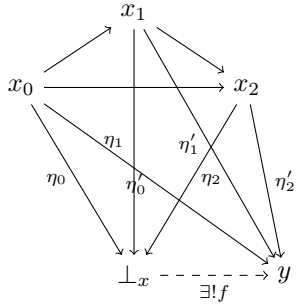
\begin{center}
	\begin{tikzpicture}

	\node (v1) at (-4,2.5) {$x_0$};
	\node (v2) at (-2.5,3.5) {$x_1$};
	\node (v3) at (-1,2.5) {$x_2$};
	\node (v4) at (-2.5,0) {$\bot_x$};
		\node (v5) at (-0.5,0) {$y$};
	
	\draw[->]  (v1) edge (v2);
	\draw[->]  (v2) edge (v3);
	\draw[->]  (v1) edge (v3);
	\draw[->]  (v1) edge node[left, font= \scriptsize]{$\eta_0$} (v4);
	\draw[->]  (v2) edge node[left, font= \scriptsize]{$\eta_1$} (v4);
	\draw[->]  (v3) edge node[right, font= \scriptsize]{$\eta_2$} (v4);
	
	\draw[->]  (v1) edge node[left, font= \scriptsize]{$\eta^\prime_0$} (v5);
	\draw[->]  (v2) edge node[left, font= \scriptsize]{$\eta^\prime_1$} (v5);
	\draw[->]  (v3) edge node[right, font= \scriptsize]{$\eta^\prime_2$} (v5);
	\draw[->,dashed]  (v4) edge node[below, font= \scriptsize]{$\exists ! f$} (v5);
	\end{tikzpicture}
\end{center}
where $y \in \mathbb{E}_0$ is an object and $\eta^\prime : x \rightarrow s^{(n)}y$ is a map in $\mathbb{E}_n$. The speciality of double colimits is that we require the map $f$ to be in $\mathbb{E}_0$ such that $$s^{(n)}f \circ \eta = \eta^\prime$$ rather then any map $s^{(n)}\bot_x \rightarrow s^{(n)}y$ making the diagram commute. 

Cotabulators turn out to be a fundamental notion. We first study the nature of the association $$x \mapsto \bot_x$$. It is obvious that we have a functor $E_n \rightarrow \mathbb{E}_0$ but there is more. Let $\textbf{Gro}^\prime(E)$ be the category with:
\begin{itemize}
	\item[($\bullet$)] objects all $x \in \mathbb{E}_n$ for various $[n] \in \Delta$
	\item[($\rightarrow$)] morphisms from $x\in E_m$ to $y\in \mathbb{E}_n$ being pairs $(\theta, f)$ where $\theta: [m] \rightarrow [n]$ is a map in $\Delta$ and $f: x \rightarrow \theta(y)$ is a morphism in $\mathbb{E}_n$
	\item[($\circ$)] composition of morphisms similar to the Grothendieck construction presented in the previous section 
\end{itemize}
This category is dual to the Grothendieck construction and hence the notation. Now we observe that the cotabulator construction extends to a functor $$\bot: \textbf{Gro}^\prime(E) \rightarrow \mathbb{E}_0$$
Indeed let $(\theta,f) : x \rightarrow y$ be as above. The composite map in $\mathbb{E}_m$
$$x \xrightarrow{f} \theta(y) \xrightarrow{\theta(\eta)} \theta(s_0^{(n)}\bot_y) = s_0^{(m)}\bot_y$$ 
where $\eta$ denotes the structure map of the cotabulator as in the above definition,
induces a morphism $$\bot(\theta,f) : \bot_x \rightarrow \bot_y$$
by the universal property of cotabulators.

Next we prove a structure theorem which says that any double colimit as introduced above can be formed by cotabulators and ordinary colimits in $\mathbb{E}_n$. An analogous result for double categories can be found in \cite{grandis1999limits}. 

Again let 
$$(F,\psi) : \mathcal{J} \rightarrow \mathbb{E}$$
be a diagram. Let $\textbf{Gro}(\mathcal{J})$ be the Grothendieck construction of the simplicial set $\mathcal{J}$. Explicitly this is the category with 
\begin{itemize}
	\item[($\bullet$)] objects the simplicies $\sigma \in \mathcal{J}_n$ for various $[n] \in \Delta$
	\item[($\rightarrow$)] morphisms from $\sigma \in \mathcal{J}_n$ to $\tau \in \mathcal{J}_m$ being maps $\theta : [m] \rightarrow [n]$ such that $\tau = \theta(\sigma)$
	\item[($\circ$)] composition as usual
\end{itemize}
Then $(F,\psi)$ induces a functor 
$$\textbf{Gro}(F,\psi) : \textbf{Gro}(\mathcal{J})^{op} \rightarrow \textbf{Gro}^\prime(E)$$
On objects the functor is given by $\sigma \mapsto F(\sigma)$ for $\sigma\in \mathcal{J}_n$. If $\theta: [m] \rightarrow [n]$ is a map in $\delta$ and $\tau = \theta(\sigma)$ then we have a morphism in $E_0$
$$(\theta, \psi_\theta) : F(\tau) = F(\theta(\sigma)) \rightarrow \theta(F(\sigma))$$ in $\textbf{Gro}^\prime(\mathbb{E})$. 

\begin{theorem}[Construction of double colimits]
	\label{colimconstruction}
	In a setting as above we have a natural isomorphism
	$$\text{dcolim}F \cong \text{colim}\{\textbf{Gro}(\mathcal{J})^{op} \xrightarrow{\textbf{Gro}(F,\psi)} \textbf{Gro}^\prime(\mathbb{E}) \xrightarrow{\bot} \mathbb{E}_0\}$$
	given that all cotabulators exist in $\mathbb{E}$ and $\mathbb{E}_0$ is cocomplete.
\end{theorem}
\begin{proof}
	The statement of the theorem is somehow tautological. Let us denote the colimit on the right by $c \in \mathbb{E}_0$. $c$ comes equipped with a morphism, say $\eta : \bot_{F(\sigma)} \rightarrow c$ for each $\sigma \in \mathcal{J}$, such that the appropriate diagrams commute. 
	
	But $\eta$ can be thought of as a morphism in $F(\sigma) \rightarrow s_0^{(n)}c$ because of the universal property of cotabulators. Hence $c$ satisfies the universal property of double colimits and we have $$c \cong \text{dcolim}F$$
\end{proof}

We conclude with a proposition which relates cotabulators and universal extensions.

\begin{proposition}
	\label{cotabextension}
	Let $\mathbb{E}$ be an equipment and $x \in \mathbb{E}_n$ equipped with morphisms $f_\alpha : x_\alpha \rightarrow y_\alpha$ from a family of subsimplicies $x_\alpha$ as in the setting of the strong equipment property \ref{strong}. Then the cotabulator of the universal extension $x(f_\alpha)$ is given as a pushout
	\begin{center}
		\begin{tikzpicture}
		\node (v1) at (-3.5,4) {$\coprod_\alpha \bot_{x_\alpha}$};
		\node (v2) at (-1.5,4) {$\bot_x$};
		
		\node (v3) at (-3.5,2) {$\coprod_\alpha \bot_{y_\alpha}$};
		\node (v4) at (-1.5,2) {$\bot_{x(f_\alpha)}$};
		\draw[->]  (v1) edge (v2);
		\draw[->]  (v1) edge node[left, font = \scriptsize]{$\coprod_\alpha f_\alpha$} (v3);
		\draw[->, dashed]  (v2) edge (v4);
		\draw[->, dashed]  (v3) edge (v4);
		\end{tikzpicture}
	\end{center}
	 
\end{proposition}
\begin{proof}
	The proposition is a mere restatement of the universality of $x(f_\alpha)$. By definition of cotabulators, a map in $\mathbb{E}_0$ 
	$$\bot_{x(f_\alpha)} \rightarrow z$$
	for some object $z$ corresponds to a morphism in $\mathbb{E}_n$
	$$x(f_\alpha) \rightarrow s^nz$$
	which in turn corresponds, by universality, to a map $x \rightarrow s^nz$ such that the restriction $x_\alpha \rightarrow s^nz_\alpha$ factors through $f_\alpha$ for all $\alpha$. Since all $s^nz_\alpha$ are degenerate, by the universal property of cotabulators, the above is precisely the pushout property. 
\end{proof}

Dually we may define the tabulator $\top_x$ for $x\in \mathbb{E}_n$ and the double limit $\text{dlim}F$ of a lax functor $(F,\psi)$. All the above statements have their duals for tabulators and limits. In particular all double limits in $\mathbb{E}$ are formed by tabulators and limits in $\mathbb{E}_0$. We discuss duality in the last section.

\begin{note}
	We established that the association $x \mapsto \bot_x$ for an object $x \in \mathbb{E}_n$ extends to a functor $\textbf{Gro}^\prime(\mathbb{E}) \xrightarrow{\bot} \mathbb{E}_0$. We may also observe that the structure map $\eta : x \rightarrow s_0^{(0)}\bot_x$ gives us a costar in $E_0$ via the vertex maps $x_i \rightarrow \bot_x$. This should define a map 
	$$\bot : E \rightarrow \text{\textbf{coSpan}}(\mathbb{E}_0)$$
	This map seems to describe a cospan representation in the examples. Cospan representation for double categories is discussed in \cite{grandis2017span} and a similar trend seems to appear for simplicial categories. We leave these considerations for future work as well.
\end{note}

\subsection{Simplicially enriched categories}
\label{ssetcat}

Categories enriched over simplicial sets (with its cartesian monoidal structure) are a fundamental structure in categorical homotopy theory. Unfortunately they are named simplicial categories in a lot of the literature on the subject. Our references for this subject are \cite{riehl2014categorical} and \cite{shulman2006homotopy}. 

A simplicially enriched category $\mathcal{C}$, or \textbf{sSet}-category for short, consists of:
\begin{itemize}
	\item[($\bullet$)] A collection of objects $x,y,z \dots$
	\item[($\triangle$)] A simplicial set $\mathcal{C}(x,y)$ for any two objects $x,y$
	\item[($=$)] A $0$-simplex $$1_x : \Delta^0 \rightarrow \mathcal{C}(x,x)$$
	called the identity, for each object $x$
	\item[($\circ$)] A composition operation 
	$$\circ: \mathcal{C}(y,z) \times \mathcal{C}(x,y) \rightarrow \mathcal{C}(x,z)$$
	which is a morphism of simplicial sets, for any three objects $x,y,z$
\end{itemize}
satisfying associativity and unital laws depicted as commutative diagrams
\begin{center}
	\begin{tikzpicture}
	
	\node (v1) at (-3.5,3) {$\Delta^0 \times \mathcal{C}(x,y)$};
	\node (v2) at (0.5,3) {$\mathcal{C}(x,x) \times \mathcal{C}(x,y)$};
	\node (v3) at (-1.5,1.5) {$\mathcal{C}(x,y)$};
	\draw[->]  (v1) edge node[above, font= \scriptsize] {$1_x \times 1_{\mathcal{C}(x,y)}$} (v2);
	\draw[->]  (v2) edge node[right, font= \scriptsize]{$\circ$} (v3);
	\draw[->]  (v1) edge node[left, font= \scriptsize]{$\cong$}(v3);
	\end{tikzpicture} \ \ \ \  \
	\begin{tikzpicture}
	\node (v1) at (-3.5,3) {$\mathcal{C}(x,y) \times \Delta^0$};
	\node (v2) at (0.5,3) {$\mathcal{C}(x,y) \times \mathcal{C}(y,y)$};
	\node (v3) at (-1.5,1.5) {$\mathcal{C}(x,y)$};
	\draw[->]  (v1) edge node[above, font= \scriptsize] {$ 1_{\mathcal{C}(x,y)} \times 1_y$} (v2);
	\draw[->]  (v2) edge node[right, font= \scriptsize]{$\circ$} (v3);
	\draw[->]  (v1) edge node[left, font= \scriptsize]{$\cong$} (v3);
	\end{tikzpicture}
\end{center}
\begin{center}
	\begin{tikzpicture}

	\node (v1) at (-3,3.5) {$(\mathcal{C}(z,w) \times \mathcal{C}(y,z)) \times \mathcal{C}(x,y)$};
	\node (v2) at (3.5,3.5) {$\mathcal{C}(z,w) \times (\mathcal{C}(y,z) \times \mathcal{C}(x,y))$};
	\node (v3) at (-3,1.5) {$\mathcal{C}(y,w) \times \mathcal{C}(x,y)$};
	\node (v4) at (3.5,1.5) {$\mathcal{C}(z,w) \times \mathcal{C}(x,z)$};
	\node (v5) at (0.3,0) {$\mathcal{C}(x,w)$};
	\draw[->]  (v1) edge node[above, font= \scriptsize] {$\cong$}  (v2);
	\draw[->]  (v1) edge node[left, font= \scriptsize]{$\circ \times 1_{\mathcal{C}(x,y)}$} (v3);
	\draw[->]  (v2) edge node[right, font= \scriptsize]{$ 1_{\mathcal{C}(z,w)} \times \circ$} (v4);
	\draw[->]  (v3) edge node[below, font= \scriptsize]{$\circ$} (v5);
	\draw[->]  (v4) edge node[below, font= \scriptsize]{$\circ$} (v5);
	\end{tikzpicture}
\end{center}
for all objects $x,y,z,w$. 

Examples of simlicially enriched categories are numerous and occur naturally. Most prominently the category of simplicial sets is itself simplicially enriched. For two simplicial sets $X,Y$ the $n$-simplicies of the mapping space are given by 
$$\text{Map}(X,Y)_n = \text{\textbf{sSet}}(X \times \Delta^n, Y)$$
with the obvious face and degeneracy relations. The category of topological spaces can be seen to be simplicially enriched in a similar fashion. For two spaces $X,Y$ we can declare the $n$-simplicies of the mapping space to be continuous maps 
$$X \times |\Delta^n| \rightarrow Y$$
where $|\Delta^n|$ is the standard topological $n$-simplex. 

This way in both examples $0$-simplicies in $\text{Map}(X,Y)$ are ordinary maps of simplicial sets, $1$-simplicies are homotopies and higher simplicies are higher homotopies. 

We may (and will) use this interpretation of simplicial enrichment even for abstract \textbf{sSet}-categories $\mathcal{C}$. Let $x,y \in \mathcal{C}$ be two objects. A $0$-simplex in $f \in \mathcal{C}(x,y)_0$ may be depicted as a usual morphism $x \xrightarrow{f} y$. A $1$-simplex $\alpha \in \mathcal{C}(x,y)_1$ will have two vertices in the mapping space, say $d_0\alpha = f$ and $d_1\alpha = g$. We may depict $\alpha$ as  
\begin{center}
	\begin{tikzpicture}
	\node (v1) at (-4,3) {$x$};
	\node (v2) at (-2,3) {$y$};
	\draw[->, bend left = 50]  (v1) edge node [above,font = \scriptsize]{$f$} (v2);
	\draw[->, bend right = 50]  (v1) edge node [below,font = \scriptsize]{$g$} (v2);
	\node (v3) at (-3,3.5) {};
	\node (v4) at (-3,2.5) {};
	\draw[->]  (v3) edge[double] node[right,font = \scriptsize]{$\alpha$} (v4);
	\end{tikzpicture}
\end{center} 

The above picture may be thought of as being produced by suspending the $1$-simplex $f \xrightarrow{\alpha} g$ between the objects $x$ and $y$. It is helpful in visualising \textbf{sSet}-categories to think of the whole mapping space $\mathcal{C}(x,y)$ as being suspended between $x$ and $y$.

 While in an ordinary category we compose arrows in a \textbf{sSet}-category $\mathcal{C}$ we compose higher simplicies as well. Nonetheless we have an underlying category $\mathcal{C}_0$ associated to $\mathcal{C}$ with the same objects and morphism sets given as 
$$\mathcal{C}_0(x,y) = \mathcal{C}(x,y)_0$$

Before returning to homotopy colimits we define another important notion: tensoring and cotensoring in a \textbf{sSet}-category. A \textbf{sSet}-category is said to be tensored over simplicial sets if for any object $x \in \mathcal{C}$ and simplicial set $K$ there is an object $K \odot x \in \mathcal{C}$ such that for any other object $y$ we have a natural isomorphism 
$$\mathcal{C}_0(K \odot x, y) \cong \textbf{sSet}(K, \mathcal{C}(x,y))$$
Similarly $\mathcal{C}$ is cotensored if there is an object $\{K,y \}$ for all $y \in \mathcal{C}$ and $K \in \text{\textbf{sSet}}$ such that for all objects $x$ there is a natural isomorphism 
$$\textbf{sSet}(K, \mathcal{C}(x,y)) \cong \mathcal{C}_0(x, \{K,y \})$$

The category of simplicial sets is tensored and cotensored over itself via $$K \odot X = K \times X$$ and $$\{K,X\} = \text{Map}(K,X)$$
for all $K,X \in \text{\textbf{sSet}}$. The category of topological spaces is tensored and cotensored as well. If $X$ is a space then define $$K \odot X = |K| \times X$$ and $$\{K,X\} = \text{C}(|K|,X)$$
where $|K|$ denotes the geometric realization of $K$ and $\text{C}(|K|,X)$ denotes the set of continuous maps from $|K|$ to $X$ endowed with the compact-open topology. In this case we have to restrict ourselves to a convenient category of spaces containing all CW-complexes, like for example the category of compactly generated Hausdorff spaces. 

Tensoring and cotensoring bring a \textbf{sSet}-category $\mathcal{C}$ closer to intition. For example a homotopy $\Delta^1 \rightarrow \mathcal{C}(x,y)$ can be realized as a map 
$$\Delta^1 \odot x \rightarrow y$$
Also for a map $f: x \rightarrow y$ we can define its mapping cylinder to be given as a pushout 
\begin{center}
	\begin{tikzpicture}
	
	\node (v1) at (-3,3.5) {$x$};
	\node (v2) at (-1,3.5) {$y$};
	\node (v3) at (-3,1.5) {$\Delta^1 \odot x$};
	\node (v4) at (-1,1.5) {$M_f$};
	\draw[->]  (v1) edge node [above,font = \scriptsize]{$f$} (v2);
	\draw[->]  (v1) edge node [left,font = \scriptsize]{$d^0 \odot 1_x$} (v3);
	\draw[->, dashed]  (v3) edge (v4);
	\draw[->, dashed]  (v2) edge (v4);
	\end{tikzpicture}
\end{center}

We also have enough structure to talk about homotopy colimits in $\mathcal{C}$. Let $\mathcal{J}$ be an indexing category and $$F : \mathcal{J} \rightarrow \mathcal{C}$$
be a diagram. Then the homotopy colimit of $F$ is defined via the coend formula 
$$\text{hocolim}F = \int^{i \in \mathcal{J}} i/\mathcal{J} \odot F(i)$$
Here $i/J$ for an object $i\in \mathcal{J}$ is the usual category over $i$ with
\begin{itemize}
	\item[($\bullet$)] objects pairs $(j,f)$ where $j$ is an object in $\mathcal{J}$ and $f: i \rightarrow j$ is a morphism
	\item[($\rightarrow$)] morphisms between two objects $(j,f)$ and $(j^\prime, f^\prime)$ are commutative triangles
	\begin{center}
		\begin{tikzpicture}

		\node (v1) at (-2.5,2) {$i$};
		\node (v2) at (-3.5,3) {$j$};
		\node (v3) at (-1.5,3) {$j^\prime$};
		\draw[->]  (v1) edge node [left,font = \scriptsize]{$f$} (v2);
		\draw[->]  (v1) edge node [right,font = \scriptsize]{$f^\prime$} (v3);
		\draw [->] (v2) edge (v3);
		\end{tikzpicture}
	\end{center}
\end{itemize}
Coends are a special form of colimit (see for example \cite{mac2013categories} or the dedicated manuscript \cite{loregian2015co}) which we find mysterious, so we will elaborate the homotopy colimit construction in more elementary terms.  

We first consider the case where $\mathcal{J} = \Delta^n$ so that the functor $F$ will give us an $n$-simplex, say
$$\sigma : x_0 \xrightarrow{f_1} x_1 \xrightarrow{f_2} \dots \xrightarrow{f_n} x_n$$
in the category $\mathcal{C}_0$. Then we may think of the homotopy colimit of $\sigma$ as being given by the colimit of the following staircase diagram:
\begin{center}
	\begin{tikzpicture}
	
	\node (v2) at (-5,-3) {$x_0 \odot \Delta^n$};
	\node (v1) at (-5,-1.5) {$x_0 \odot \Delta^{n-1}$};
	\node (v3) at (-2,-1.5) {$x_1 \odot \Delta^{n-1}$};
	\node (v4) at (-2,0) {$x_1 \odot \Delta^{n-2}$};
	\node (v5) at (1,0) {$x_2 \odot \Delta^{n-2}$};
	\node (v6) at (1,1.5) {$\dots$};
	\node at (1.5,2) {$\dots$};
	\node (v7) at (1.5,2.5) {$\dots$};
	\node (v8) at (3.5,2.5) {$x_{n-1} \odot \Delta^1$};
	\node (v9) at (3.5,4) {$x_{n-1}$};
	\node (v10) at (5.5,4) {$x_n$};
	\draw[->]  (v1) edge node [left,font = \scriptsize]{$1_{x_0} \odot d^0$} (v2);
	\draw[->]  (v1) edge node [above,font = \scriptsize]{$f_1 \odot \Delta^{n-1}$}(v3);
	\draw[->]  (v4) edge node [left,font = \scriptsize]{$1_{x_1} \odot d^0$} (v3);
	\draw[->]  (v4) edge node [above,font = \scriptsize]{$f_2 \odot \Delta^{n-2}$} (v5);
	\draw[->]  (v6) edge (v5);
	\draw[->]  (v7) edge (v8);
	\draw[->]  (v9) edge node [left,font = \scriptsize]{$1_{x_{n-1}} \odot d^0$} (v8);
	\draw[->]  (v9) edge node [above,font = \scriptsize]{$f_n$} (v10);
	\end{tikzpicture}
\end{center}
We will denote this colimit by $M_\sigma$ and refer to it as the (higher) mapping cylinder of $\sigma$. 

Next we may wonder about the association $$\sigma \mapsto M_\sigma$$
of the mapping cylinder to an $n$-simplex in $\mathcal{C}_0$. As expected this extends to define a functor 
$$M_* : \text{\textbf{Gro}}(\mathcal{C}_0)^{op} \rightarrow \mathcal{C}_0$$
To construct this functor we consider a map $\theta: [m] \rightarrow [n]$ in $\Delta$ and then tries to see that there is an induced map $M_{\theta(\sigma)} \rightarrow M\sigma$ by constructing a map between the diagrams that form these mapping cylinders. This is not a difficult exercice and we leave it to the reader. We also invite the reader to contemplate the induced map geometrically in light of the picture we drew in the first section. 

Finally we may define (or conclude depending on where we start) $$\text{hocolim F} = \text{colim} \{ M_* \circ \text{\textbf{Gro}}(F)\}$$
where the functor on the right is the composite 
$$\text{\textbf{Gro}}(\mathcal{J})^{op} \xrightarrow{\text{\textbf{Gro}}(F)} \text{\textbf{Gro}}(\mathcal{C}_0)^{op} \xrightarrow{M_*} \mathcal{C}_0$$
All we have done is form the mapping cylinder corresponding to each simplex in $\mathcal{J}$ and then we glued these higher cylinders along their faces faces. Homotopy limits are defined by dualizing everything.

\begin{remark}
	The above are labelled "local homotopy colimits" in \cite{shulman2006homotopy}. The reason being they do not necessarily satisfy the "global" properties to define derived functors. In the setting of simplicial enrichment one defines homotopy colimits without ever mentioning weak equivalences at all. 
\end{remark}

\subsection{Double colimits and homotopy colimits}
\label{thm}

Now we reveal the promised connection between double colimits in higher equipments and homotopy colimits. The fist were defined for a simplicial category and the latter can be defined in various contexts (\textbf{sSet}-categories, model categories etc, see \cite{dugger2008primer}). So we start by building a bridge between the two worlds. 

Let $\mathbb{E}$ be a simplicial category.There is a \textbf{sSet}-category $\mathbb{E}_v$ with:
\begin{itemize}
	\item[($\bullet$)] objects those of $\mathbb{E}_0$
	\item[($\triangle$)] $n$-simplicies of the mapping space for two objects $x,y \in \mathbb{E}_0$ given by 
	$$\mathbb{E}_v(x,y)_n = \mathbb{E}_n(s_0^{(n)}x, s_0^{(n)}y)$$
	\item[($\rightsquigarrow$)] for a map $\theta: [m] \rightarrow [n]$ in $\Delta$ the induced map given as 
	$$\mathbb{E}(\theta) : \mathbb{E}_n(s_0^{(n)}x, s_0^{(n)}y) \rightarrow \mathbb{E}_m(s_0^{(m)}x, s_0^{(m)}y)$$
	\item[($\circ$)] composition given in the obvious way by composition in the categories $\mathbb{E}_n$
\end{itemize}
The above construction is derived from the analogy between the following figures
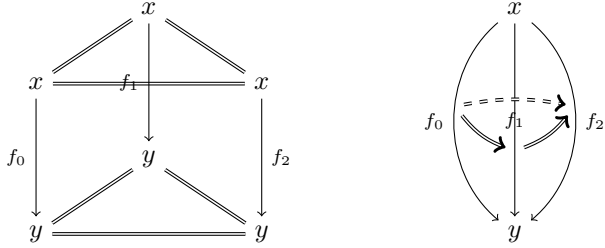
\begin{center}
	\begin{tikzpicture}

	\node (v1) at (-4,2.5) {$x$};
	\node (v2) at (-2.5,3.5) {$x$};
	\node (v3) at (-1,2.5) {$x$};
	\node (v4) at (-4,0.5) {$y$};
	\node (v5) at (-2.5,1.5) {$y$};
	\node (v6) at (-1,0.5) {$y$};
	\draw[-]  (v1) edge[double] (v2);
	\draw[-]  (v2) edge[double] (v3);
	\draw[-]  (v1) edge[double] (v3);
	\draw[->]  (v1) edge node[left, font= \scriptsize]{$f_0$} (v4);
	\draw[->]  (v2) edge node[left, font= \scriptsize]{$f_1$} (v5);
	\draw[->]  (v3) edge node[right, font= \scriptsize]{$f_2$} (v6);
	\draw[-]  (v4) edge[double] (v5);
	\draw[-]  (v5) edge[double] (v6);
	\draw[-]  (v4) edge[double] (v6);
	\end{tikzpicture} \ \ \ \ \ \ \ \ \ \ \ \
	\begin{tikzpicture}
	
	\node (v1) at (-2,3.5) {$x$};
	\node (v2) at (-2,0.5) {$y$};
	\draw[->, bend right = 50]  (v1) edge node[left, font= \scriptsize]{$f_0$} (v2);
	\draw[->, bend left  = 50]  (v1) edge node[right, font= \scriptsize]{$f_2$} (v2);
	\draw[->]  (v1) edge node[ font= \scriptsize]{$f_1$} (v2);
	
	\node (v3) at (-2.8,2.2) {};
	\node (v4) at (-2,1.6) {};
	\draw[->, bend right = 15]  (v3) edge[double] (v4);
	
	\node (v5) at (-1.2,2.2) {};
	\draw[->, bend right = 15]  (v4) edge[double] (v5);
	
	\draw[->, bend left = 10, dashed]  (v3) edge[double] (v5);
	\end{tikzpicture}
\end{center}
On the left we have depicted a map in $\mathcal{\mathbb{E}}_2$ between $s_0^{(2)}x$ and $s_0^{(2)}y$ and on the right we have depicted a $2$-simplex in the mapping space between $x$ and $y$ in the \textbf{sSet}-category $\mathbb{E}_v$. 

This construction is also analogous to the construction of the vertical $2$-category associated to a double category. So we have discovered something important: 
\begin{quote}
	The vertical direction of a simplicial category (when looked at as a two-fold structure) consists of a simplicially enriched category!
\end{quote}
We dare rephrase the above and say that double categories are to $2$-categories what simplicial categories are to simplicially enriched categories, and write
$$\frac{\text{double categories}}{2\text{-categories}} = \frac{\text{simplicial categories}}{\text{\textbf{sSet}-categories}}$$

The connection between simplicial categories and simplicially enriched categories is not unknown, although our construction and interpretation seems to be novel (to the best of our knowledge). If $\mathcal{C}$ is a \textbf{sSet}-category, besides $\mathcal{C}_0$ we may construct categories $\mathcal{C}_n$ with
\begin{itemize}
	\item[($\bullet$)] objects those of $\mathcal{C}$
	\item[($\rightarrow$)] morphism set between two objects $x$ and $y$ being the $n$-simplicies of their mapping space 
	$$\mathcal{C}_n(x,y) = \mathcal{C}(x,y)_n$$
	\item[($\circ$)] composition as prescribed by the composition map in $\mathcal{C}$
\end{itemize}
The reader can utilize the above pictures again to grasp the categories $\mathcal{C}_n$, and also the fact that the mapping $[n] \mapsto \mathcal{C}_n$ gives us a simplicial category, which we denote $|\mathcal{C}|$, with faces and degeneracies defined in the obvious way (this is mentioned in \cite{riehl2014categorical} when discussing the coherent nerve construction). 

Both constructions extend to give us functors and as a matter of fact we have an adjunction 
$$|\cdot| : \textbf{sSet-Cat} \rightleftarrows \textbf{sCat} : (\cdot)_v$$
with $|\cdot|$ the left adjoint and $(\cdot_v)$ the right adjoint. Moreover we have $|\mathcal{C}|_v = \mathcal{C}$ for all $\mathcal{C}$ and $|\cdot|$ fully faithful. Next we state a theorem whose proof is easy but meaning essential in light of this work.

\begin{theorem}
	Let $\mathbb{E}$ be a simplicial category. If $\mathbb{E}$ has double colimits then $\mathbb{E}_v$ is cotensored. 
\end{theorem}
\begin{proof}
	Let $x$ be an object in $\mathbb{E}$ and $K$ be a simplicial set. Consider the constant diagram 
	$$K_x : K \rightarrow \mathbb{E}$$
	which sends everything to $x$. Define 
	$$K \odot x = \text{dcolim}{K_x}$$
	In particular 
	$$\Delta^n \odot x = \bot_{s^{(n)}x}$$
	Let $y$ be another object. Then for all $[n] \in \Delta$ we have 
	\begin{align*}
		\textbf{sSet}(\Delta^n, \mathbb{E}_v(x,y)) &= \mathbb{E}_v(x,y)_n \\ &= E_n(s_0^{(n)}x, s_0^{(n)}y) \\
		&\cong \mathbb{E}_0(\text{dcolim}{\Delta^n_x}, y) \\
		&= \mathbb{E}_0(\Delta^n \odot x, y) \\
		&= \mathbb{E}_{v0}(\Delta^n \odot x, y)
	\end{align*}
	Since every simplicial set $K$ is a colimit of its simplicies we conclude 
	$$\textbf{sSet}(K, \mathbb{E}_v(x,y)) \cong \mathbb{E}_{v0}(K\odot x, y)$$
	for all $x$ and $y$. 
	
	Dually we have 
	$$\{K,x \} = \text{dlim}{K_x}$$
	and in particular
	$$\{ \Delta^n, x \} = \top_{s^{(n)}x}$$
\end{proof}

Now we state and prove the main theorem of this work:

\hodcolim
\begin{proof}
	In light of the construction theorem for double colimits (Theorem \ref{colimconstruction})  it is enough to consider $\mathcal{J} = \Delta^n$. In this case, for an $n$-simplex $\sigma$ in $\mathbb{E}_0$, the tower representation (Proposition \ref{tower}) of $\sigma^*$ together with Proposition  \ref{cotabextension} imply that morphism in $E_n$ $$\sigma^* \rightarrow s^ny$$ for some $y \in \mathbb{E}_0$ corresponds precisely to morphisms from the step diagram corresponding to $M_\sigma$ to $y$.  
\end{proof}

We would like to conclude by duality that the theorem holds for homotopy limits as well by our proof relies on the fact that colimits interact well with cotabulators of universal extensions. We address this issue in the next section.

\subsection{Duality}
\label{dual}

The equipment property as postulated previously states "every morphism from boundaries extends universally". We can also consider the dual which states "every morphism to a boundary lifts universally". Perhaps we can refer to the first as the "left equipment property" and the latter as "right equipment property". 

More precisely, let $\mathbb{E}$ be a simplicial category, $y \in \mathbb{E}_n$, $x^\bullet : \partial \Delta^n \rightarrow \mathbb{E}$ and $f^\bullet: x^\bullet \rightarrow \partial y$ be a morphism. Then we 
say $\mathbb{E}$ is a right equipment if given the above there if $x \in \mathbb{E}_n$ and $f: x\rightarrow y$ satisfying:
\begin{itemize}
	\item[i)] $\partial f = f^\bullet$ and as a consequence $\partial x = x^\bullet$
	\item[ii)] Any morphism $g : z \rightarrow y$ such that $\partial g$ factors through $f^\bullet$ factors uniquely through $f$.
\end{itemize}
Denote such a universal lift $(f^\bullet)y$

In a right equipment we can perform all the constructions we did on the left side and obtain a theorem about homotopy limits. We will present just a sketch. First we see that $\textbf{sSet}^\sharp$ is a right equipment. Given $X^\bullet$, $Y$ and $f^\bullet$ as above we can obtain the universal lift as a pullback
\begin{center}
	\begin{tikzpicture}

	\node (v4) at (-3.5,3) {$(f^\bullet)Y$};
	\node (v1) at (-1.5,3) {$\partial Y$};
	\node (v3) at (-3.5,1) {$X^\bullet$};
	\node (v2) at (-1.5,1) {$Y$};
	\draw[left hook->]  (v1) edge node[right]{$\partial$} (v2);
	\draw[->]  (v3) edge node[above]{$f^\bullet$} (v2);
	\draw[->, dashed]  (v4) edge (v3);
	\draw[->, dashed]  (v4) edge (v1);
	\end{tikzpicture}
\end{center}

Let $\mathbb{E}$ be a right equipment. Then we may define the (right) companion ${}^*\sigma$ of an $n$-simplex 
$$\sigma : x_0 \xrightarrow{f_1} x_1 \xrightarrow{f_2} \dots \xrightarrow{f_n} x_n$$
in $\mathbb{E}_0$ together with a structure map $\zeta_\sigma : {}^*\sigma \rightarrow s^nx_n$ recursively as 
$${}^*\sigma= \begin{cases}
\sigma &\text{if} \ \  n=0 \\
(\zeta_{d_n\sigma}, \dots, \zeta_{d_0\sigma})s^nx_n &\text{if} \ \ n>0
\end{cases}$$
We define $\zeta_\sigma = 1_\sigma $ for $\sigma \in \mathbb{E}_0$ and we define $\zeta_\sigma$ to be the induced map $(\zeta_{d_n\sigma}, \dots, \zeta_{d_0\sigma})s^nx_n \rightarrow s^nx_n$. 

This construction preserves faces but not degeneracies. As previously the axioms guarantee a comparison map 
$$s_i{}^*\sigma \rightarrow {}^*(s_i\sigma)$$
so that we obtain a lax morphism
$${}^*(\cdot) : \mathbb{E}_0 \rightarrow \mathbb{E}$$

We may define the limit of a lax horizontal diagram in $\mathbb{E}$ in the obvious way. Then $\mathbb{E}_v$ is tensored if $\mathbb{E}$ has double limits. The tensoring is given by 
$$\{\Delta^n, x\} = \top_{s^{(n)}x}$$
and in general $\{K, x\} = \text{dlim}K_x$ for $K\in \textbf{sSet}$. 

Let $\mathcal{J}$ be an indexing category and $F: \mathcal{J} \rightarrow \mathbb{E}_0$ be a functor. Then we have
$$\text{dlim} {}^*F \cong \text{holim} F$$
where ${}^*F$ is the composite
$$\mathcal{J} \xrightarrow{F} E_0 \xrightarrow{{}^*(\cdot)} \mathbb{E}$$
and the homotopy colimit on the left side is interpreted in the simplicially enriched category $\mathbb{E}_v$. 

\begin{note}
The example of simplicial sets seems to indicate that we consider the case in which both the left and right equipment property are satisfied and left companions agree with right companions. We leave these considerations for future work.
\end{note}

\pagebreak
\bibliographystyle{alpha}
\bibliography{books}

\newcommand{\etalchar}[1]{$^{#1}$}
\begin{thebibliography}{Kou15}

\bibitem[CS10]{cruttwell2010unified}
Geoffrey~SH Cruttwell and Michael~A Shulman.
\newblock A unified framework for generalized multicategories.
\newblock {\em Theory and Applications of Categories}, 24(21):580--655, 2010.

\bibitem[Dug08]{dugger2008primer}
Daniel Dugger.
\newblock A primer on homotopy colimits.
\newblock {\em preprint}, 2008.

\bibitem[Dus02]{duskin2002simplicial}
John~W Duskin.
\newblock Simplicial matrices and the nerves of weak n-categories. i. nerves of
  bicategories.
\newblock {\em Theory and applications of categories}, 9(10):198--308, 2002.

\bibitem[G{\etalchar{+}}17]{grandis2017span}
Marco Grandis et~al.
\newblock Span and cospan representations of weak double categories.
\newblock {\em Categories and General Algebraic Structures with Applications},
  6(Speical Issue on the Occasion of Banaschewski's 90th Birthday (I)):85--105,
  2017.

\bibitem[GHN]{gepner1501lax}
David Gepner, Rune Haugseng, and Thomas Nikolaus.
\newblock Lax colimits and free fibrations in infinity-categories, 2015.
\newblock {\em arXiv preprint arXiv:1501.02161}.

\bibitem[GJ09]{goerss2009simplicial}
Paul~G Goerss and John~F Jardine.
\newblock {\em Simplicial homotopy theory}.
\newblock Springer Science \& Business Media, 2009.

\bibitem[GP99]{grandis1999limits}
Marco Grandis and Robert Par{\'e}.
\newblock Limits in double categories.
\newblock {\em Cahiers de Topologie et G{\'e}om{\'e}trie Diff{\'e}rentielle
  Cat{\'e}goriques}, 40(3):162--220, 1999.

\bibitem[GP04]{grandis2004adjoint}
Marco Grandis and Robert Par{\'e}.
\newblock Adjoint for double categories.
\newblock {\em Cahiers de topologie et g{\'e}om{\'e}trie diff{\'e}rentielle
  cat{\'e}goriques}, 45(3):193--240, 2004.

\bibitem[GP08]{grandis2008kan}
Marco Grandis and Robert Par{\'e}.
\newblock Kan extensions in double categories (on weak double categories, part
  iii).
\newblock {\em Theory and Applications of Categories}, 20(8):152--185, 2008.

\bibitem[Joy08]{joyal2008theory}
Andr{\'e} Joyal.
\newblock The theory of quasi-categories and its applications.
\newblock 2008.

\bibitem[Kou15]{koudenburg2015double}
Seerp~Roald Koudenburg.
\newblock A double-dimensional approach to formal category theory.
\newblock {\em arXiv preprint arXiv:1511.04070}, 2015.

\bibitem[Lei98]{leinster1998basic}
Tom Leinster.
\newblock Basic bicategories.
\newblock {\em arXiv preprint math/9810017}, 1998.

\bibitem[Lei04]{leinster2004higher}
Tom Leinster.
\newblock {\em Higher operads, higher categories}, volume 298.
\newblock Cambridge University Press, 2004.

\bibitem[Lor15]{loregian2015co}
Fosco Loregian.
\newblock This is the (co) end, my only (co) friend.
\newblock {\em arXiv preprint arXiv:1501.02503}, 2015.

\bibitem[Lur09]{lurie2009higher}
Jacob Lurie.
\newblock {\em Higher Topos Theory (AM-170)}, volume 189.
\newblock Princeton University Press, 2009.

\bibitem[ML13]{mac2013categories}
Saunders Mac~Lane.
\newblock {\em Categories for the working mathematician}, volume~5.
\newblock Springer Science \& Business Media, 2013.

\bibitem[Par11]{pare2011yoneda}
Robert Par{\'e}.
\newblock Yoneda theory for double categories.
\newblock {\em Theory and Applications of Categories}, 25(17):436--489, 2011.

\bibitem[R{\etalchar{+}}09]{riehl2009weighted}
Emily Riehl et~al.
\newblock Weighted limits and colimits.
\newblock {\em available at math. harvard. edu/~ eriehl}, 2009.

\bibitem[Rie11]{riehl2011leisurely}
Emily Riehl.
\newblock A leisurely introduction to simplicial sets.
\newblock {\em Unpublished expository article available online at http://www.
  math. harvard. edu/\~{} eriehl}, 2011.

\bibitem[Rie14]{riehl2014categorical}
Emily Riehl.
\newblock {\em Categorical homotopy theory}, volume~24.
\newblock Cambridge University Press, 2014.

\bibitem[Rie17]{riehl2017category}
Emily Riehl.
\newblock {\em Category theory in context}.
\newblock Courier Dover Publications, 2017.

\bibitem[Shu06]{shulman2006homotopy}
Michael Shulman.
\newblock Homotopy limits and colimits and enriched homotopy theory.
\newblock {\em arXiv preprint math/0610194}, 2006.

\bibitem[Shu08]{shulman2008framed}
Michael Shulman.
\newblock Framed bicategories and monoidal fibrations.
\newblock {\em Theory and applications of categories}, 20(18):650--738, 2008.

\bibitem[Shu09]{shulman2009equip}
Michael Shulman.
\newblock Equipments.
\newblock \url{https://golem.ph.utexas.edu/category/2009/11/equipments.html},
  2009.

\end{thebibliography}

\end{document}